\documentclass[11pt,onecolumn]{IEEEtran}

\usepackage[bookmarks=false]{hyperref}

\usepackage[cmex10]{amsmath}
\usepackage{amsxtra,amssymb,amsthm,amsfonts}
\usepackage{mathtools}
\usepackage{mathrsfs}
\usepackage{mathdots}

\usepackage[ruled,boxed,linesnumbered]{algorithm2e}

\usepackage[margin=1in]{geometry}
\usepackage{setspace}

\usepackage[usenames]{color}
\usepackage{bbm}
\usepackage{bm}
\usepackage{umoline}

\usepackage{tabstackengine}
\stackMath

\usepackage{enumerate}
\usepackage{array}
\usepackage{multirow}

\usepackage{graphicx}
\usepackage{epstopdf}
\usepackage[subrefformat=parens,labelformat=parens]{subfig}

\usepackage{cleveref}
\crefname{equation}{}{}
\usepackage{cite}

\newtheorem{lemma}{Lemma}[section]
\newtheorem{prop}[lemma]{Proposition}
\newtheorem{theorem}[lemma]{Theorem}
\newtheorem{cor}[lemma]{Corollary}
\newtheorem{rem}[lemma]{Remark}

\newcommand{\cz}{{\mathbb C}}

\newcommand{\qd}{\end{proof}\vspace{0.5ex}}

\newcommand\tnorm[1]{\vert\xspace\vert\xspace\vert\mskip2mu
#1\mskip2mu \vert\xspace\vert\xspace\vert}

\newcommand\Bigtnorm[1]{\Big\vert\xspace\Big\vert\xspace\Big\vert\mskip2mu
#1\mskip2mu \Big\vert\xspace\Big\vert\xspace\Big\vert}

\newcommand{\norm}[1]{\Vert#1\Vert}
\newcommand{\Bignorm}[1]{\Big\Vert#1\Big\Vert}

\newcommand{\transpose}{\top}

\newcommand{\conv}{\circledast}

\newcommand{\opconj}{\bm{J}}

\renewcommand{\baselinestretch}{1.3}

\newcommand{\E}{\operatorname{E}}

\renewcommand{\d}[1]{d#1}

\newcommand{\vct}[1]{\boldsymbol{#1}}
\newcommand{\mtx}[1]{\boldsymbol{#1}}

\newcommand{\set}[1]{\mathcal{#1}}

\DeclareMathOperator*{\minimize}{\text{minimize}}

\newcommand{\va}{\vct{a}}
\newcommand{\vb}{\vct{b}}
\newcommand{\vc}{\vct{c}}
\newcommand{\vd}{\vct{d}}
\newcommand{\ve}{\vct{e}}
\newcommand{\vf}{\vct{f}}
\newcommand{\vg}{\vct{g}}
\newcommand{\vh}{\vct{h}}

\newcommand{\vq}{\vct{q}}

\newcommand{\vu}{\vct{u}}
\newcommand{\vv}{\vct{v}}
\newcommand{\vw}{\vct{w}}
\newcommand{\vx}{\vct{x}}
\newcommand{\vy}{\vct{y}}
\newcommand{\vz}{\vct{z}}

\newcommand{\vxi}{\vct{\xi}}

\newcommand{\vphi}{\vct{\phi}}

\newcommand{\vzero}{\vct{0}}

\newcommand{\mA}{\mtx{A}}
\newcommand{\mB}{\mtx{B}}
\newcommand{\mC}{\mtx{C}}
\newcommand{\mD}{\mtx{D}}
\newcommand{\mE}{\mtx{E}}
\newcommand{\mF}{\mtx{F}}

\newcommand{\mH}{\mtx{H}}

\newcommand{\mM}{\mtx{M}}
\newcommand{\mN}{\mtx{N}}

\newcommand{\mP}{\mtx{P}}
\newcommand{\mQ}{\mtx{Q}}

\newcommand{\mS}{\mtx{S}}
\newcommand{\mT}{\mtx{T}}
\newcommand{\mU}{\mtx{U}}
\newcommand{\mV}{\mtx{V}}

\newcommand{\mX}{\mtx{X}}
\newcommand{\mY}{\mtx{Y}}
\newcommand{\mZ}{\mtx{Z}}
\newcommand{\mGamma}{\mtx{\Gamma}}

\newcommand{\mUpsilon}{\mtx{\Upsilon}}
\newcommand{\mPhi}{\mtx{\Phi}}
\newcommand{\mPsi}{\mtx{\Psi}}

\newcommand{\mId}{{\bf I}}

\newcommand{\setR}{\set{R}}

\title{Fast and guaranteed blind multichannel deconvolution under a bilinear system model}
\author{Kiryung Lee, Ning Tian, and Justin Romberg
\thanks{This work was supported in part by NSF grants CCF 14-22540, IIS 14-47879, and by ONR grant N00014-13-1-0632.}
\thanks{An earlier version of this work was presented in part at the 2016 IEEE Information Theory Workshop \cite{lee2016fast} and the 22nd International Conference on Digital Signal Processing \cite{lee2007dsp}.}\thanks{The authors are with the School of Electrical and Computer Engineering,
Georgia Institute of Technology, Atlanta, GA 30332, USA (e-mail:
\{kiryung,jrom\}@ece.gatech.edu, ningtian@gatech.edu).}
}

\begin{document}

\maketitle


\begin{abstract}

We consider the multichannel blind deconvolution problem where we observe the output of multiple channels that are all excited with the same unknown input.  From these observations, we wish to estimate the impulse responses of each of the channels.  We show that this problem is well-posed if the channels follow a bilinear model where the ensemble of channel responses is modeled as lying in a low-dimensional subspace but with each channel modulated by an independent gain.  Under this model, we show how the channel estimates can be found by minimizing a quadratic functional over a non-convex set.

We analyze two methods for solving this non-convex program, and provide performance guarantees for each.  The first is a method of alternating eigenvectors that breaks the program down into a series of eigenvalue problems.  The second is a truncated power iteration, which can roughly be interpreted as a method for finding the largest eigenvector of a symmetric matrix with the additional constraint that it adheres to our bilinear model.  As with most non-convex optimization algorithms, the performance of both of these algorithms is highly dependent on having a good starting point.  We show how such a starting point can be constructed from the channel measurements.

Our performance guarantees are non-asymptotic, and provide a sufficient condition on the number of samples observed per channel in order to guarantee channel estimates of a certain accuracy.  Our analysis uses a model with a ``generic'' subspace that is drawn at random, and we show the performance bounds hold with high probability.  Mathematically, the key estimates are derived by quantifying how well the eigenvectors of certain random matrices approximate the eigenvectors of their mean.

We also present a series of numerical results demonstrating that the empirical performance is consistent with the presented theory.

\end{abstract}


\section{Introduction}
\label{sec:intro}

Blind deconvolution, where we estimate two unknown signals from an observation of their convolution, is a classical problem in signal processing.  It is ubiquitous, appearing in applications including channel estimation in communications, image deblurring and restoration, seismic data analysis, speech dereverberation, medical imaging, and convolutive dictionary learning.  While algorithms based on heuristics for particular applications have existed for decades, it is not until recently that a rich mathematical theory has developed around this problem.
The fundamental identifiability of solutions to this problem has been studied from an information theoretic perspective \cite{choudhary2014sparse,li2016identifiability,riegler2015information,li2016optimal,kech2017optimal,li2017identifiability,li2017identifiability_minimal}. Practical algorithms with provable performance guarantees that make the problem well-posed by imposing structural constraints on the signals have arisen based on ideas from compressed sensing and low-rank matrix recovery.  These include methods based on convex programming \cite{ahmed2014blind,tang14co,bahmani15li}, alternating minimization \cite{lee2017blind}, and gradient descent \cite{li2016rapid}. More recent works studied the more challenging problem of blind deconvolution with off-the-grid sparsity models \cite{chi2016guaranteed,yang2016super}.

In this paper, we consider the multichannel blind deconvolution problem: we observe a single unknown signal (the ``source'') convolved with a number of different ``channels''.  The fact that the input is shared makes this problem better-posed than in the single channel case.  Mathematical theory for the multichannel problem under various constraints has existed since the 1990s (see \cite{abed1997blind,tong1998multichannel} for surveys).  One particular strand of this research detailed in \cite{xu1995least,moulines1995subspace,gurelli1995evam} gives concrete results under the very loose assumption that the channel responses are time-limited.  These works show how with this model in place, the channel responses can be estimated by forming a cross-correlation matrix from the channel outputs and then computing its smallest eigenvector.  This estimate is consistent in that it is guaranteed to converge to the true channel responses as the number of observations gets infinitely large.  However, no performance guarantees were given for a finite number of samples, and the method tends to be unstable for moderate sample values in even modest noise.
Recent work \cite{lee2017spectral} has shown that this spectral method can indeed be stabilized by introducing a more restrictive linear (subspace) model on the channel responses.  

Our main contributions in this paper are methods for estimating the channel responses when the ensemble has a certain kind of bilinear structure.  In particular, we model the ensemble of channel responses as lying in a low-dimensional subspace, but with each channel modulated by an independent constant; we will discuss in the next section an application in which this model is relevant.  Our estimation framework again centers on constructing a cross-correlation matrix and minimizing a quadratic functional involving this matrix over the unit sphere, but with the additional constraint that the solution can be written as the Kronecker product of two shorter vectors.  This optimization program, which might be interpreted as a kind of structured eigenvalue problem, is inherently non-convex.  We propose two iterative methods for solving it, each with very simple, computationally efficient iterations.  The first is a method of alternating eigenvectors, where we alternate between fixing a subset of the unknowns and estimating the other by solving a standard eigenvalue problem.  The second method is a truncated power iteration, where we repeatedly apply the cross-correlation matrix to an initial point, but project the result after each application to enforce the structural constraints.  We derive performance guarantees for both of these algorithms when the low-dimensional subspace is generic (i.e.\ generated at random).

\subsection*{Related work}

Closely related to the problem of multichannel blind deconvolution is the problem of blind calibration.  Here we observe the product of an unknown weighting vector applied to a series of other unknown vectors.  Non-convex optimization algorithms for blind calibration  have been studied and analyzed in \cite{ling2015self,cambareri2016through}.

Multichannel blind deconvolution can also be approached by linearizing the problem in the Fourier domain.  This has been proposed for various applications, including the calibration of a sensor network \cite{balzano2007blind}, computational relighting in inverse rendering \cite{nguyen2013subspace}, and auto-focus in synthetic aperture radar \cite{morrison2009mca}.  Under a generic condition that the unknown impulse responses belong to random subspaces, necessary and sufficient conditions for the unique identification of the solution have been put forth in \cite{li2016optimal}, and a rigorous analysis of a least-squares method has been studied \cite{ling2016self}.

More recently, performance guarantees for spectral methods for for both subspace and sparsity models have been developed in \cite{li2017sampta}.  As in this paper, these methods are estimating the channel by solving a structured eigenvalue problem.  The structural model, however, is very different than the one considered here.
%
%

Algorithms for solving non-convex quadratic and bilinear problems have recently been introduced for solving problems closely related to blind deconvolution.
A non-convex optimization over matrix manifolds provides a guaranteed solution for matrix completion \cite{keshavan2010matrix}.
Alternating minimization is another non-convex optimization algorithm for matrix completion that provides a provable performance guarantee \cite{keshavan2012efficient,jain2013low,hardt2014understanding}.  Yet for another example, a suite of gradient-based algorithms with a specially designed regularizer within the conventional Euclidean geometry have been studied recently \cite{sun2016guaranteed}.  Phase retrieval is cast as a non-convex optimization due to the nonlinearity in generating the observation.  Wirtinger flow algorithms \cite{candes2015phase,chen2015solving,wang2016solving,cai2016optimal} and alternating minimization \cite{netrapalli2013phase,waldspurger2016phase} are non-convex optimization algorithms for the phase retrieval problem.  Dictionary learning is another bilinear problem arising in numerous applications. Convergence of a Riemannian trust-region method for dictionary learning has been studied with a thorough geometric analysis \cite{sun2017completeI,sun2017completeII}.  On the other hand, although it provides the convergence analysis of different nature, convergence of an alternating minimization for blind Ptychographic diffraction imaging to a local minima regardless of the initial point has been shown under a mild condition \cite{hesse2015proximal}.


%

\subsection*{Organization}
The rest of this paper is organized as follows.  The multichannel blind deconvolution problem is formulated under a bilinear channel model in Section~\ref{sec:setup}.  After we review relevant previous methods for multichannel blind deconvolution in Section~\ref{sec:spectral}, we present two iterative algorithms for multichannel blind deconvolution under the bilinear channel model in Section~\ref{sec:nonconvex}, which are obtained by modifying the classical cross-convolution method.  Our main results on non-asymptotic stable recovery are presented in Section~\ref{sec:main_results} with an outline of the proofs. Detailed analysis of the spectral initialization and the two iterative algorithms are derived in Sections~\ref{sec:proof:prop:init}, \ref{sec:proof:prop:alteig}, and \ref{sec:conv_tpm}. We demonstrate numerical results that support our theory in Section~\ref{sec:num_res}, and summarize our conclusions in Section~\ref{sec:conclusion}.

\section{Problem Statement}
\label{sec:setup}


In the classic multichannel blind deconvolution problem, we observe an unknown signal $\vx \in \mathbb{C}^L$ that has been convolved with $M$ different unknown channel responses $\vh_1,\ldots,\vh_m\in\mathbb{C}^L$:
\begin{equation}
	\label{eq:mcmdl}
	\vy_m = \vh_m \circledast \vx + \vw_m, \quad m = 1,\dots,M,
\end{equation}
where $\circledast$ denotes circular convolution\footnote{We are using circular convolution in our model problem for the ease of analysis. } modulo $L$ and $\vw_m \in \mathbb{C}^L$ denotes additive noise.
Given the outputs $\{\vy_m\}_{m=1}^M$, and working without knowledge of the common input $\vx$, we want to recover the unknown channel impulse responses $\{\vh_m\}_{m=1}^M$.

We will show how we can solve this problem when the channels are time-limited, and obey a {\em bilinear model}.  By time-limited, we mean that only the first $K$ entries in the $\vh_m$ can be non-zero; we can write
\begin{equation}
	\label{eq:def_FIR}
	\vh_m = \mS^\transpose \underline{\vh}_m,
	\quad\text{where}\quad
	\mS := \begin{bmatrix} \mId_K &\vzero_{K,L-K}\end{bmatrix}.
\end{equation}
In addition, the $\vh_m$ are {\em jointly} modeled as lying in a $D$-dimensional subspace of $\mathbb{C}^K$, but are modulated by unknown channel gains $a_1,\ldots,a_M > 0$.  This means that
\begin{equation}
	\label{eq:bilinear_cir_mdl}
	\underline{\vh}_m = a_m \mPhi_m \vb, \quad \forall m = 1,\dots,M,
\end{equation}
where the $\mPhi_m$ are complex $K\times D$ matrices, whose columns span are the parts of the basis vectors corresponding to channel $m$, and $\vb\in\mathbb{C}^D$ is the common set of basis coefficients.  Stacking up the channel responses into a single vector $\underline{\vh} \in \mathbb{C}^{MK}$ and the gains into $\va\in\mathbb{C}^M$, an equivalent way to write \eqref{eq:bilinear_cir_mdl} is
\begin{equation}
	\label{eq:bilinear_cir_mdl2}
	\underline{\vh} = \mPhi (\va \otimes \vb),
	\quad\text{where}\quad
	\mPhi :=
	\setstackgap{L}{1.1\baselineskip}
	\fixTABwidth{T}
	\bracketMatrixstack{
	\mPhi_1 & \vzero & \dots & \vzero \\
	\vzero & \mPhi_2 & \dots & \vzero \\
	\vdots & \vdots & \ddots & \vdots \\
	\vzero & \vzero & \dots & \mPhi_M}
	~~\text{and}\quad
	\va \otimes \vb =
	\begin{bmatrix} a_1\vb \\ a_2\vb \\ \vdots \\ a_M\vb \end{bmatrix}.
\end{equation}
This alternative expression can be interpreted as a linear subspace model with respect to the basis $\mPhi \in \mathbb{C}^{MK \times MD}$ with a separability (rank-1) prior on the coefficient vector.

\begin{figure}
  \centering
  \subfloat[]{\includegraphics[height=1.75in]{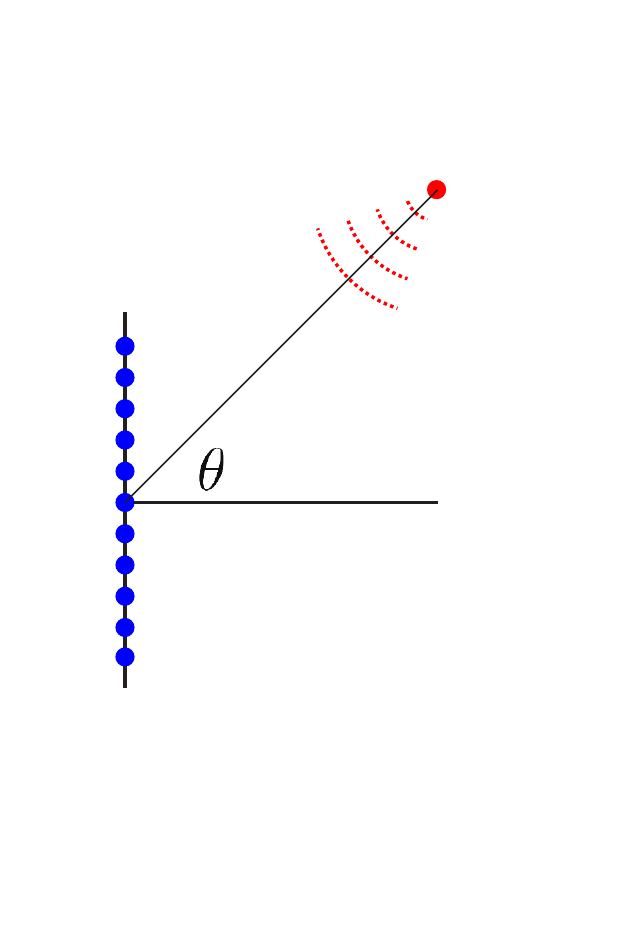}\label{fig:array}}
  \subfloat[]{\includegraphics[height=1.75in]{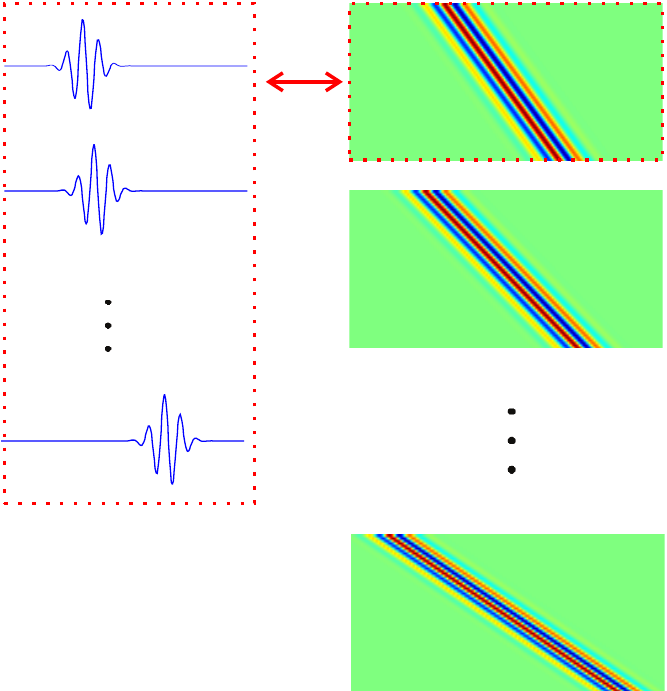}\label{fig:CIR_blocks}}
  \hfil
  \subfloat[]{\includegraphics[height=1.75in]{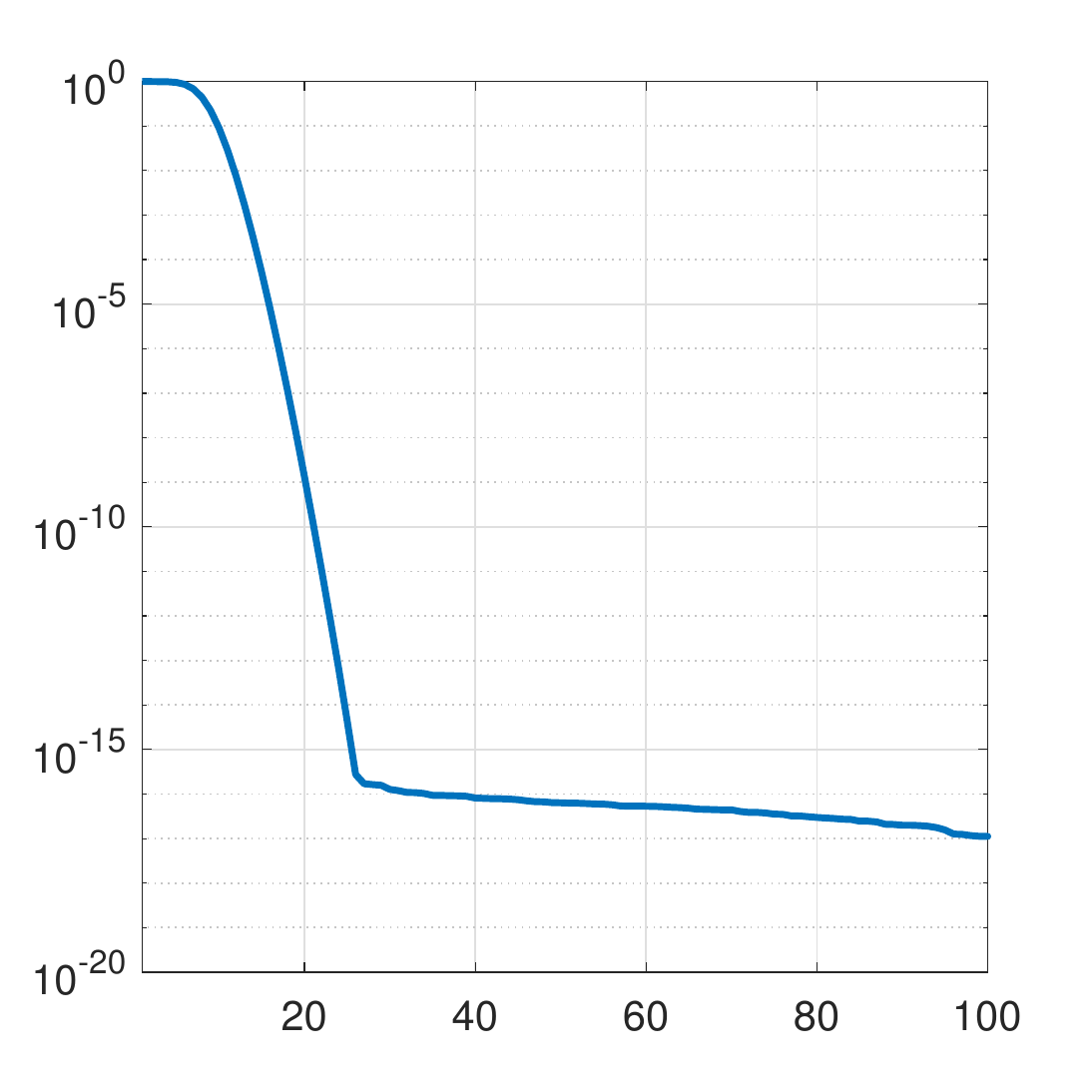}\label{fig:PCA}}
  \hfil
  \subfloat[]{\includegraphics[height=1.75in]{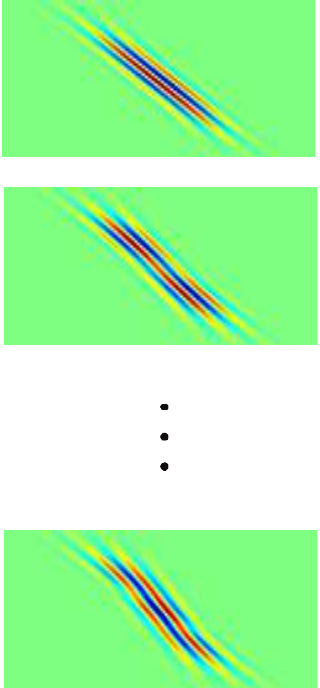}\label{fig:eigen_blocks}}
  \caption{Illustration of construction of a joint linear subspace model from a parametric model.
  \protect\subref{fig:array} A uniform array of $M$ sensors.
  \protect\subref{fig:CIR_blocks} Examples of $\{\vg_{\vec{r}}\}$ rearranged as $M$-by-$K$ matrices.
  \protect\subref{fig:PCA} Sorted eigenvalues of $\mH_{\mathcal{R}}$ in a logarithmic scale.
  \protect\subref{fig:eigen_blocks} First few dominant eigenvectors of $\mH_{\mathcal{R}}$ rearranged as $M$-by-$K$ matrices.
  }
  \label{fig:example_array_processing}%
\end{figure}

For an example of how a model like this might arise, we consider the following stylized problem for array processing illustrated in Figure~\ref{fig:example_array_processing}.  Figure~\protect\subref*{fig:array} shows a linear array.  Suppose we know that if a source is at location $\vec{r}$ then the concatenation of the channel responses between that source location and the array elements is $\vg_{\vec{r}}\in\mathbb{C}^{MK}$.  In simple environments, these channel responses might look very similar to one another in that they are all (sampled) versions of the same shifted function (See Figure~\protect\subref*{fig:CIR_blocks}).  The delays are induced by the differences in sensor locations relative to the source, while the shape of the response might be determined by the instrumentation used to take the measurements (e.g.\ the frequency response of the sensors) --- there could even be small differences in this shape from element to element.

Suppose now that there is uncertainty in the source location that we model as $\vec{r}\in\setR$, where $\setR$ is some region in space.  As we vary $\vec{r}$ over the set $\setR$, the responses $\vg_{\vec{r}}$ trace out a portion of a manifold in $\mathbb{C}^{MK}$.  We can (approximately) embed this manifold in a linear subspace of dimension $D$ by looking at the $D$ principal eigenvectors of the matrix
\[
	\mH_{\setR} = \int_{\vec{r}\in\setR} \vg_{\vec{r}}\,\vg_{\vec{r}}^*~\d{\vec{r}}.
\]
The dimension $D$ that allows an accurate embedding will depend on the size of $\setR$ and smoothness properties of the mapping from $\vec{r}$ to $\vg_{\vec{r}}$.  In this case, we are building $\mPhi$ above by taking the $MK\times D$ matrix that has the principal eigenvectors as columns and apportioning the first $K$ rows to $\mPhi_1$, the next $K$ rows to $\mPhi_2$, etc.

This technique of embedding a parametric model into a linear space has been explored for source localization and channel estimation in underwater acoustics in \cite{mantzel14ro,tian2017multichannel}, and some analysis in the context of compressed sensing is provided in \cite{mantzel2015compressed}.  However, it is not robust in one important way.  In practice, the gains (the amplitude of the channel response) can vary between elements in the array, and this variation is enough to compromise the subspace embedding described above.  The bilinear model \eqref{eq:bilinear_cir_mdl2} explicitly accounts for these channel-to-channel variations.

In this paper, we are interested in when equations of the form \eqref{eq:mcmdl} can be solved for $\vh_m$ with the structural constraint \eqref{eq:bilinear_cir_mdl2}; we present two different algorithms for doing so in the sections below.  These effective of these algorithms  will of course affected by properties of $\mPhi$ (including the number of channels $M$ and embedding dimension $D$) as well as the number of samples $L$.  While empirical models like the one described above are used in practice (see in particular \cite{tian2017multichannel}), we will analyze generic instances of this problem, where the linear model is drawn at random.

\section{Spectral Methods for Multichannel Blind Deconvolution}
\label{sec:spectral}

A classical method for treating the multichannel blind deconvolution problem is to recast it as an eigenvalue problem: we create a correlation matrix using the measured data $\{\vy_m\}$, and estimate the channels from the smallest eigenvector\footnote{By which mean the eigenvector corresponding to the smallest eigenvalue.} of this matrix.  These methods were pioneered in the mid-1990s in \cite{xu1995least,moulines1995subspace,gurelli1995evam}, and we briefly review the central ideas in this section. The methods we present in the next section operate on the same basic principles, but explicitly enforce structural constraints on the solution.



The cross-convolution method for multichannel blind deconvolution \cite{xu1995least} follows directly from the commutivity of the convolution operator.  If there is no noise in the observations \eqref{eq:mcmdl}, then it must be the case that
\[
	\vy_{m_1}\circledast\vh_{m_2} - \vy_{m_2}\circledast\vh_{m_1} = \vzero,
	\quad\text{for all}\quad
	m_1,m_2=1,\ldots,M.
\]
Using $\mT_{\vy_m}$ as the matrix whose action is convolution with $\vy_m$ with a signal of length $K$, we see that the channel responses $\vh_{m_1}$ and $\vh_{m_2}$ must obey the linear constraints $\mT_{\vy_{m_1}}\underline{\vh}_{m_2}-\mT_{\vy_{m_2}}\underline{\vh}_{m_1}=\vzero$.   We can collect all pairs of these linear constraints into a large system, expressed as
\begin{equation}
	\label{eq:mcsys2}
	\mY \underline{\vh} = \vzero_{M(M-1)L/2,1},
\end{equation}
where $\mY \in \mathbb{C}^{M(M-1)L/2 \times MK}$ is defined by
\begin{equation}
	\label{eq:Ymat}
	\mY =
	\begin{bmatrix}
		\mY^{(1)} \\ \mY^{(2)} \\ \vdots \\ \mY^{(M-1)}
	\end{bmatrix},
	\quad
	\mY^{(i)} =
	\begin{bmatrix}
	\underbrace{
	\begin{matrix}
	\bm{0}_{L,K} & \dots & \bm{0}_{L,K} \\
	\vdots & & \vdots \\
	\bm{0}_{L,K} & \dots & \bm{0}_{L,K} \end{matrix}}_{\text{$(i-1)$ blocks}}
	&
	\underbrace{
	\begin{matrix}
	\bm{T}_{\bm{y}_{i+1}} & -\bm{T}_{\bm{y}_i} & & \\
	\vdots & & \ddots & \\
	\bm{T}_{\bm{y}_M} & & & -\bm{T}_{\bm{y}_i}
	\end{matrix}}_{\text{$(M-i+1)$ blocks}}
	\end{bmatrix}.
\end{equation}

It is shown in \cite{xu1995least,gurelli1995evam} that $\underline{\vh}$ is uniquely determined up to a scaling by \eqref{eq:mcsys2} (i.e.\ $\mY$ has a null space that is exactly $1$ dimensional) under the mild algebraic condition that the polynomials generated by the $(\underline{\vh}_m)_{m=1}^M$ have no common zero.  In the presence of noise, $\underline{\vh}$ is estimated as the minimum eigenvector of $\mY^* \mY$:
\begin{equation}
	\label{eq:kailathmin}
	\hat{\underline{\vh}} = \arg\min_{\|\vg\|_2=1} ~\vg^*\mY^*\mY\vg.
\end{equation}
Note that $\mY^* \mY$ is computed cross-correlating the outputs.  Therefore, $\mY^* \mY$ is computed at a low computational cost using the fast Fourier transform.  Furthermore, the size of $\mY^* \mY$, which is $MK \times MK$, does not grow with as the length of observation increases.  When there is white additive noise, this cross-correlation matrix will in expectation be the noise-free version plus a scaled identity.  These means that as the sample size gets large, the noise and noise-free cross-correlations will have the same eigenvectors, and the estimate \eqref{eq:kailathmin} is consistent.

A similar technique can be used if we have a linear model for the channel responses, $\underline{h}=\mPhi\vu$.  We can estimate the expansion coefficients $\vu$ by solving
\begin{equation}
	\label{eq:leemin}
	\minimize_{\vv}~\vv^*\mPhi^*\left(\mY^*\mY - \varrho\mId\right)\mPhi\vv \quad\text{subject to}\quad \|\vv\|_2=1,
\end{equation}
where $\varrho$ is a scalar that depends on the variance of the additive noise (this correction is made so that eigenstructure more closely matches that of $\mPhi^*\mY^*\mY\mPhi$ for noise-free $\mY$).  In \cite{lee2017spectral}, it was shown that a linear model can significantly improve the stability of the estimate of $\underline{\vh}$ in the presence of noise, and gave a rigorous non-asymptotic analysis of the estimation error for generic bases $\mPhi$.

\section{Non-convex Optimization Algorithms}
\label{sec:nonconvex}

Our proposed framework is to solve an optimization program similar to \eqref{eq:kailathmin} and \eqref{eq:leemin} above, but with the additional constraint that $\underline{\vh}$ obey the bilinear form \eqref{eq:bilinear_cir_mdl2}.

Given the noisy measurements $\{\vy_m\}$ in \eqref{eq:mcmdl}, we create the matrix
\[
	\mA = \mPhi^*(\mY^* \mY - \hat{\sigma}_w^2 (M-1)L \mId_{MK}) \mPhi,
\]
where $\hat{\sigma}_w^2$ is an estimate of the noise variance $\sigma_w^2$ (we will briefly discuss how to estimate the noise variance at the end of this section), and $\mY$ is formed as in \eqref{eq:Ymat} in the previous section.  We then solve a program that is similar to the eigenvalue problems above, but with a Kronecker product constraint on the expansion coefficients:
\begin{equation}
	\label{eq:const_evd}
	\minimize_{\vv,\vc,\vd}~
	\vv^*\mA\vv
	\quad\text{subject to}\quad
	\norm{\vv}_2 = 1, \quad \vv = \vc\otimes\vd.
\end{equation}
The norm and bilinear constraints make this a non-convex optimization program, and unlike the spectral methods discussed in the last section, there is no (known) computationally efficient algorithm to compute its solution.

We propose and analyze two non-convex optimization algorithms below for solving \eqref{eq:const_evd}.  The first is an alternating eigenvalue method, which iterates between minimizing for $\vc$ in \eqref{eq:const_evd} with $\vd$ fixed, then $\vd$ with $\vc$ fixed.  The second is a variation on the truncated power method \cite{yuan2013truncated}, whose iterations consist of applications of the matrix $\mA$ (just like the standard power method) followed by a projection to enforce the structural constraints.

The performance of both of these methods relies critically on constructing a suitable starting point.  We discuss one method for doing so below, then establish its efficacy in Proposition~\ref{prop:init} in Section~\ref{sec:proof_main_res} below.

\vspace{.2in}
\noindent
{\bf Alternating eigenvectors.}  While program \eqref{eq:const_evd} is non-convex, it becomes tractable if either of the terms in the tensor constraint are held constant.  If we have an estimate $\widehat{\vb}$ for $\vb$, and fix $\vd=\widehat{\vb}$, then we can solve for the optimal $\vc$ with
\[
	\minimize_{\vc} ~\vc^*\mA_{\widehat{\vb}}\vc\quad\text{subject to}\quad\|\vc\|_2=1,
\]
where
\[
	\mA_{\widehat{\vb}} = (\mId_M\otimes\hat{\vb})^*\mA(\mId_M\otimes\hat{\vb}),
	\quad
	\mId_M\otimes\hat{\vb} =
	\setstackgap{L}{1.1\baselineskip}
	\fixTABwidth{T}
	\bracketMatrixstack{
	\widehat{\vb} & \vzero & \dots & \vzero \\
	\vzero & \widehat{\vb} & \dots & \vzero \\
	\vdots & \vdots & \ddots & \vdots \\
	\vzero & \vzero & \dots & \widehat{\vb}}.
\]
The solution is the eigenvector corresponding to the smallest eigenvalue of $\mA_{\widehat{\vb}}$.  Similarly, with an estimate $\widehat{\va} = [\hat{a}_1,\dots,\hat{a}_M]^\transpose$ plugged in for $\vc$, we solve
\[
	\minimize_{\vd}~\vd^*\mA_{\widehat{\va}}\vd \quad\text{subject to}\quad\|\vd\|_2=1,
\]
where
\[
	\mA_{\widehat{\va}} = (\widehat{\va}\otimes\mId_D)^*\mA(\widehat{\va}\otimes\mId_D),
	\quad
	\widehat{\va}\otimes\mId_D =
	\begin{bmatrix}
		\hat a_1\mId \\ \hat a_2\mId \\ \vdots \\ \hat{a}_M\mId
	\end{bmatrix},
\]
which is again given by the smallest eigenvector of $\mA_{\widehat{\va}}$.

We summarize this method of ``alternating eigenvectors'' in Algorithm~\ref{alg:alteig}.  The function $\operatorname{MinEigVector}$ returns the eigenvector of the input matrix corresponding to its smallest eigenvalue.

\begin{algorithm}
\SetKwData{Left}{left}\SetKwData{This}{this}\SetKwData{Up}{up}
\SetKwFunction{Union}{Union}\SetKwFunction{MinEigVector}{MinEigVector}
\SetKwInOut{Input}{input}\SetKwInOut{Output}{output}
\LinesNumbered
\SetAlgoNoLine
\caption{Alternating Eigenvectors}
\label{alg:alteig}
\Input{$\mA$, $\vb_0$}
\Output{$\widehat{\vh}$}
$\widehat{\vb}$ $\leftarrow$ $\vb_0$\;
\While{stop condition not satisfied}{
    $\widehat{\va}$ $\leftarrow$ MinEigVector($(\mId_M \otimes \widehat{\vb}^*) \mA (\mId_M \otimes \widehat{\vb})$)\;\label{alg:update_a}
    $\widehat{\vb}$ $\leftarrow$ MinEigVector($(\widehat{\va}^* \otimes \mId_D) \mA (\widehat{\va} \otimes \mId_D)$)\;\label{alg:update_b}
}
$\widehat{\vh}$ $\leftarrow$ $\mPhi (\widehat{\va} \otimes \widehat{\vb})$\;
\end{algorithm}

\vspace{.2in}
\noindent
{\bf Rank-1 truncated power method.} A standard tool from numerical linear algebra to compute the largest eigenvector of a symmetric matrix is the {\em power method}, where the matrix is iteratively applied to a starting vector, with renormalization at each step.  (The same method can be used to compute the smallest eigenvector simply by subtracting the matrix from an appropriate scalar multiple of the identity.)  In \cite{yuan2013truncated}, a variation on this algorithm was introduced to force the iterates to be sparse.  This was done simply by hard thresholding after each application of the matrix.

Our Rank-1 truncated power method follows the same template.  We create a matrix $\mB$ by subtracting $\mA$ above from a multiple of the identity,
\[
	\mB = \gamma \mId_{MD} - \mA,
\]
then iteratively apply $\mB$ starting with an initial vector $\vu_0$.  After each application of $\mB$, we project the result onto the set of rank-$1$ matrices by computing the singular vector corresponding to the largest singular value, and then renormalize.

We summarize the {\em rank-1 truncated power method} in Algorithm~\ref{alg:tpm}.  Some care must be taken in choosing the value of $\gamma$.  We want to ensure that the smallest eigenvalue of $\mA$ gets mapped to the largest (in magnitude) eigenvalue of $\mB$, but we also want the relative gap between the largest and second largest eigenvalues of $\mB$ to be as large as possible.  In our analysis below, we use the conservative value of $\gamma=\E[\|\mA\|]$.  We also used this in the numerical results in Section~\ref{sec:num_res}.
Alternatively, one could estimate the largest rank-1 constrained ``eigenvalue'' by applying Algorithm~\ref{alg:tpm} to $\mA$ itself, which may accelerate the convergence.

\begin{algorithm}
\SetKwFunction{Rank1Approximation}{Rank1Approximation}
\SetKwInOut{Input}{input}\SetKwInOut{Output}{output}
\LinesNumbered
\SetAlgoNoLine
\caption{Rank-1 Truncated Power Method}
\label{alg:tpm}
\Input{$\mB$, $\vv_0$}
\Output{$\vv_t$, a vectorized rank-$1$ matrix whose factors are the estimates $\widehat{\va},\widehat{\vb}$}
$t \leftarrow 1$\;
\While{stop condition not satisfied}{
    $\widetilde{\vv}_t \leftarrow \mB \vv_{t-1}$\;
    $\widehat{\mV}_t \leftarrow \operatorname{Rank1Approx}\left(\operatorname{mat}(\widetilde{\vv}_t)\right)$ \;
    $\vv_t \leftarrow \operatorname{vec}(\widehat{\mV}_t)/\norm{\operatorname{vec}(\widehat{\mV}_t)}_2$ \;
    $t \leftarrow t+1$\;
}
\end{algorithm}

\noindent
{\bf Spectral initialization.}
Both the alternating eigenvectors and rank-$1$ truncated power methods require an initial estimate of the channel gains $\va$ and the basis coefficients $\vb$.  Because the program they are trying to solve is non-convex, this starting point must be chosen carefully.

Our spectral initialization is inspired from the lifting reformulation (e.g., see \cite{ahmed2014blind} for the lifting in blind deconvolution).  The observation equations \eqref{eq:mcmdl} can be recast as a linear operator acting on a $L\times D\times M$ tensor formed from the Kronecker products of the unknowns $\vx,\vb,\va$.  Let $\mathcal{A}:\mathbb{C}^{LDM} \to \mathbb{C}^{ML}$ be a linear map such that\footnote{We have defined how $\mathcal{A}$ operates on length $LDM$ vectors that can be arranged as rank-$1$ tensors.  Its action on a general vector in $\mathbb{C}^{LDM}$ can be derived by applying the expression in \eqref{eq:tensormap} to a series of $LDM$ vectors that form a separable basis for tensors in $\mathbb{C}^{L}\times\mathbb{C}^D\times\mathbb{C}^M$.}
\begin{equation}
	\label{eq:tensormap}
	\mathcal{A}(\vx \otimes \vb \otimes \va)
	=
	\begin{bmatrix}
		\vx \conv a_1 \mS^* \mPhi_1 \vb \\
		\vdots \\
		\vx \conv a_M \mS^* \mPhi_M \vb
	\end{bmatrix}.
\end{equation}
Concatenating the $\{\vy_m\}$ and $\{\vw_m\}$ into single vectors of length $ML$, we can rewrite \eqref{eq:mcmdl} as
\[
	\vy = \mathcal{A}(\mathcal{X}) + \vw,
\]
where $\mathcal{X} = \vx \otimes \vb \otimes \va$.

A natural initialization scheme is to apply the adjoint of $\mathcal{A}$ to $\vy$, then project the result onto the feasible set of vectors that can be arranged as rank-$1$ tensors (this technique is often used to initialize non-convex programs for recovering rank-$1$ matrices from linear measurements \cite{lee2010admira,jain2010guaranteed}).  However, there is no known algorithm for computing the projection onto the set of rank-1 tensors that has strong optimality guarantees.

We avoid this by exploiting the structure on the factor $\va$, in particular that the $a_m\geq 0$.  The action of the operator $(\mId_{LD}\otimes\vct{1}_{M,1})$ has the effect of summing down the third mode of the tensor; in particular
\[
	(\mId_{LD} \otimes \vct{1}_{M,1})(\vx\otimes\vb\otimes\va)= \Big(\sum_{m=1}^M a_m\Big) (\vx \otimes \vb).
\]
When the factor $\sum_{m=1}^M a_m$ has a sufficiently large magnitude, we can get an estimate of $\vx\otimes\vb$ by applying this operator to $\mathcal{A}^*\vy$.  This is the cases if the channel gains are positive.  However, without the positivity constraint on $\va$, the factor can be arbitrary small in magnitude, which may turn the initialization vulnerable to noise.  The positivity constraint on $\va$ can be weakened if estimates of the phases of $a_1,\dots,a_M$ are available as prior information.  In this scenario, the known phase information is absorbed into the basis $\mPhi$ and one can focus on estimating only the gains.

The first step of our initialization, then, is to compute
\begin{equation}
	\label{eq:def_Gamma}
	\mGamma = \operatorname{mat}\left((\mId_{LD} \otimes \vct{1}_{M,1})\mathcal{A}^*\vy\right),
\end{equation}
where the operator $\operatorname{mat}(\cdot)$ takes a vector in $\mathbb{C}^{LD}$ and produces a $L\times D$ matrix by column-major ordering.

Once corrected for noise, the leading eigenvector of $\mGamma\mGamma^*$ gives us a rough estimate of the channel coefficients $\vb$.   In Section~\ref{sec:proof:prop:init}, we show that the random matrix $\mGamma \mGamma^* - \sigma_w^2 L\sum_{m=1}^M  \mPhi_m^* \mPhi_m$ concentrates around a multiple of  $\vb \vb^*$.

Finally, we note that there is a closed-form expression for computing $\mGamma$ from the measurements $\{\vy_m\}$.  This is given in the following lemma that is proved in Appendix~\ref{sec:proof:lemma:matZ}.
\begin{lemma}
	\label{lemma:matZ}
	The matrix $\mGamma$ in \eqref{eq:def_Gamma} can be written as
	\begin{equation}
	\label{eq:def_Gamma2}
		\mGamma = \sum_{m=1}^M \mPhi_m^* \mS \mC_{\vy_m} \opconj,
	\end{equation}
    where $\mC_{\vy_m} \in \mathbb{C}^{L \times L}$ is the matrix whose action is the circular convolution with $\vy_m \in \mathbb{C}^L$, $\opconj$ is the ``flip operator'' modulo $L$:
	\begin{equation}
		\label{eq:def_flip}
			\opconj := \begin{bmatrix} \ve_1 & \ve_L & \ve_{L-1} & \cdots & \ve_2\end{bmatrix},
	\end{equation}
	and the $\ve_1,\dots,\ve_L$ are the standard basis vectors for $\mathbb{R}^L$.
\end{lemma}


We summarize our spectral initialization technique in Algorithm~\ref{alg:spectral_init}.
\begin{algorithm}
\SetKwData{Left}{left}\SetKwData{This}{this}\SetKwData{Up}{up}
\SetKwFunction{Union}{Union}\SetKwFunction{MaxEigVector}{MaxEigVector}
\SetKwInOut{Input}{input}\SetKwInOut{Output}{output}
\LinesNumbered
\SetAlgoNoLine
\caption{Spectral Initialization}
\label{alg:spectral_init}
\Input{$\{\vy_m\}_{m=1}^M$, $\{\mPhi_m\}_{m=1}^M$, $L$, and an estimate of noise variance $\hat\sigma_w^2$}
\Output{$\vb_0$}
$\mGamma \leftarrow \sum_{m=1}^M \mPhi_m^* \mS \mC_{\vy_m} \opconj$\;
$\vb_0$ $\leftarrow$ MaxEigVector($\mGamma \mGamma^* - \hat{\sigma}_w^2 L \sum_{m=1}^M \mPhi_m^* \mPhi_m$)\;
\end{algorithm}

Our analysis of the initialization, we assume that we know the noise variance $\sigma_w^2$.  Having a good estimate can indeed make a difference in terms of numerical performance.  In the numerical experiments in Section~\ref{sec:num_res}, we include simulations where we assume we know the noise variance exactly, and where we take the crude guess $\hat\sigma_w^2=0$.  The latter of course does not perform as well as the former, but it still offers significant gains over disregarding the bilinear structure all together.

It is also possible to get an estimate of the noise variance through the low-rank matrix denoising technique described in   \cite{donoho2014minimax}, where we solve the convex program
\[
	\minimize_{\mX,\alpha}
	\norm{\mGamma \mGamma^* - \alpha\sum_{m=1}^M \mPhi_m^* \mPhi_m - \mX}_{\mathrm{F}}^2 + 	
	\lambda \norm{\mX}_*,
\]
and take $\hat\sigma_w^2 = \hat\alpha/L$.  The theory developed in \cite{donoho2014minimax} for this procedure relies on the perturbation to the low-rank matrix being subgaussian, which unfortunately does not apply here, as the perturbation involves both intra- and inter-channel convolutions of the noise processes $\vw_m$.

\section{Main Results}
\label{sec:main_results}

\subsection{Non-asymptotic analysis}

Our main results give non-asymptotic performance guarantees for both Algorithm~\ref{alg:alteig} and Algorithm~\ref{alg:tpm} when their iterations start from the initial estimate by Algorithm~\ref{alg:spectral_init} under the following two assumptions\footnote{These same assumuptions were used for the analysis of the spectral method \eqref{eq:leemin} in  \cite{lee2017spectral}.}:
\begin{itemize}

  \item[(A1)] {\bf Generic subspaces.} The random matrices $\mPhi_1,\dots,\mPhi_M$ are independent copies of a $K$-by-$D$ complex Gaussian matrix whose entries are independent and identically distributed (iid) as $\mathcal{CN}(0,1)$.  Our theorems below hold with high probability with respect to $(\mPhi_m)_{m=1}^M$.

  \item[(A2)] {\bf Random noise.}  The perturbations to the measurements $\vw_1,\dots,\vw_M \in \mathbb{C}^L$ are independent subgaussian vectors with $\mathbb{E}[\vw_m]=\vzero$ and $\mathbb{E}[\vw_m\vw_m^*] = \sigma_w^2\mId_L$, and are independent of the bases $(\mPhi_m)_{m=1}^M$.

\end{itemize}

We present two main theorems in two different scenarios.  In the first, we assume that the input source is a white subgaussian random process.  In the second scenario, we assume that the input source satisfies a kind of incoherence condition that essentially ensures that it is not too concentrated in the frequency domain (a characteristic a random source has with high probability).  The error bound for the deterministic model is more general but is also slightly weaker than that for the random model.

The theorems provide sufficient conditions on the observation length $L$ that guarantees that the estimation error will fall below a certain threshold.  The number of samples we need will depend on the length of the filter responses $K$, their intrinsic dimensions $D$, the number of channels $M$, and the signal-to-noise-ratio (SNR) defined as
\begin{equation}
    \label{eq:SNRdef}
    \eta
    := \frac{\mathbb{E}_{\vphi}[\sum_{m=1}^M \norm{\vh_m \circledast \vx}_2^2]}{\mathbb{E}_{\vw}[\sum_{m=1}^M \norm{\vw_m}_2^2]}.
\end{equation}
Under (A1) and (A2), it follows from the commutativity of convolution and Lemma~\ref{lemma:expectation1} that $\eta$ simplifies as
\begin{equation}\label{eq:etasimp}
\eta = \frac{K \norm{\vx}_2^2 \norm{\vu}_2^2}{M L \sigma_w^2}.
\end{equation}
In addition, the bounds will depend on the spread of the channel gains.  We measure this disparity using the two flatness parameters
\begin{equation}
  \label{eq:defmua}
  \mu := \max_{1\leq m\leq M} \frac{\sqrt{M} a_m}{\norm{\va}_2}
\end{equation}
and
\begin{equation}
  \label{eq:defnua}
  \nu := \min_{1\leq m\leq M} \frac{\sqrt{M} a_m}{\norm{\va}_2}.
\end{equation}
Our results are most interesting when there are not too many weak channels, meaning $\mu = O(1)$ and $\nu = \Omega(1)$.  To simplify the theorem statements below, we will assume these conditions on $\mu$ and $\nu$.  It is possible, however, to re-work their statements to make the dependence on $\mu,\nu$ explicit.

We now present our first main result.  Theorem~\ref{thm:main_randx} below assumes a {\em random common source} signal $\vx$.  We present guarantees for Algorithms~\ref{alg:alteig} and \ref{alg:tpm} simultaneously, with $\underline{\vh}_t = \mPhi \vv_t$ as the channel estimate after iteration $t$ (for the alternating eigenvectors method, take $\vv_t = \hat{\va}_t\otimes\hat{\vb}_t$).

\begin{theorem}[Random Source]
	\label{thm:main_randx}
	We observe noisy channel outputs $\{\vy_m\}$ as in \eqref{eq:mcmdl}, with SNR $\eta$ as in \eqref{eq:SNRdef}, and form a sequence of estimates $(\underline{\vh}_t)_{t \in \mathbb{N}}$ of the channel responses by either Algorithm~\ref{alg:alteig} or Algorithm~\ref{alg:tpm} from the initial estimate by Algorithm~\ref{alg:spectral_init}.  Suppose assumptions (A1) and (A2) above hold, let $\vx$ be a sequence of zero-mean iid subgaussian random variables with variance $\sigma_x^2$, $\eta \geq 1$, $\mu = O(1)$, and $L \geq 3K$.\footnote{Without the subspace prior, $L > K$ is necessary to claim that $\mY^* \mY$ has nullity 1 in the noiseless case.  We used $L \geq 3K$ in the proof in order to use the identity that the circular convolutions of three vectors of length $K$ modulo $L$ indeed coincide with their linear convolution.}
	Then for any $\beta \in \mathbb{N}$, there exist absolute constants $C > 0, \alpha \in \mathbb{N}$ and constants $C_1(\beta), C_2(\beta)$  such that if there are a sufficient number of channels,
	\begin{equation}
		\label{eq:condM}
        M \geq C_1(\beta) \log^\alpha (MKL),
	\end{equation}
	that are sufficiently long,
	\begin{equation}
		\label{eq:condK}
        K \geq C_1(\beta) D \log^\alpha (MKL),
	\end{equation}
	and we have observed the a sufficient number of samples at the output of each channel,
	\begin{equation}
		\label{eq:condL_randx}
        L \geq
        \frac{C_1(\beta) \log^\alpha (MKL)}{\eta}
        \Big(
        \frac{K}{M^2} + \frac{D}{D \wedge M} \,
        \Big),
	\end{equation}
	then with probability exceeding $1-CK^{-\beta}$, we can bound the approximation error as
	\begin{equation}
		\label{eq:error_randx}
        \sin \angle(\underline{\vh}_t,\underline{\vh})
        \leq
        2^{-t} \angle(\underline{\vh}_0,\underline{\vh})
        + C_2(\beta) \log^{\alpha}(MKL)
        \Big(
        \frac{1}{\sqrt{\eta L}}
        \Big(
        \frac{\sqrt{K}}{M} + \sqrt{\frac{D}{D \wedge M}} \,
        \Big)
        + \frac{\sqrt{D}}{\eta \sqrt{ML}}
        \Big), \quad \forall t \in \mathbb{N}.
	\end{equation}

\end{theorem}

\begin{rem}
\rm
The SNR requirement $\eta \geq 1$ was introduced to simplify the expressions in Theorem~\ref{thm:main_randx}.  The conditions in the low SNR regime $\eta <1$ can be easily extracted from the proof of the theorem and Proposition~\ref{prop:conv_both} below.
\end{rem}

We make the following remarks about the assumption \eqref{eq:condM} --\eqref{eq:condL_randx} in Theorem~\ref{thm:main_randx}.  The lower bound on the number of channels in \eqref{eq:condM} is very mild, $M$ has to be only a logarithmic factor of the number of parameters involved in the problem.  The condition \eqref{eq:condK} allows a low-dimensional subspace, the dimension of which scales proportional to the length of filter $K$ up to a logarithmic factor.
For a fixed SNR and a large number of channels ($M = \Omega(\sqrt{K/D})$), the condition in \eqref{eq:condL_randx} says that the length of observation can grow proportional to $\sqrt{KD}$ --- this is suboptimal when compared to the degrees of freedom $D$ per channel (as $K>D$), with the looseness probably being an artifact of how our proof technique handles the fact that the channels are time-limited to legnth $K$.
However, this still marks a significant improvement over an earlier analysis of this problem \cite{lee2016fast_arxiv}, which depended on the concentration of subgaussian polynomial \cite{adamczak2015concentration} and union bound arguments.
The scaling laws of parameters have been sharpened significantly, and as we will see in the next section, its prediction is consistent with the empirical results by Monte Carlo simulations in Section~\ref{sec:num_res}.  Compared to the analysis for the other spectral method under the linear subspace model \cite{lee2017spectral}, Theorem~\ref{thm:main_randx} shows that the estimation error becomes smaller by factor $\sqrt{D}$.

To prove Theorem~\ref{thm:main_randx}, we establish an intermediate result for the case where the input signal $\vx$ is deterministic.  In this case, our bounds depend on the spectral norm $\rho_{\vx}$ of the (appropriately restricted) autocorrelation matrix of $\vx$,
\[
	\rho_{\vx} := \|\widetilde{\mS}\mC_{\vx}^*\mC_{\vx}\widetilde{\mS}^*\|,
\]
where
\begin{equation}
\label{eq:deftildeS}
\widetilde{\mS} =
\begin{bmatrix}
\begin{bmatrix} \bm{0}_{K-1,L-K+1} & \mId_{K-1} \end{bmatrix} \\
\begin{bmatrix} \mId_{2K-1} & \bm{0}_{2K-1,L-2K+1} \end{bmatrix}
\end{bmatrix}.
\end{equation}

Then the deterministic version of our recovery result is:
\begin{theorem}[Deterministic Source]
	\label{thm:main_detx}
	Suppose that the same assumptions hold as in Theorem~\ref{thm:main_randx}, only with $\vx$ as a fixed sequence of numbers obeying
	\begin{equation}
		\label{eq:cond_rho}
		\rho_{\vx} \leq C_3\|\vx\|_2^2.
	\end{equation}
	If  \eqref{eq:condK} and \eqref{eq:condM} hold, and
	\begin{equation}
		\label{eq:condL_detx}
        L \geq
        \frac{C_1(\beta) \log^\alpha (MKL)}{\eta}
        \Big(
        \frac{K^2}{M^2} + \frac{KD}{D \wedge M}
        \Big),
	\end{equation}
	then with probability exceeding $1-CK^{-\beta}$, we can bound the approximation error as
	\begin{equation}
		\label{eq:error_detx}
        \sin \angle(\underline{\vh}_t,\underline{\vh})
        \leq
        2^{-t} \angle(\underline{\vh}_0,\underline{\vh})
        + \frac{C_2(\beta) \log^{\alpha}(MKL)}{\sqrt{\eta L}}
        \Big(
        \frac{K}{M} + \sqrt{\frac{KD}{D \wedge M}} \,
        \Big), \quad \forall t \in \mathbb{N}.
	\end{equation}
\end{theorem}

The condition \eqref{eq:cond_rho} can be interpreted as a kind of incoherence condition on the input signal $\vx$.  Since
\[
	\rho_{\vx} \leq \|\mC_{\vx}\|^2 = L \norm{\widehat{\vx}}_\infty^2,
\]
where $\widehat{\vx} \in \mathbb{C}^L$ is the normalized discrete Fourier transform of $\vx$, it is sufficient that $\hat\vx$ is approximately flat for \eqref{eq:cond_rho} to hold.  This is a milder assumption than imposing an explicit stochastic model on $\vx$ as in Theorem~\ref{thm:main_randx}.  For the price of this relaxed condition, the requirement on $L$ in \eqref{eq:condL_detx} that activates Theorem~\ref{thm:main_detx} is more stringent compared to the analogous condition \eqref{eq:condL_randx} in Theorem~\ref{thm:main_randx}.

\subsection{Proof of main results}
\label{sec:proof_main_res}

The main results in Theorems~\ref{thm:main_randx} and \ref{thm:main_detx} are obtained by the following proposition, the proof of which is deferred to Section~\ref{sec:proof:prop:conv_both}.

\begin{prop}
\label{prop:conv_both}
Suppose the assumptions in (A1) and (A2) hold, $\rho_x$ satisfies \eqref{eq:cond_rho}, $L \geq 3K$, $\mu = O(1)$, and $\nu = \Omega(1)$.
For any $\beta \in \mathbb{N}$, there exist absolute constants $C > 0, \alpha \in \mathbb{N}$ and constants $C_1, C_2$ that only depend on $\beta$, for which the following holds: If
\begin{equation}
\label{eq:condK_prop_both}
K \geq C_1 D \log^\alpha (MKL),
\end{equation}
\begin{equation}
\label{eq:condM_prop_both}
M \geq C_1 \log^\alpha (MKL),
\end{equation}
and
\begin{equation}
\label{eq:condL_prop_both}
L \geq
C_1 \log^\alpha (MKL)
\Big[
\frac{\rho_{x,w}^2}{\eta K \sigma_w^2 \norm{\vx}_2^2}
\Big(
\frac{D}{K \wedge M} + \frac{K}{M^2} + 1
\Big)
+ \frac{D}{\eta^2}
\Big]
\end{equation}
then
\begin{equation}
\sin \angle(\underline{\vh}_t, \underline{\vh}) \leq 2^{-t} \sin \angle(\underline{\vh}_0, \underline{\vh})
+ \kappa, \quad \forall t \in \mathbb{N}
\label{eq:conv_prop_both}
\end{equation}
with probability $1-CK^{-\beta}$, where $\kappa$ satisfies \eqref{eq:updatea_errbnd_prop}.
\end{prop}

Then the proofs for Theorems~\ref{thm:main_randx} and \ref{thm:main_detx} are given by combining Proposition~\ref{prop:conv_both} with the following lemmas, taken from \cite{lee2017spectral}, which provide tail estimates on the signal autocorrelation and the signal-noise cross correlation.

\begin{lemma}[{\hspace{1sp}\cite[Lemma~3.9]{lee2017spectral}}]
\label{lemma:rho_wx}
Suppose (A2) holds and let $\vx$ be a fixed sequence of numbers obeying \eqref{eq:cond_rho}. For any $\beta \in \mathbb{N}$, there exists an absolute constant $C$ such that
\[
\rho_{x,w} \leq C K \sigma_w \sqrt{\rho_x} \sqrt{1 + \log M + \beta \log K}
\]
holds with probability $1 - K^{-\beta}$.
\end{lemma}

\begin{lemma}[{\hspace{1sp}\cite[Lemma~3.10]{lee2017spectral}}]
\label{lemma:rho_wx_randx}
Suppose (A2) holds and let $\vx$ be a sequence of zero-mean iid subgaussian random variables with variance $\sigma_x^2$.  Then
\[
\frac{\rho_x}{\norm{\vx}_2^2}
\leq
\frac{L + C_\beta K \log^{3/2} L \sqrt{\log K}}{L - \sqrt{2L \beta \log K}}
\]
and
\[
\frac{\rho_{x,w}}{\sigma_w \norm{\vx}_2}
\leq
\frac{C_\beta K \log^5 (MKL)}{\sqrt{L - \sqrt{2L \beta \log K}}}
\]
hold with probability $1-3K^{-\beta}$.
\end{lemma}

\subsection{Proof of Proposition~\ref{prop:conv_both}}
\label{sec:proof:prop:conv_both}

The proof of Proposition~\ref{prop:conv_both} is given by a set of propositions, which provide guarantees for Algorithms~\ref{alg:spectral_init}, \ref{alg:alteig} and Algorithm~\ref{alg:tpm}.  The first proposition provides a performance guarantee for the initialization by Algorithm~\ref{alg:spectral_init}. The proof of Proposition~\ref{prop:init} is given in Section~\ref{sec:proof:prop:init}.

\begin{prop}[Initialization]
\label{prop:init}
Suppose the assumptions in (A1) and (A2) hold, $\rho_x$ satisfies \eqref{eq:cond_rho}, and $L \geq 3K$.  Let $\mu,\nu,\eta$ be defined in \eqref{eq:defmua}, \eqref{eq:defnua}, \eqref{eq:etasimp}, respectively.  For any $\beta \in \mathbb{N}$, there exist absolute constants $C > 0, \alpha \in \mathbb{N}$ and constants $C_1, C_2$ that only depend on $\beta$, for which the following holds:
If
\begin{equation}
\label{eq:init_condM}
M \geq C_1 \log^\alpha (MKL) \cdot \Big(\frac{\mu}{\nu}\Big)^2
\end{equation}
and
\begin{equation}
\label{eq:init_condL}
L \geq
C_1 \log^\alpha (MKL)
\cdot
\Big[\frac{\rho_{x,w}^2}{\eta K \sigma_w^2 \norm{\vx}_2^2} \cdot
\Big(\frac{\mu^2 K}{\nu^4 M^2} + \frac{D}{\nu^2 M}\Big)
+ \frac{D}{\eta^2 \nu^4 M}\Big],
\end{equation}
then the estimate $\widehat{\vb}$ by Algorithm~\ref{alg:spectral_init} satisfies
\begin{equation}
\label{eq:init_errbnd}
\sin \angle(\widehat{\vb},\vb)
\leq
C_2 \log^\alpha (MKL)
\Big[
\frac{\mu}{\nu \sqrt{M}}
+ \frac{\rho_{x,w}}{\sqrt{\eta KL} \sigma_w \norm{\vx}_2} \cdot
\Big(\frac{\mu \sqrt{K}}{\nu^2 M} + \frac{\sqrt{D}}{\nu \sqrt{M}}\Big)
+ \frac{\sqrt{D}}{\eta \nu^2 \sqrt{ML}}
\Big]
\end{equation}
holds with probability $1-CK^{-\beta}$.
\end{prop}

The second proposition, proved in Section~\ref{subsec:proof_prop_update_a}, provides a performance guarantee for the update of $\widehat{\va}$ by Step~\ref{alg:update_a} of Algorithm~\ref{alg:alteig}.

\begin{prop}[Update of Channel Gains]
\label{prop:update_a}
Suppose the assumptions in (A1) and (A2) hold, $\rho_x$ satisfies \eqref{eq:cond_rho}, $L \geq 3K$, and the previous estimate $\widehat{\vb}$ satisfies
\begin{equation}
\label{eq:starterrb}
\angle(\vb,\widehat{\vb}) \leq \frac{\pi}{4}.
\end{equation}
For any $\beta \in \mathbb{N}$, there exist absolute constants $C > 0, \alpha \in \mathbb{N}$ and constants $C_1, C_2$ that only depend on $\beta$, for which the following holds: If
\begin{equation}
\label{eq:updatea_condK_prop}
K \geq C_1 \mu^4 D \log^\alpha (MKL),
\end{equation}
\begin{equation}
\label{eq:updatea_condM_prop}
M \geq C_1 \mu^4 \log^\alpha (MKL),
\end{equation}
and
\begin{equation}
\label{eq:updatea_condL_prop}
L \geq
C_1 \log^\alpha (MKL)
\Big[
\frac{\rho_{x,w}^2}{\eta K \sigma_w^2 \norm{\vx}_2^2}
\Big(
\mu^2 \Big( \frac{D}{K \wedge M} + \frac{K}{M^2} \Big) + 1
\Big)
+ \frac{D}{\eta^2}
\Big]
\end{equation}
then the updated $\widehat{\va}$ by Step~\ref{alg:update_a} of Algorithm~\ref{alg:alteig} satisfies
\begin{equation}
\label{eq:updatea_recursion}
\sin\angle(\va,\widehat{\va})
\leq \frac{1}{2} \sin\angle(\vb,\widehat{\vb}) + \kappa
\end{equation}
with probability $1-CK^{-\beta}$, where $\kappa$ satisfies
\begin{equation}
\label{eq:updatea_errbnd_prop}
\kappa \leq
C_2 \log^\alpha(MKL) \Big[ \frac{\rho_{x,w}}{\sqrt{\eta KL}\sigma_w \norm{\vx}_2}
\Big(
\mu\Big(\frac{\sqrt{K}}{M} + \sqrt{\frac{D}{M}} + \sqrt{\frac{D}{K}} \, \Big) + 1
\Big)
+ \frac{\sqrt{D}}{\eta \sqrt{ML}}
\Big].
\end{equation}
\end{prop}

We have a similar result for the update of $\widehat{\vb}$ by Step~\ref{alg:update_b} of Algorithm~\ref{alg:alteig}, which is stated in the following proposition. The proof of Proposition~\ref{prop:update_b} is provided in Section~\ref{subsec:proof_prop_update_b}.

\begin{prop}[Update of Subspace Coefficients]
\label{prop:update_b}
Suppose the assumptions in (A1) and (A2) hold, $\rho_x$ satisfies \eqref{eq:cond_rho}, $L \geq 3K$, and the previous estimate $\widehat{\va}$ satisfies
\begin{equation}
\label{eq:starterra}
\angle(\va,\widehat{\va}) \leq \frac{\pi}{4}.
\end{equation}
For any $\beta \in \mathbb{N}$, there exist absolute constants $C > 0, \alpha \in \mathbb{N}$ and constants $C_1, C_2$ that depend on $\beta$, for which the following holds: If \eqref{eq:updatea_condK_prop}, \eqref{eq:updatea_condM_prop}, and \eqref{eq:updatea_condL_prop} are satisfied, then the updated $\widehat{\vb}$ by Step~\ref{alg:update_b} of Algorithm~\ref{alg:alteig} satisfies
\[
\sin\angle(\vb,\widehat{\vb})
\leq \frac{1}{2} \sin\angle(\va,\widehat{\va}) + \kappa
\]
with probability $1-CK^{-\beta}$, where $\kappa$ satisfies \eqref{eq:updatea_errbnd_prop}.
\end{prop}


The next proposition shows the convergence of the rank-1 truncated power method from a good initialization. See Section~\ref{sec:conv_tpm} for the proof.

\begin{prop}[Local Convergence of Rank-1 Truncated Power Method]
\label{prop:tpm4mbd}
Suppose the assumptions in (A1) and (A2) hold, $\rho_x$ satisfies \eqref{eq:cond_rho}, and $L \geq 3K$.  Let $0 < \mu < 1$, $0 < \tau < \frac{1}{3\sqrt{2}}$, and
\[
c(\mu,\tau) = \min\Big(\frac{1}{\mu\sqrt{1-\tau^2}}, \frac{(1+\mu)\tau}{1-\mu}\Big).
\]
For any $\beta \in \mathbb{N}$, there exist absolute constants $C > 0, \alpha \in \mathbb{N}$, constants $C_1', C_2'$ that only depend on $\beta$, for which the following holds: If \eqref{eq:updatea_condK_prop}, \eqref{eq:updatea_condM_prop}, and \eqref{eq:updatea_condL_prop} are satisfied for $C_1 = c(\mu,\tau) C_1', C_2 = c(\mu,\tau) C_2'$ and $\vu_0$ satisfies
\[
\sin \angle(\vu_0, \vu) \leq \tau,
\]
then $(\vu_t)_{t\in\mathbb{N}}$ by Algorithm~\ref{alg:tpm} for $\mB = \norm{\mathbb{E}[\mA]} \mId_{MD} - \mA$ with $\vu_0$ satisfies
\begin{equation}
\sin \angle(\vu_t, \vu) \leq \mu^t \sin \angle(\vu_0, \vu)
+ \frac{(1+\mu)\kappa}{1-\mu}, \quad \forall t \in \mathbb{N}
\label{eq:conv_tpm_mbd}
\end{equation}
with probability $1-CK^{-\beta}$, where $\kappa$ satisfies \eqref{eq:updatea_errbnd_prop}.
\end{prop}

Finally, we derive the proof of Proposition~\ref{prop:conv_both} by combining the above propositions.

\begin{proof}[Proof of Proposition~\ref{prop:conv_both}]
Similarly to the proof of \cite[Proposition~3.3]{lee2017spectral}, we show that
\begin{equation}
\label{eq:rel_angles}
\sin \angle(\underline{\vh}_t, \underline{\vh})
\leq \frac{\sigma_{\max}(\mPhi)}{\sigma_{\min}(\mPhi)} \cdot \sqrt{2} \sin \angle(\vu_t,\vu)
\end{equation}
and
\begin{equation}
\label{eq:rel2_angles}
\sin \angle(\vu_t,\vu)
\leq \max\Big[\sin \angle(\va_t,\va), \sin \angle(\vb_t,\vb)\Big].
\end{equation}
Furtheremore, as we choose $C_1$ in \eqref{eq:condK_prop_both} sufficiently large, we can upper bound the condition number of $\mPhi$ by a constant (e.g., 3) with high probability. We proceed the proof under this event.  Then the convergence results in Propositions~\ref{prop:update_a}, \ref{prop:update_b}, and \ref{prop:tpm4mbd} imply \eqref{eq:conv_prop_both}.

Since $\mu = O(1)$, the conditions in \eqref{eq:updatea_condK_prop}, \eqref{eq:updatea_condM_prop}, \eqref{eq:updatea_condL_prop} respectively reduce \eqref{eq:condK_prop_both}, \eqref{eq:condM_prop_both}, \eqref{eq:condL_prop_both}.  Furthermore, since $\nu = \Omega(1)$, \eqref{eq:init_condL} is implied by \eqref{eq:condL_prop_both}.  By choosing $C_1$ large enough, we can make the initial error bound in \eqref{eq:init_errbnd} small so that the conditions for previous estimates in Propositions~\ref{prop:update_a}, \ref{prop:update_b}, \ref{prop:tpm4mbd} are satisfied and the assertion is obtained by these propositions.
\end{proof}

\section{Numerical Results}
\label{sec:num_res}

In this section, we provide observation on empirical performance of the alternating eigenvectors method (AltEig) in Algorithm~\ref{alg:alteig} and the rank-1 truncated power method (RTPM) in Algorithm~\ref{alg:tpm}, both initialized by the spectral initialization in Algorithm~\ref{alg:spectral_init}.  We compare the two iterative algorithms to the classical cross-convolution method (CC) by Xu et al. \cite{xu1995least}, which only imposes the time-limited model on impulse responses, and to the subspace-constrained cross-convolution method (SCCC) \cite{lee2017spectral}, which imposes a linear subspace model on impulse responses.  This comparison will demonstrate how the estimation error improves progressively as we impose a stronger prior model on impulse responses.

In our first experiment, we tested the algorithms on generic data where the basis $\mPhi$ is an i.i.d. Gaussian matrix.  The input source signal $\vx$, subspace coefficient vector $\vb$, and additive noise are i.i.d. Gaussian too.  The channel gain vector is generated by adding random perturbation to all-one vector so that $\va = \bm{1}_{M,1} + \alpha \vxi/\norm{\vxi}_\infty$, where $\vxi = [\xi_1,\dots,\xi_M]^\transpose$ and $\xi_1,\dots,\xi_M$ are independent copies of a uniform random variable on $[-1,1)$.  We use a performance metric given as the 95th percentile of the estimation error in the sine of the principal angle between the estimate and the ground truth out of 1,000 trials.  This amounts to the error for the worst-case except 5\% of the instances.  In other words, the estimation error is less than this threshold with high probability no less than 0.95.

\begin{figure}
  \centering
  \subfloat[]{\includegraphics[width=2.1in]{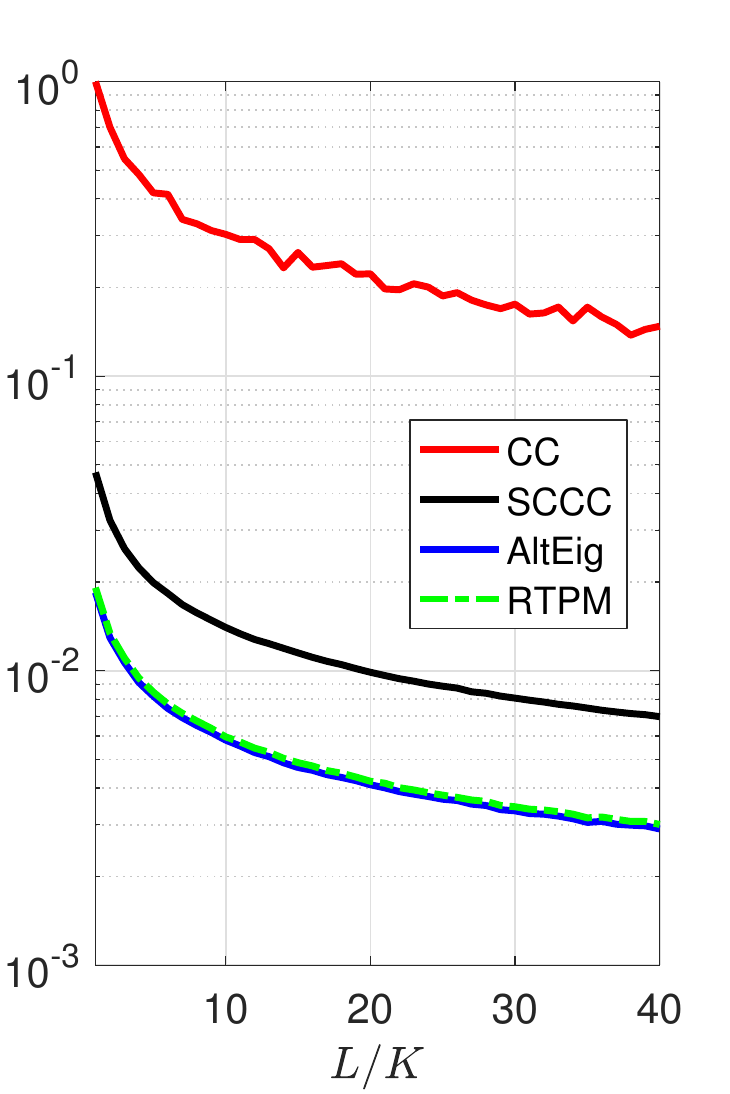}\label{fig:vary_L}}
  \hfil
  \subfloat[]{\includegraphics[width=2.1in]{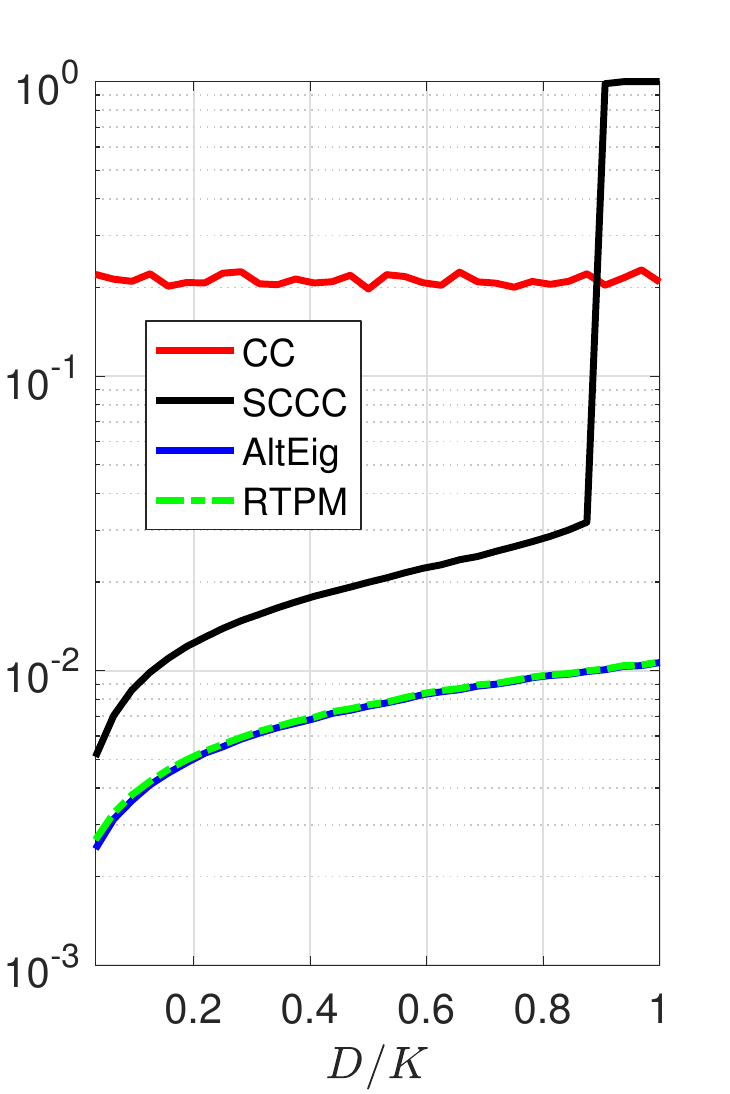}\label{fig:vary_D}}
  \\
  \subfloat[]{\includegraphics[width=2.1in]{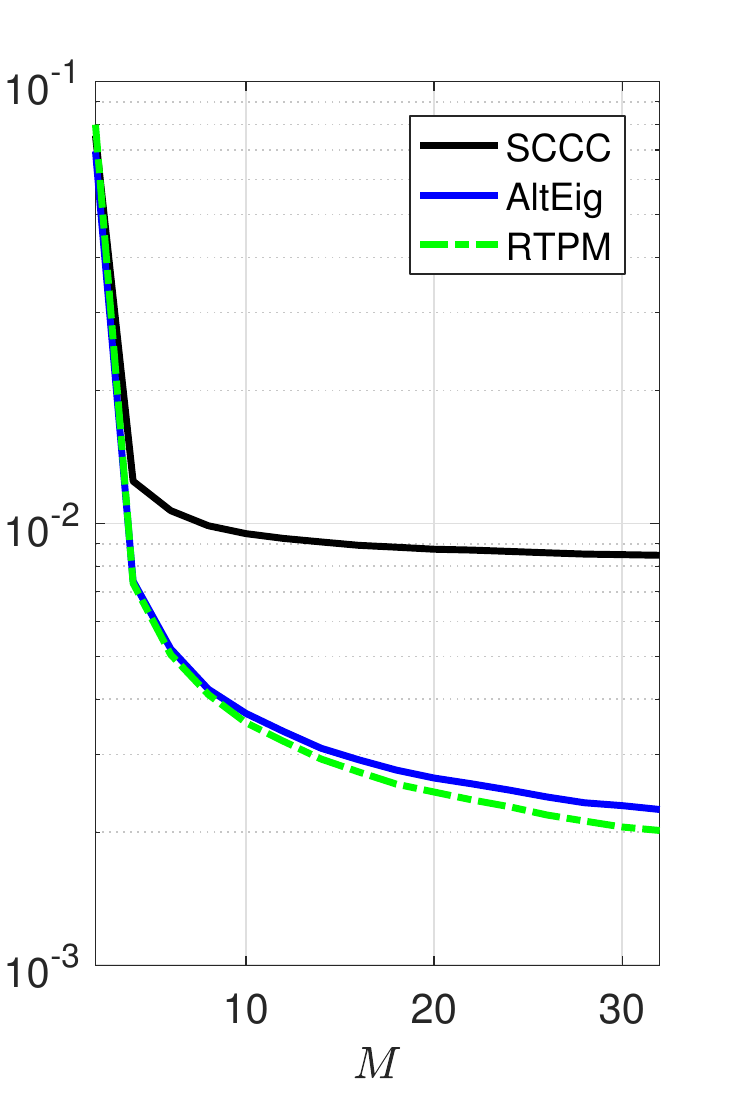}\label{fig:vary_M}}
  \hfil
  \subfloat[]{\includegraphics[width=2.1in]{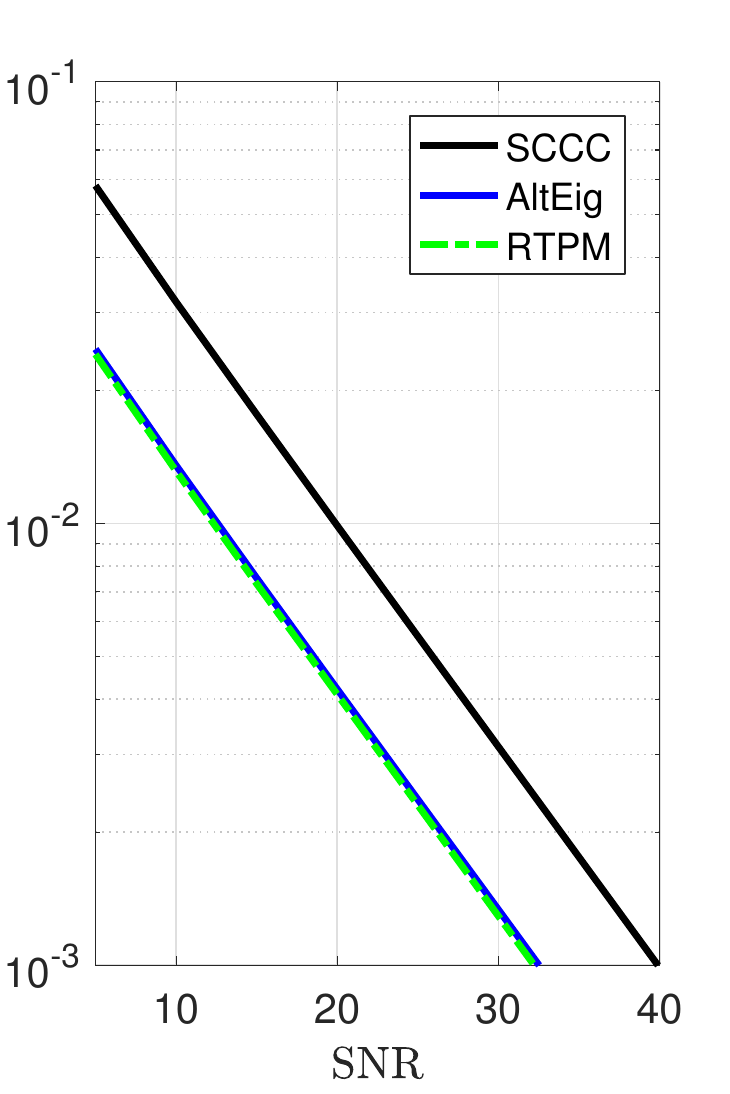}\label{fig:vary_SNR}}
  \hfil
  \subfloat[]{\includegraphics[width=2.1in]{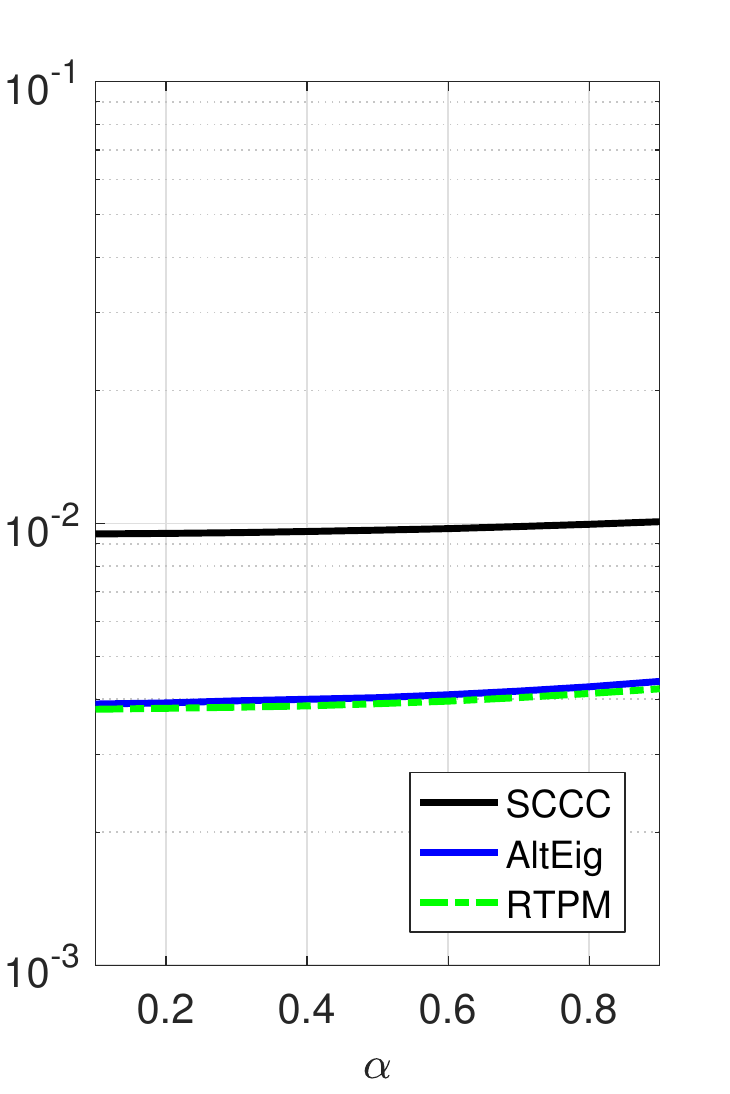}\label{fig:vary_alpha}}
  \caption{Comparison of cross-convolution (CC), subspace-constrained cross-convolution (SCCC), alternating eigenvectors method (AltEig), and rank-1 truncated power method (RTPM). Default parameter setting: $M = 8$, $K = 256$, $D = 32$, $L = 20 K$, SNR = 20 dB. The 95th percentile estimation error is plotted in a logarithmic scale as we vary each parameter as follows:
  \protect\subref{fig:vary_L} $L$;
  \protect\subref{fig:vary_D} $D$;
  \protect\subref{fig:vary_M} $M$;
  \protect\subref{fig:vary_SNR} SNR;
  \protect\subref{fig:vary_alpha} $\alpha$.
  }
  \label{fig:cgsp_line}%
\end{figure}

\begin{figure}
  \centering
  \includegraphics[width=2.1in]{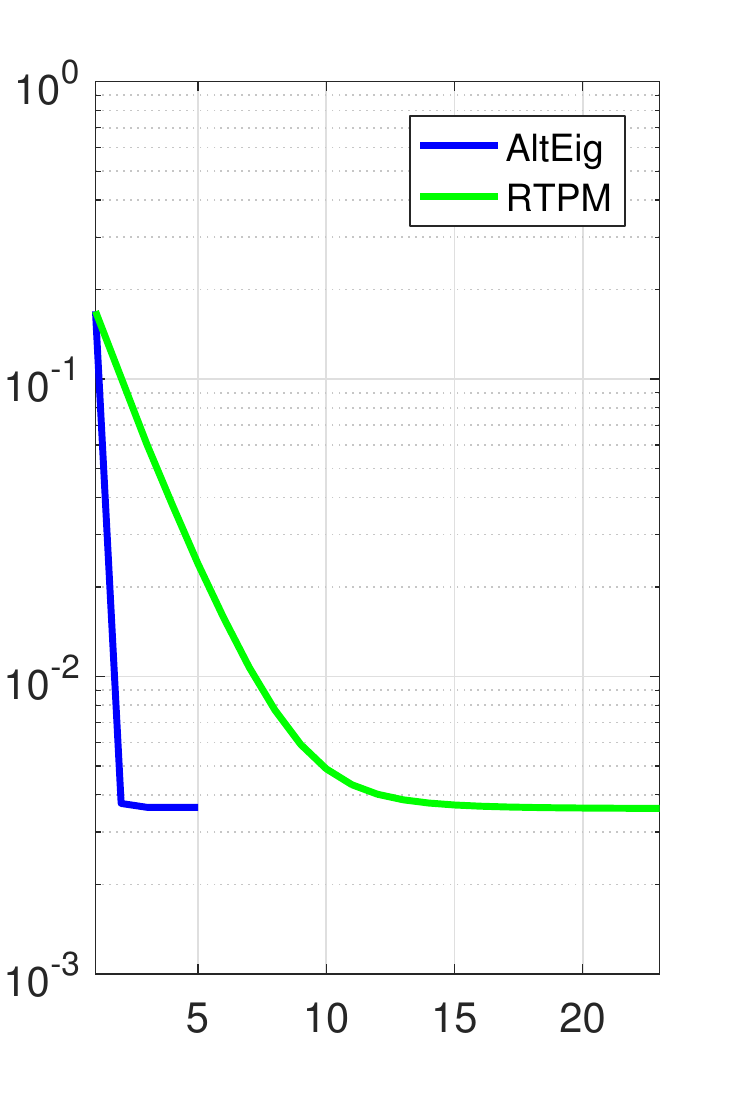}
  \caption{Convergence of alternating eigenvectors method (AltEig) and rank-1 truncated power method (RTPM) for a random instance. $x$-axis: iteration index, $y$-axis: log of the estimation error. $M = 8$, $K = 256$, $D = 32$, $L = 20 K$, SNR = 20 dB.}
  \label{fig:convergence}%
\end{figure}

Figure~\ref{fig:cgsp_line} compares the estimation error by the four algorithms as we vary the problem parameters.  Figure~\protect\subref*{fig:vary_L} shows that the error as a function of the oversampling factor $L/K$, which is the ratio of the length of observation $L$ to the number of nonzero coefficients in each impulse response.  SCCC provides smaller estimation error than CC in order of magnitude by exploiting the additional linear subspace prior.  Then AltEig and RTPM provide further reduced estimation error again in order of magnitude compared to SCCC by exploiting the bilinear prior that imposes the separability structure in addition to the linear subspace prior.  As expected, longer observation provides smaller estimation error for all methods.  Furthermore, as shown in Figure~\protect\subref*{fig:vary_D}, the estimation error increases as a function of the ratio $D/K$, which accounts for the relative dimension of the subspace.  More interestingly, as our main theorems imply, the performance difference between SCCC and AltEig/RTPM becomes more significant as we add more channels (Figure~\protect\subref*{fig:vary_M}).  The estimation error by these method scales proportionally as a function of SNR (Figure~\protect\subref*{fig:vary_SNR}).  Similarity of channel gains, that is implied by parameter $\alpha$, did not affect the estimation error significantly (Figure~\protect\subref*{fig:vary_alpha}).  Moreover, when the two iterative algorithms (AltEig and RTPM) provide stable estimate, they converge fast.  Figure~\ref{fig:convergence} illustrate the convergence of the two algorithms for a random instance.  The estimation error decays progressively for RTPM, whereas AltEig converges faster within less than 5 iterations.

\begin{figure}
  \centering
  \subfloat[]{\includegraphics[width=2.3in]{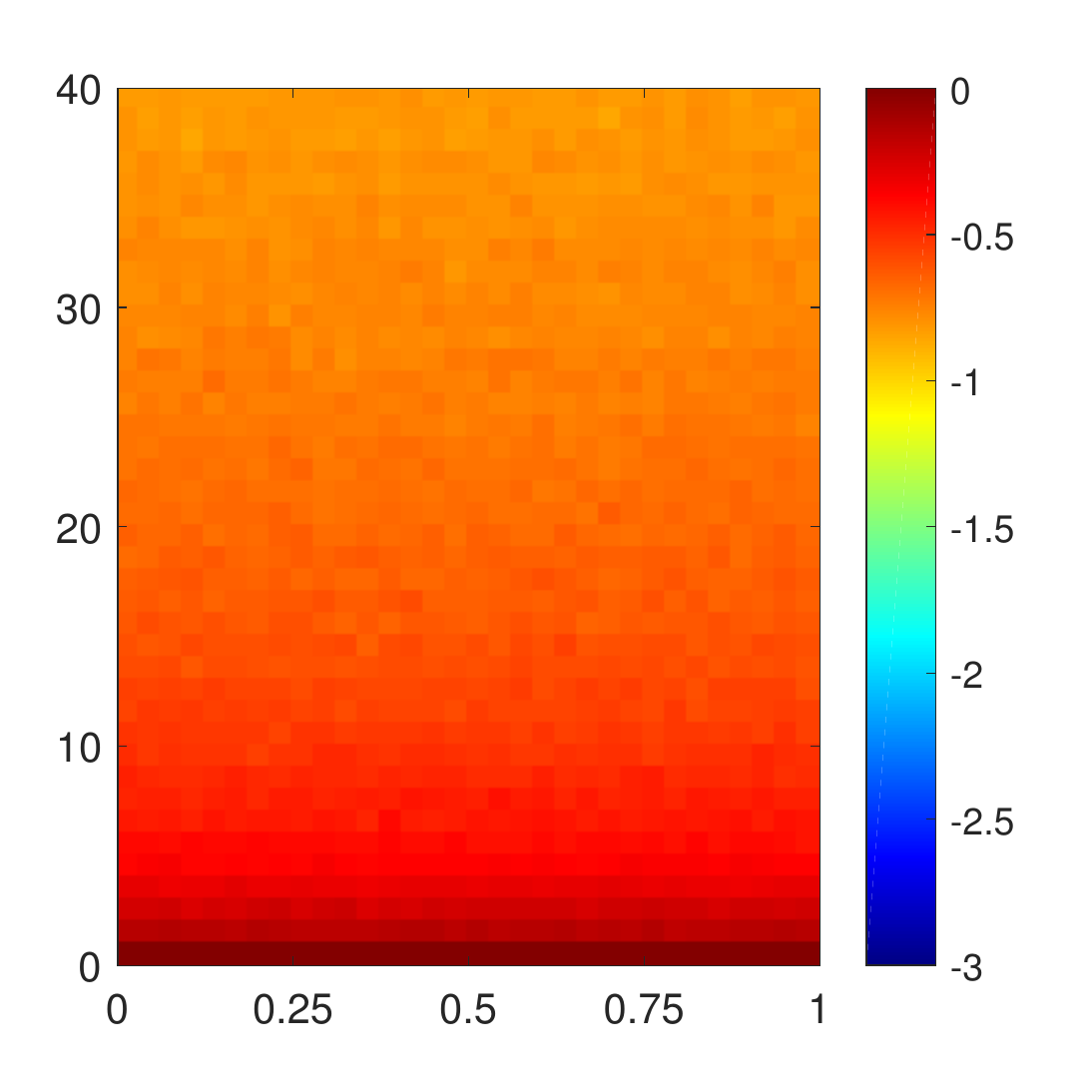}\label{fig:cgsp_v5pt_cc}}
  \hfil
  \subfloat[]{\includegraphics[width=2.3in]{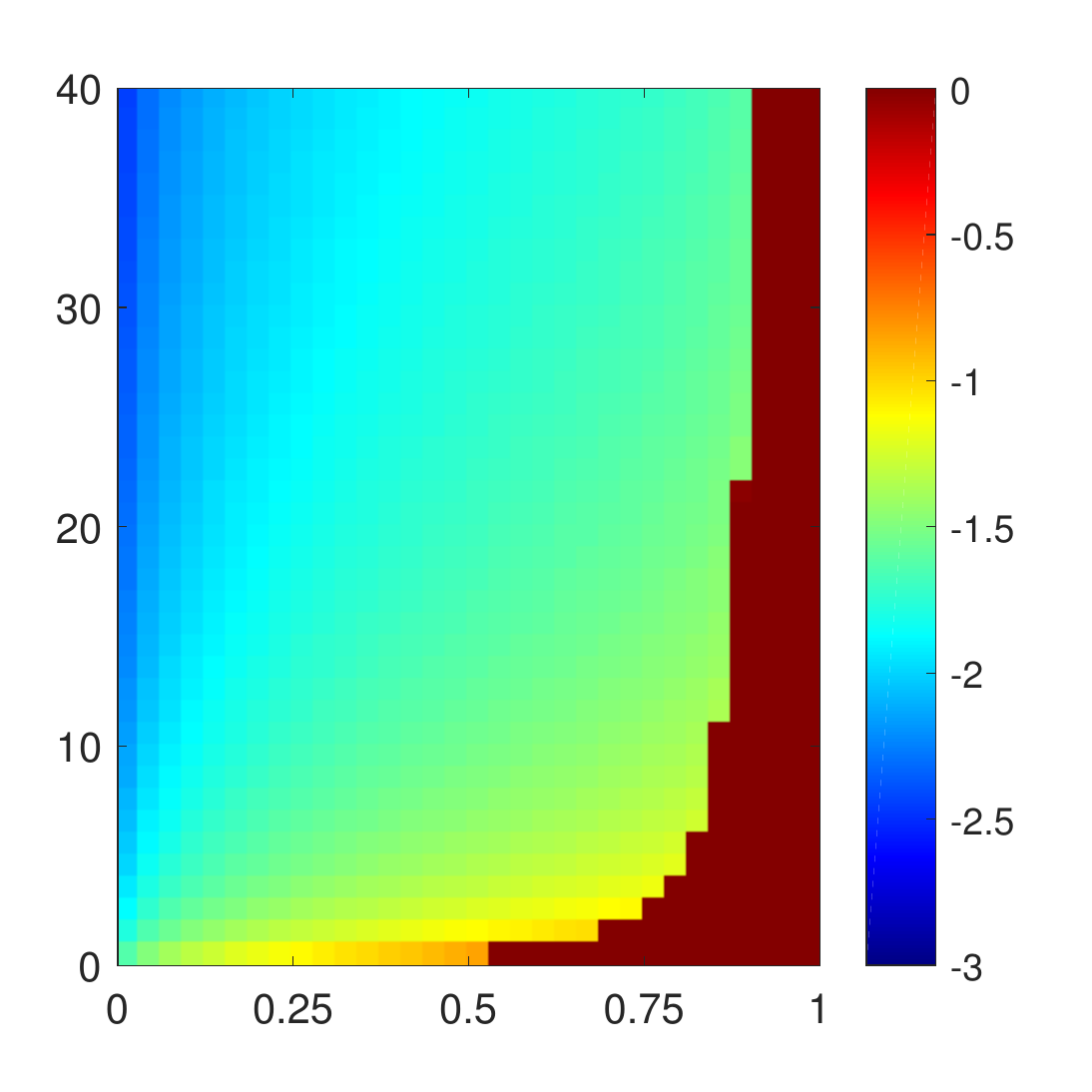}\label{fig:cgsp_v5pt_sccc}}
  \\
  \subfloat[]{\includegraphics[width=2.3in]{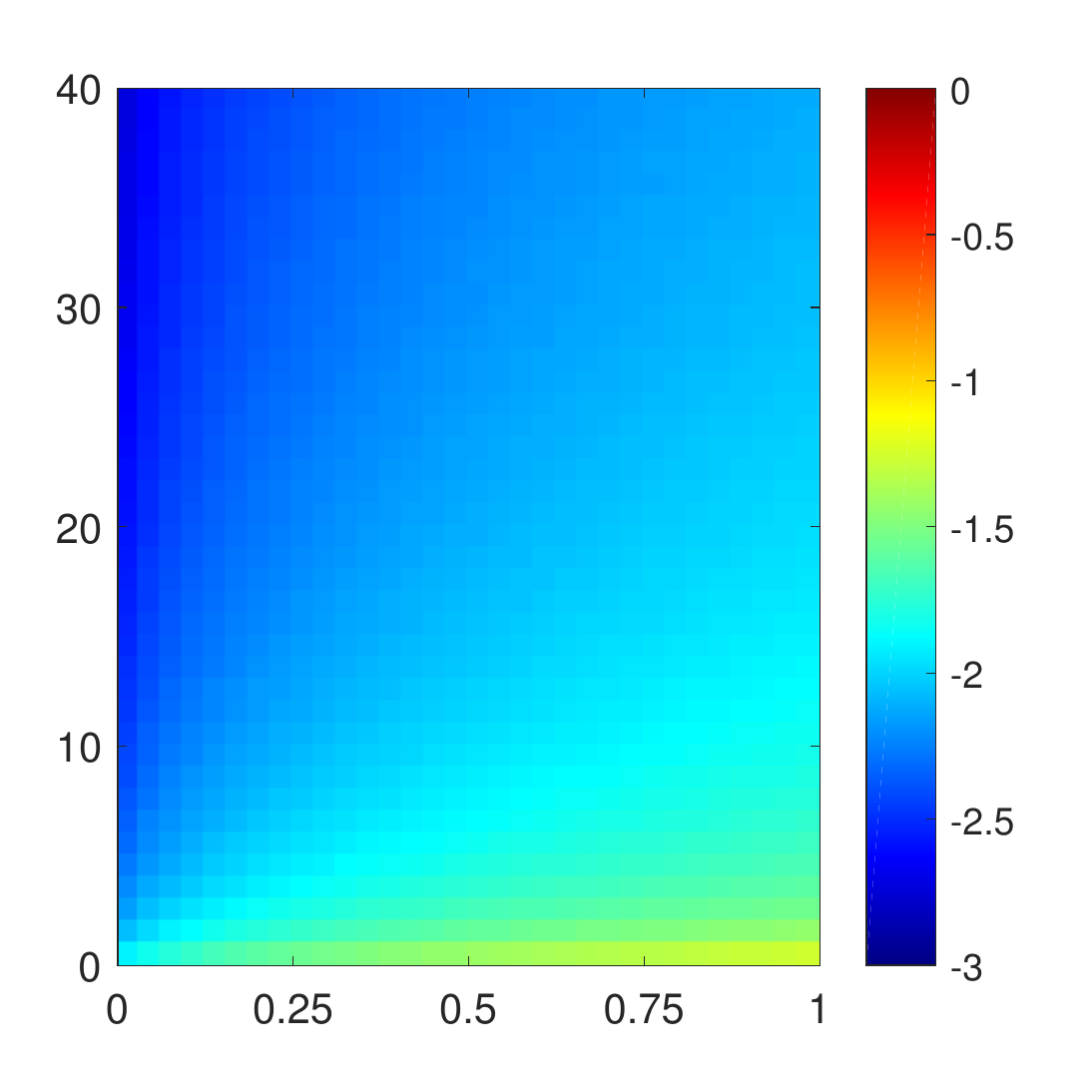}\label{fig:cgsp_v5pt_alteig_nr}}
  \hfil
  \subfloat[]{\includegraphics[width=2.3in]{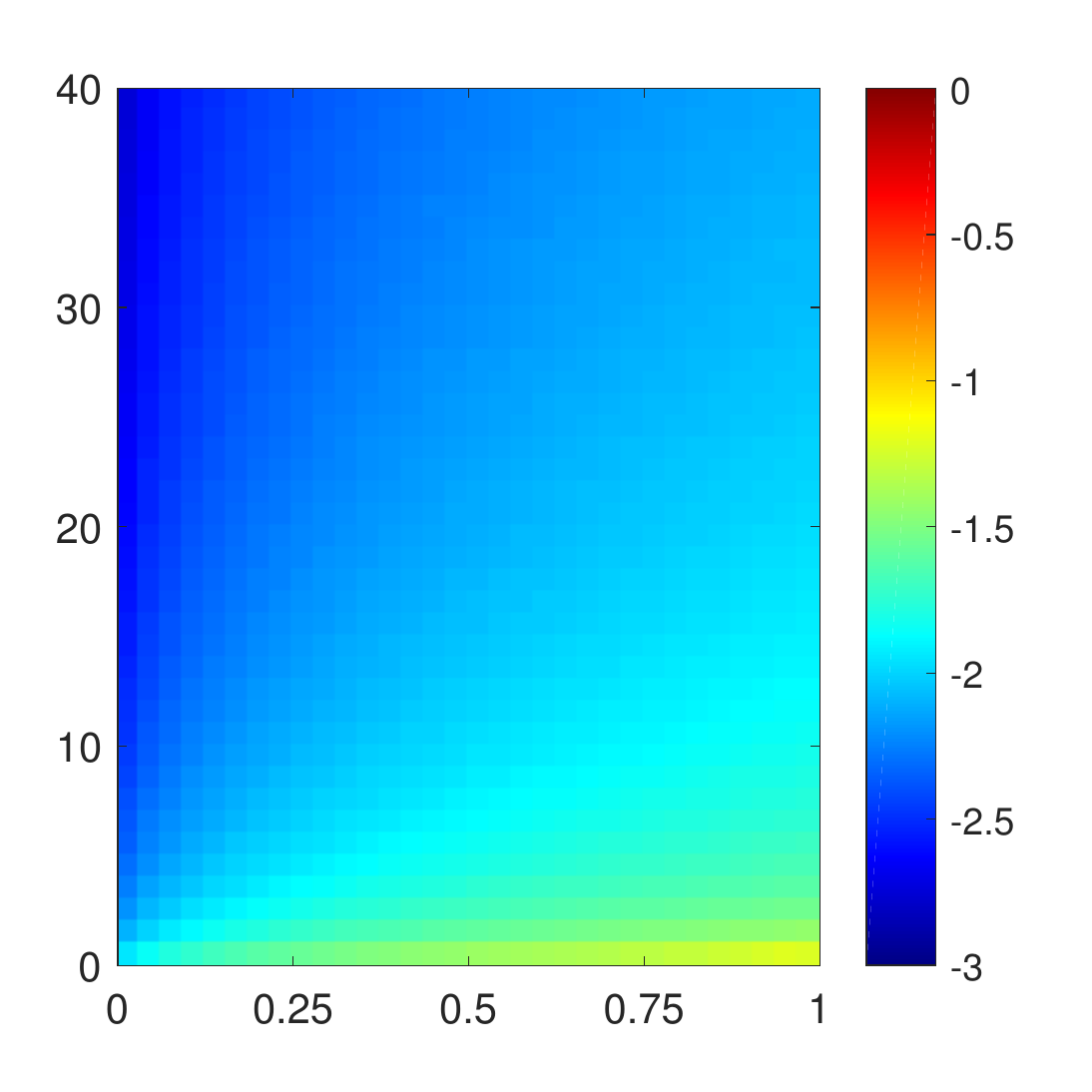}\label{fig:cgsp_v5pt_tpm_nr}}
  \\
  \subfloat[]{\includegraphics[width=2.3in]{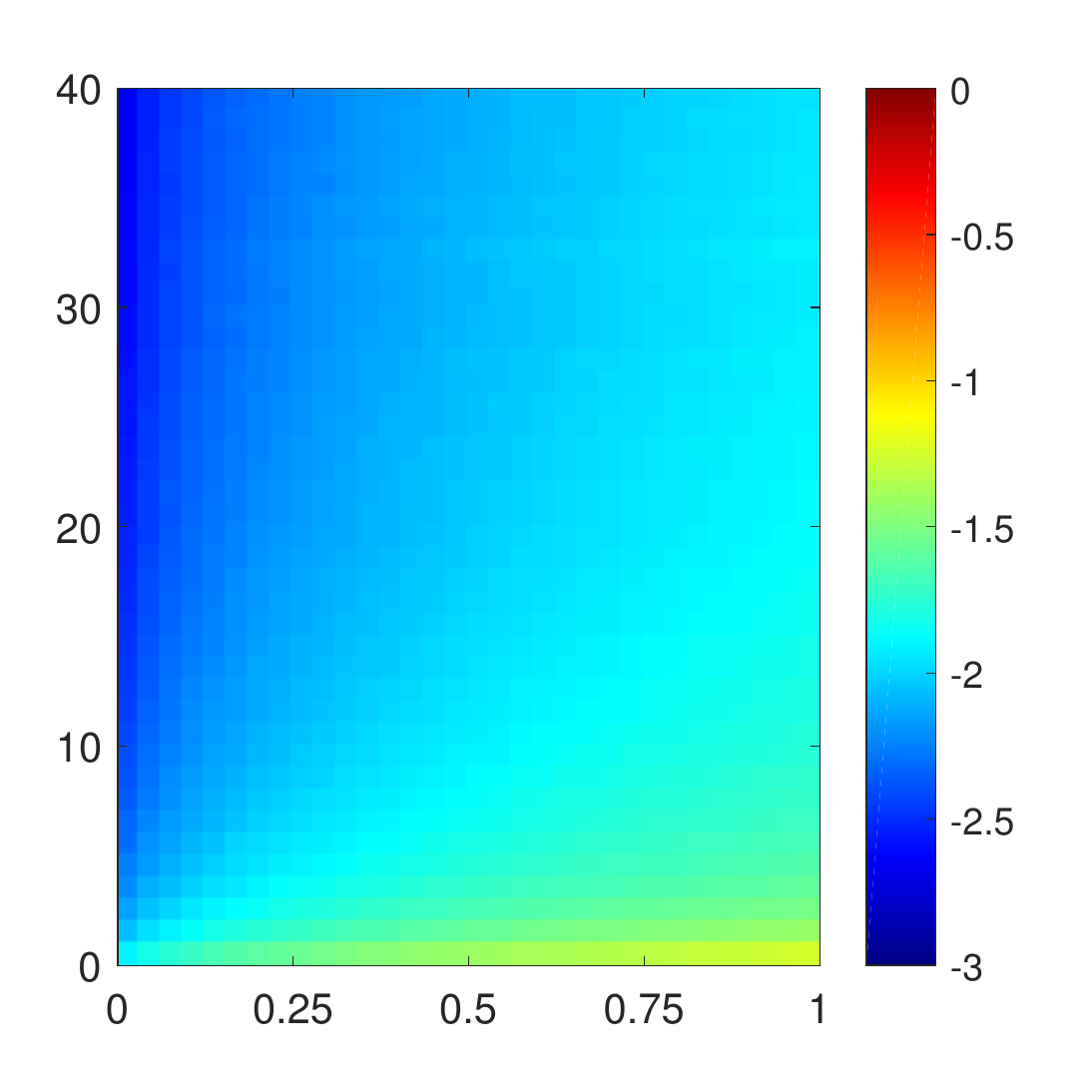}\label{fig:cgsp_v5pt_alteig}}
  \hfil
  \subfloat[]{\includegraphics[width=2.3in]{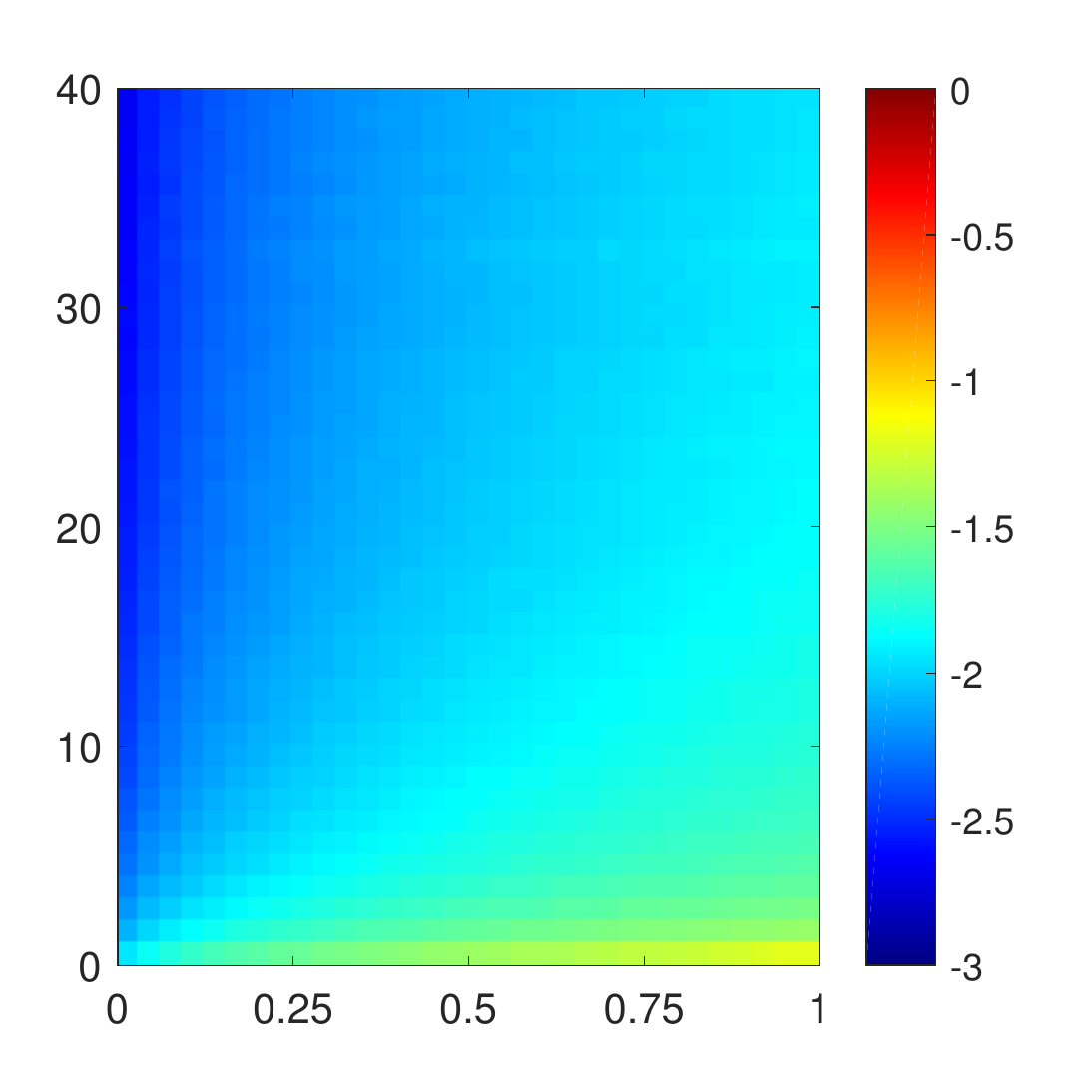}\label{fig:cgsp_v5pt_tpm}}
  \caption{Empirical phase transition in the 95th percentile of the log of the estimation error. $x$-axis: $D/K$. $y$-axis: $L/K$. $K = 256$, $M = 8$, SNR = 20 dB.
  \protect\subref{fig:cgsp_v5pt_cc} cross-convolution method \cite{xu1995least}. \protect\subref{fig:cgsp_v5pt_sccc} subspace-constrained cross-convolution method \cite{lee2017spectral}. \protect\subref{fig:cgsp_v5pt_alteig_nr} alternating eigenvectors method ($\hat{\sigma}_w^2 = \sigma_w^2$). \protect\subref{fig:cgsp_v5pt_tpm_nr} rank-1 truncated power method ($\hat{\sigma}_w^2 = \sigma_w^2$). \protect\subref{fig:cgsp_v5pt_alteig} alternating eigenvectors method ($\hat{\sigma}_w^2 = 0$). \protect\subref{fig:cgsp_v5pt_tpm} rank-1 truncated power method ($\hat{\sigma}_w^2 = 0$).
  }
  \label{fig:cgsp_pt}%
\end{figure}

To better visualize the overall trend, we performed a Monte Carlo simulation for the empirical phase transition, which is illustrated in Figure~\ref{fig:cgsp_pt} with a color coding that uses a logarithmic scale with blue denoting the smallest and red the largest error within the regime of $(D/K,L/K)$.  The error in the estimate by CC is large ($\geq 0.1$) regardless of $D/K$ for the entire regime (Figure~\protect\subref*{fig:cgsp_v5pt_cc}).  SCCC provides accurate estimates for small $D/K$ and for large enough $L/K$ (Figure~\protect\subref*{fig:cgsp_v5pt_sccc}).  On the other hand, it totally fails unless the dimension $D$ of subspace is not less than a certain threshold.  Finally, AltEig and RTPM show almost the same empirical phase transitions, which imply robust recovery for larger $D/K$ and for smaller $L/K$ (Figures~\protect\subref*{fig:cgsp_v5pt_alteig_nr} and \protect\subref*{fig:cgsp_v5pt_tpm_nr}).  Up to this point, we presented the performance of SCCC, AltEig, and RTPM for $\hat{\sigma}_w^2 = \sigma_w^2$, i.e. when the true noise variance is given.  Figures~\protect\subref*{fig:cgsp_v5pt_alteig} and \protect\subref*{fig:cgsp_v5pt_tpm} illustrate the empirical phase transitions for AltEig and RTPM when a conservative estimate of $\sigma_w^2$ given as $\hat{\sigma}_w^2 = 0$ is used instead.  These figures show that there is a nontrivial difference in the regime for accurate estimation depending on the quality of the estimate $\hat{\sigma}_w^2$.  This opens up an interesting question of how to show a guarantee for the noise variance estimation.  Nonetheless, even with $\hat{\sigma}_w^2 = 0$, both AltEig and RTPM show improvements in their empirical performances due to the extra structural constraint on impulse responses over CC and SCCC, which are (partially) blind to the bilinear prior model.

In our second experiment, we tested the algorithms on synthesized data with a realistic underwater acoustic channel model, where the impulse responses are approximated by a bilinear channel model.  In an ocean acoustic array sensing scenario, receivers of the vertical line array (VLA) with equal distance spacing listen to the same source near the ocean surface in a distance.  The channel impulse response (CIR) of each receiver on the VLA can be modeled using a pulse with a certain arrival-time and gain, which characterize the property of the propagating sound travels along the direct path.  Since sound traveling along the same path will experience almost the same media (speed of sound and loss), any environmental change or disturbance of the media will result in the same fluctuation of arrival-times of such sound pulses.  Therefore, arrival-times of receivers are linked across all channels while gains of each pulse are still independent.  A bilinear model of the channel responses are then introduced, where the basis $\mPhi$ defines a subspace for pulses that have linked arrival-times.  A detailed description on how to form the basis $\mPhi$ for a particular underwater environment can be found on the authors' incoming paper with N. Durofchalk and K. Sabra \cite{tian2017blind}.

\begin{figure}
  \centering
  \includegraphics[width=2.1in]{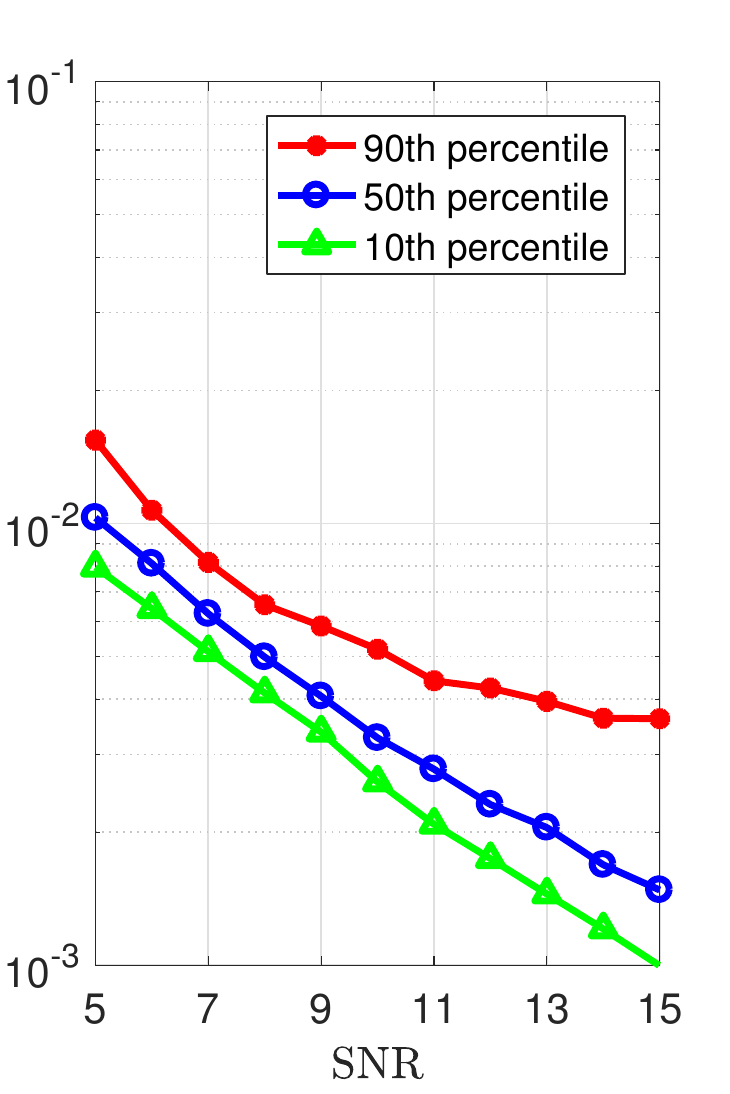}
  \caption{Order statistics for the log of estimation error in varying SNR for the underwater channel model.}
  \label{fig:underwater}%
\end{figure}

We performed Monte-Carlo simulation to demonstrate the robustness of our method on realistic acoustic channels which represent an at-sea experiment carried out at Santa Barbara Channel.  In the simulation, the common driving source signal, $\vx \in \mathbb{R}^L$, is Gaussian white noise filtered in an arbitrary bandwidth representative of shipping noise spectra (400--600 Hz) for $L = 2000$.  Each CIR is of length $K = 500$ and represents a Gaussian-windowed pulse in the band of (400--600 Hz).  The number of channels $M$ is $31$.  The basis $\mPhi \in \mathbb{R}^{K \times D}$ is of dimension $D = 8$.  
The number of trials in the Monte-Carlo simulation is $100$.  In this experiment, unlike the previous experiments with Gaussian bases, AltEig did not provide stable recovery. Therefore, we report the results only for RTPM.  As for the estimation of the noise variance, since the basis matrices are unitary, there is no need to subtract the expectation of the noise auto-correlation term.  Figure~\ref{fig:underwater} shows order statistics of the estimation error in a log scale.  The recovery is successful at a bit lower frequency.  The median of the estimation error approaches to the modeling error due to approximation with a bilinear model as we increase SNR.

\section{Analysis of Spectral Initialization}
\label{sec:proof:prop:init}

We prove Proposition~\ref{prop:init} in this section.
Recall that Algorithm~\ref{alg:spectral_init} computes an initial estimate $\widehat{\vb}$ of the true parameter vector $\vb$ as an eigenvector of $\mGamma \mGamma^* - \sum_{m=1}^M \sigma_w^2 L \mPhi_m^* \mPhi_m$ corresponding to the largest eigenvalue in magnitude.
Let us decompose the matrix $\mGamma$ in \eqref{eq:def_Gamma} as $\mGamma = \mGamma_{\mathrm{s}} + \mGamma_{\mathrm{n}}$, where $\mGamma_{\mathrm{s}}$ and $\mGamma_{\mathrm{n}}$ respectively correspond to the noise-free portion and noise portion of $\mGamma$.  In other words, $\mGamma_{\mathrm{s}}$ is obtained as we replace $\vy_m = \vh_m \circledast \vx + \vw_m$ in the expression of $\Gamma$ in \eqref{eq:def_Gamma2} by its first summand $\vh_m \circledast \vx$.  Similarly, $\mGamma_{\mathrm{n}}$ is obtained as we replace $\vy_m$ by $\vw_m$.  Then it follows that
\[
\mathbb{E}_w[\mGamma_{\mathrm{n}}\mGamma_{\mathrm{n}}^*] = \sum_{m=1}^M \sigma_w^2 L \mPhi_m^* \mPhi_m.
\]

By direct calculation, we obtain that the expectation of $\mGamma_{\mathrm{s}}$ is written as
\begin{equation}
\label{eq:exp_Gammas}
\mathbb{E}[\mGamma_{\mathrm{s}}] = \sum_{m=1}^M K a_m \vb \vx^\transpose = K \norm{\va}_1 \vb \vx^\transpose.
\end{equation}
Therefore,
\begin{equation}
\mathbb{E}[\mGamma_{\mathrm{s}}] \mathbb{E}[\mGamma_{\mathrm{s}}]^*
= K^2 \norm{\vx}_2^2 \norm{\va}_1^2 \vb \vb^*.
\label{trumat:proof:prop:init}
\end{equation}
It is straightforward to check that the rank-1 matrix $\mathbb{E}[\mGamma_{\mathrm{s}}] \mathbb{E}[\mGamma_{\mathrm{s}}]^*$ has an eigenvector, which is collinear with $\vb$.
Thus as we interpret $\mGamma \mGamma^* - \sum_{m=1}^M \sigma_w^2 L \mPhi_m^* \mPhi_m$ as a perturbed version of $\mathbb{E}[\mGamma_{\mathrm{s}}] \mathbb{E}[\mGamma_{\mathrm{s}}]^*$, the error in $\widehat{\vb}$ is upper bounded by the classical result in linear algebra known as the Davis-Kahan theorem \cite{davis1970rotation}. Among numerous variations of the original Davis-Kahan theorem available in the literature, we will use a consequence of a particular version \cite[Theorem~8.1.12]{golub2012matrix}, which is stated as the following lemma.

\begin{lemma}[A Special Case of the Davis-Kahan Theorem]
\label{lemma:daviskahan}
Let $\mM, \underline{\mM} \in \mathbb{C}^{n \times n}$ be symmetric matrices and $\lambda$ denote the largest eigenvalue of $\underline{\mM}$ in magnitude.  Suppose that $\lambda > 0$ and has multiplicity 1. Let $\mQ = [\vq_1, \mQ_2] \in \mathbb{C}^{n \times n}$ be a unitary matrix such that $\vq_1$ is an eigenvector of $\underline{\mM}$ corresponding to $\lambda$. Partition the matrix $\mQ^* \underline{\mM} \mQ$ as follows:
\[
\mQ^* \underline{\mM} \mQ = \begin{bmatrix} \lambda & \bm{0}_{1,n-1} \\ \bm{0}_{n-1,1} & \mD \end{bmatrix}.
\]
If
\begin{equation}
\label{eq:smallperturb}
\norm{\mD} + \norm{\mM-\underline{\mM}} \leq \frac{\lambda}{5},
\end{equation}
then the largest eigenvalue of $\mM$ in magnitude has multiplicity 1 and the corresponding eigenvector $\widetilde{\vq}$ satisfies
\begin{equation}
\label{eq:errbnd}
\sin\angle(\widetilde{\vq},\vq_1) \leq \frac{4\norm{(\mM - \underline{\mM}) \vq_1}_2}{\lambda}.
\end{equation}
\end{lemma}

\begin{rem}{\rm
In Lemma~\ref{lemma:daviskahan}, the rank-1 matrix $\lambda \vq_1 \vq_1^*$ is considered as the ground truth matrix.  Then $\mM - \underline{\mM} + \mQ_2 \mD \mQ_2^*$ corresponds to perturbation in $\mM$ relative to the ground truth matrix $\underline{\mM}$.  Also note that $\mQ_2 \mD \mQ_2^* \vq_1 = \vzero$.}
\end{rem}

In the remainder of this section, we obtain an upper bound on the error in $\widehat{\vb}$ by applying Lemma~\ref{lemma:daviskahan} to $\underline{\mM} = \mathbb{E}[\mGamma_{\mathrm{s}}] \mathbb{E}[\mGamma_{\mathrm{s}}]^*$, $\mM = \mGamma \mGamma^* - \sum_{m=1}^M \sigma_w^2 L \mPhi_m^* \mPhi_m$, $\vq_1 = \vb$, and $\widehat{\vq} = \widehat{\vb}$.

By \eqref{trumat:proof:prop:init}, we have $\mD = \vzero$ and $\lambda = K^2 \norm{\vx}_2^2 \norm{\va}_1^2 \norm{\vb}_2^2$.  Then we show that the spectral norm of the perturbation term, which is rewritten as
\begin{subequations}
\begin{align}
& \mGamma \mGamma^* - \mathbb{E}_w[\mGamma_{\mathrm{n}}\mGamma_{\mathrm{n}}^*] - \mathbb{E}[\mGamma_{\mathrm{s}}] \mathbb{E}[\mGamma_{\mathrm{s}}]^* \nonumber \\
&= \mGamma_{\mathrm{s}}\mGamma_{\mathrm{s}}^* - \mathbb{E}[\mGamma_{\mathrm{s}}] \mathbb{E}[\mGamma_{\mathrm{s}}]^*, \label{eq:decomp_init_Es} \\
&+ \mGamma_{\mathrm{s}}\mGamma_{\mathrm{n}}^*
+ \mGamma_{\mathrm{n}}\mGamma_{\mathrm{s}}^*, \label{eq:decomp_init_Ec} \\
&+ \mGamma_{\mathrm{n}}\mGamma_{\mathrm{n}}^* - \mathbb{E}_w[\mGamma_{\mathrm{n}}\mGamma_{\mathrm{n}}^*], \label{eq:decomp_init_En}
\end{align}
\end{subequations}
satisfies \eqref{eq:smallperturb}.  We will compute an upper estimate of the spectral norm of each summand, divided by $\lambda$, separately. Then we combine these estimates using the triangle inequality.

\vspace{.2in}
\noindent\textbf{Perturbation due to signal term:}
Note that the first summand $\mGamma_{\mathrm{s}}\mGamma_{\mathrm{s}}^* - \mathbb{E}[\mGamma_{\mathrm{s}}] \mathbb{E}[\mGamma_{\mathrm{s}}]^*$ in \eqref{eq:decomp_init_Es} has entries, which are fourth-order Gaussian random variables.
We decompose it using second-order random variables as
\begin{equation}
\label{eq:init_Es_decomp1}
\begin{aligned}
& \mGamma_{\mathrm{s}}\mGamma_{\mathrm{s}}^* - \mathbb{E}[\mGamma_{\mathrm{s}}] \mathbb{E}[\mGamma_{\mathrm{s}}]^* \\
&= (\mGamma_{\mathrm{s}} - \mathbb{E}[\mGamma_{\mathrm{s}}]) (\mGamma_{\mathrm{s}} - \mathbb{E}[\mGamma_{\mathrm{s}}])^*
+ \mathbb{E}[\mGamma_{\mathrm{s}}] (\mGamma_{\mathrm{s}} - \mathbb{E}[\mGamma_{\mathrm{s}}])^*
+ (\mGamma_{\mathrm{s}} - \mathbb{E}[\mGamma_{\mathrm{s}}]) \mathbb{E}[\mGamma_{\mathrm{s}}]^*.
\end{aligned}
\end{equation}

We have already computed $\mathbb{E}[\mGamma_{\mathrm{s}}]$ in \eqref{eq:exp_Gammas}. It remains to upper bound the spectral norm of $\mGamma_{\mathrm{s}} - \mathbb{E}[\mGamma_{\mathrm{s}}]$.
By the definitions of $\mGamma_{\mathrm{s}}$ and $\rho_x$, we obtain
\begin{align}
\norm{\mGamma_{\mathrm{s}} - \mathbb{E}[\mGamma_{\mathrm{s}}]}
&\leq
\Big\|
\sum_{m=1}^M a_m
(\mPhi_m^* \mS \mC_{\mS^* \mPhi_m \vb} - \mathbb{E}_{\phi}[\mPhi_m^* \mS \mC_{\mS^* \mPhi_m \vb}]) \breve{\mS}^* \breve{\mS} \mC_{\vx} \Big\| \nonumber \\
&\leq
\Big\|
\sum_{m=1}^M a_m (\mPhi_m^* \mS \mC_{\mS^* \mPhi_m \vb} - \mathbb{E}_{\phi}[\mPhi_m^* \mS \mC_{\mS^* \mPhi_m \vb}]) \breve{\mS}^* \Big\|
\norm{\breve{\mS} \mC_{\vx}} \nonumber \\
&\leq
\sqrt{\rho_x} \Big\|
\sum_{m=1}^M a_m (\mPhi_m^* \mS \mC_{\mS^* \mPhi_m \vb} - \mathbb{E}_{\phi}[\mPhi_m^* \mS \mC_{\mS^* \mPhi_m \vb}]) \breve{\mS}^* \Big\|,
\label{eq:init_Es_bnd1}
\end{align}
where $\breve{\mS} \in \mathbb{R}^{(2K-1) \times L}$ is defined by
\[
\breve{\mS} =
\begin{bmatrix}
\begin{bmatrix}
\bm{0}_{K-1,L-K+1} & \mId_{K-1}
\end{bmatrix} \\
\begin{bmatrix}
\mId_{K} & \bm{0}_{K,L-K}
\end{bmatrix}
\end{bmatrix}.
\]

The right-hand side of \eqref{eq:init_Es_bnd1} except the constant factor $\sqrt{\rho_x}$ is upper bounded by the following lemma, which is proved in Appendix~\ref{sec:proof:lemma:init_Gamma_Es}.

\begin{lemma}
\label{lemma:init_Gamma_Es}
Suppose that (A2) holds.
For any $\beta \in \mathbb{N}$, there is a constant $C(\beta)$ that depends only on $\beta$ such that
\begin{equation}
\label{eq:obj_init_Zs}
\begin{aligned}
\Big\|
\sum_{m=1}^M a_m (\mPhi_m^* \mS \mC_{\mS^* \mPhi_m \vb} - \mathbb{E}_{\phi} \mPhi_m^* \mS \mC_{\mS^* \mPhi_m \vb}) \breve{\mS}^* \Big\|
\leq C(\beta) K \sqrt{M} \norm{\va}_\infty \norm{\vb}_2 \log^\alpha (MKL)
\end{aligned}
\end{equation}
holds with probability $1-K^{-\beta}$.
\end{lemma}

By applying \eqref{eq:exp_Gammas}, \eqref{eq:init_Es_bnd1}, Lemma~\ref{lemma:init_Gamma_Es} together with the fact $\sqrt{\rho_x} \leq C_0 \norm{\vx}_2$ to \eqref{eq:init_Es_decomp1}, we obtain that
\begin{equation}
\label{eq:init_Es_res}
\frac{\norm{\mGamma_{\mathrm{s}}\mGamma_{\mathrm{s}}^* - \mathbb{E}[\mGamma_{\mathrm{s}}] \mathbb{E}[\mGamma_{\mathrm{s}}]^*}}{\lambda}
\leq \frac{C(\beta) \sqrt{M} \norm{\va}_\infty \log^\alpha (MKL)}{\norm{\va}_1}
\leq \frac{C(\beta) \mu \log^\alpha (MKL)}{\nu \sqrt{M}}
\end{equation}
holds with probability $1-K^{-\beta}$.

\vspace{.2in}
\noindent\textbf{Perturbation due to signal-noise cross term:}
Next we consider the second term in \eqref{eq:decomp_init_Ec}.
By the triangle inequality, we have
\[
\norm{\mGamma_{\mathrm{s}}\mGamma_{\mathrm{n}}^* + \mGamma_{\mathrm{n}}\mGamma_{\mathrm{s}}^*} \leq
\norm{\mGamma_{\mathrm{s}}\mGamma_{\mathrm{n}}^*} + \norm{\mGamma_{\mathrm{n}}\mGamma_{\mathrm{s}}^*}
\leq 2 \norm{\mGamma_{\mathrm{s}} \mGamma_{\mathrm{n}}^*}.
\]
Therefore, it suffices to upper estimate $\norm{\mGamma_{\mathrm{s}} \mGamma_{\mathrm{n}}^*}$.
To this end, we decompose $\mGamma_{\mathrm{s}} \mGamma_{\mathrm{n}}^*$ as
\begin{equation}
\label{eq:init_Ec_decomp1}
\mGamma_{\mathrm{s}} \mGamma_{\mathrm{n}}^*
= (\mGamma_{\mathrm{s}} - \mathbb{E}[\mGamma_{\mathrm{s}}]) \mGamma_{\mathrm{n}}^*
+ \mathbb{E}[\mGamma_{\mathrm{s}}] \mGamma_{\mathrm{n}}^*.
\end{equation}

Note that the first summand in the right-hand-side of \eqref{eq:init_Ec_decomp1} is written as \begin{equation}
\label{eq:init_Ec_decomp2}
(\mGamma_{\mathrm{s}} - \mathbb{E}[\mGamma_{\mathrm{s}}]) \mGamma_{\mathrm{n}}^*
=
\Big(
\sum_{m=1}^M a_m \mPhi_m^* \mS \mC_{\mS^* \mPhi_m \vb} \breve{\mS}^* - \mathbb{E}_{\phi}[a_m \mPhi_m^* \mS \mC_{\mS^* \mPhi_m \vb} \breve{\mS}^*]
\Big)
\Big(
\sum_{m'=1}^M
\breve{\mS} \mC_{\vx} \mC_{\vw_{m'}}^* \mS^* \mPhi_{m'}
\Big),
\end{equation}
where the first and second factors of the right-hand-side of \eqref{eq:init_Ec_decomp2}
are upper bounded in the spectral norm respectively by Lemma~\ref{lemma:init_Gamma_Es} and by the following lemma. (See Appendix~\ref{sec:proof:lemma:init_Gamma_Ec} for the proof.)

\begin{lemma}
\label{lemma:init_Gamma_Ec}
Suppose that (A1) and (A2) hold.
For any $\beta \in \mathbb{N}$, there is a constant $C(\beta)$ that depends only on $\beta$ such that
\begin{equation}
\label{eq:init_Gamma_Ec}
\Big\|
\sum_{m=1}^M
\breve{\mS} \mC_{\vx} \mC_{\vw_m}^* \mS^* \mPhi_m
\Big\|
\leq C(\beta) \rho_{x,w} \sqrt{MK} \log^\alpha(MKL)
\end{equation}
holds with probability $1-K^{-\beta}$.
\end{lemma}

By applying Lemmas~\ref{lemma:init_Gamma_Es} and \eqref{lemma:init_Gamma_Ec} to \eqref{eq:init_Ec_decomp2}, we obtain that
\begin{equation}
\label{eq:init_Ec_res1}
\norm{(\mGamma_{\mathrm{s}} - \mathbb{E}[\mGamma_{\mathrm{s}}]) \mGamma_{\mathrm{n}}^*}
\leq C(\beta) \rho_{x,w} M K^{3/2} \norm{\va}_\infty \norm{\vb}_2 \log^\alpha (MKL)
\end{equation}
holds with probability $1-K^{-\beta}$.

Next, the second summand in the right-hand-side of \eqref{eq:init_Ec_decomp1} is written as
\begin{equation}
\label{eq:init_Ec_decomp3}
\mathbb{E}[\mGamma_{\mathrm{s}}] \mGamma_{\mathrm{n}}^*
= K \norm{\va}_1 \vb
\Big(
\sum_{m'=1}^M
\ve_1^* \mC_{\vx} \mC_{\vw_{m'}}^* \mS^* \mPhi_{m'}
\Big),
\end{equation}
whose spectral norm is upper bounded by using the following lemma.

\begin{lemma}
\label{lemma:init_Gamma_Ec:fixz}
Suppose that (A1) and (A2) hold.
For any $\beta \in \mathbb{N}$, there is a constant $C(\beta)$ that depends only on $\beta$ such that
\[
\Big\|
\sum_{m'=1}^M
\ve_1^* \mC_{\vx} \mC_{\vw_{m'}}^* \mS^* \mPhi_{m'}
\Big\|
\leq C(\beta) \rho_{x,w} \sqrt{MD} \log^\alpha(MKL)
\]
holds with probability $1-K^{-\beta}$.
\end{lemma}
The proof of Lemma~\ref{lemma:init_Gamma_Ec:fixz} is very similar to that of Lemma~\ref{lemma:init_Gamma_Ec}.
The proof of Lemma~\ref{lemma:init_Gamma_Ec} involves the following optimization formulation:
\[
\max_{\vz \in B_2^{2K-1}} \max_{\vq \in B_2^D}
\sum_{m=1}^M
\vz^* \breve{\mS} \mC_{\vx} \mC_{\vw_m}^* \mS^* \mPhi_m \vq.
\]
Instead of maximizing over $\vz \in B_2^{2K-1}$, we fix $\vz$ to $\breve{\mS} \ve_1$.
Equivalently, we replace the unit ball $B_2^{2K-1}$ by the singleton set $\{\breve{\mS} \ve_1\}$. This replacement simply removes the entropy integral corresponding to $B_2^{2K-1}$. Except this point, the proofs for the two lemmas are identical. Thus we omit further details.

Applying Lemma~\ref{lemma:init_Gamma_Ec:fixz} to \eqref{eq:init_Ec_decomp3} implies that
\begin{equation}
\label{eq:init_Ec_res2}
\norm{\mathbb{E}[\mGamma_{\mathrm{s}}] \mGamma_{\mathrm{n}}^*}
\leq C(\beta) \rho_{x,w} \sqrt{M} K \sqrt{D} \norm{\va}_1 \norm{\vb}_2 \log^\alpha (MKL)
\end{equation}
holds with probability $1-K^{-\beta}$.

By combining \eqref{eq:init_Ec_res1} and \eqref{eq:init_Ec_res2}, after plugging in the definitions of $\eta$, $\mu$, and $\nu$, we obtain that
\begin{equation}
\label{eq:init_Ec_res}
\frac{\norm{\mGamma_{\mathrm{s}}\mGamma_{\mathrm{n}}^* + \mGamma_{\mathrm{n}}\mGamma_{\mathrm{s}}^*}}{\lambda}
\leq
\frac{C(\beta) \log^\alpha (MKL)}{\sqrt{\eta}} 
\cdot
\frac{\rho_{x,w}}{\norm{\vx}_2 \sigma_w \sqrt{L}}
\cdot
\Big(
\frac{\mu}{\nu^2 M} + \frac{\sqrt{D}}{\nu \sqrt{MK}}
\Big)
\end{equation}
holds with probability $1-2K^{-\beta}$.

\vspace{.2in}
\noindent\textbf{Perturbation due to noise term:}
Finally, we derive an upper bound on the spectral norm of the last term in \eqref{eq:decomp_init_En} using the following lemma, which is proved in Appendix~\ref{sec:proof:lemma:init_Gamma_En}.

\begin{lemma}
\label{lemma:init_Gamma_En}
Suppose that (A2) holds.
For any $\beta \in \mathbb{N}$, there is a constant $C(\beta)$ that depends only on $\beta$ such that
\[
\norm{\mGamma_{\mathrm{n}}\mGamma_{\mathrm{n}}^* - \mathbb{E}_w[\mGamma_{\mathrm{n}}\mGamma_{\mathrm{n}}^*]}
\leq C(\beta) \rho_w M^{3/2} \sqrt{KD} \log^\alpha (MKL)
\]
holds with probability $1-K^{-\beta}$.
\end{lemma}

We also a tail bound on $\rho_w$ given by the following lemma from \cite{lee2017spectral}.
\begin{lemma}[{\cite[Lemma~5.9]{lee2017spectral}}]
\label{lemma:rho_w}
Suppose that (A2) holds.
For any $\beta \in \mathbb{N}$, there is a constant $C(\beta)$ that depends only on $\beta$ such that
\begin{equation}
\label{eq:ubrhow}
\rho_w \leq C(\beta) \sigma_w^2 \sqrt{KL} \log^\alpha (MKL)
\end{equation}
holds with probability $1 - K^{-\beta}$.
\end{lemma}

By Lemma~\ref{lemma:rho_w} and \eqref{eq:etasimp}, the corresponding relative perturbation is upper bounded by
\begin{equation}
\label{eq:init_En_res}
\frac{\norm{\mGamma_{\mathrm{n}}\mGamma_{\mathrm{n}}^* - \mathbb{E}_w[\mGamma_{\mathrm{n}}\mGamma_{\mathrm{n}}^*]}
}{K^2 \norm{\vx}_2^2 \norm{\va}_1^2 \norm{\vb}_2^2}
\leq
\frac{C(\beta) \log^\alpha (MKL)}{\eta} \cdot \frac{\sqrt{D}}{\nu^2 \sqrt{ML}}
\end{equation}
with probability $1-K^{-\beta}$.

Then it follows from \eqref{eq:init_Es_res}, \eqref{eq:init_Ec_res}, and \eqref{eq:init_En_res} that the condition in \eqref{eq:smallperturb} is satisfied by the assumptions in \eqref{eq:init_condL} and \eqref{eq:init_condM}.
Therefore, Lemma~\ref{lemma:daviskahan} provides the upper bound on the estimation error in \eqref{eq:init_errbnd}, which is obtained by plugging \eqref{eq:init_Es_res}, \eqref{eq:init_Ec_res}, and \eqref{eq:init_En_res} to \eqref{eq:errbnd}.  This completes the proof.

\section{Convergence of Alternating Eigen Method}
\label{sec:proof:prop:alteig}

Algorithm~\ref{alg:alteig} iteratively updates the estimates of $\va$, $\vb$ from a function of the matrix $\mA = \mPhi^* (\mY^* \mY - \sigma_w^2 (M-1)L \mId_{MK}) \mPhi$ and previous estimates. Propositions~\ref{prop:update_a} and \ref{prop:update_b} show the convergence of the iterations in Algorithm~\ref{alg:alteig} that alternately update the estimates $\widehat{\va}$ and $\widehat{\vb}$ under the randomness assumptions in (A1) and (A2).  Similarly to the analysis of the spectral initialization in Section~\ref{sec:proof:prop:init}, we prove Propositions~\ref{prop:update_a} and \ref{prop:update_b} by using the Davis-Kahan Theorem in Lemma~\ref{lemma:daviskahan}.
To this end, we first compute tail estimates of norms of the deviation of the random matrix $\mA$ from its expectation $\underline{\mA} = \mathbb{E}[\mA]$ below.

\subsection{Tail estimates of deviations}

Algorithm~\ref{alg:alteig} updates the estimates $\widehat{\va}$ as the least dominant eigenvector of $(\mId_M \otimes \widehat{\vb}^*) \mA (\mId_M \otimes \widehat{\vb})$ where $\widehat{\vb}$ denotes the estimate in the previous step.  The other estimate $\widehat{\vb}$ is updated similarly from $(\widehat{\va}^* \otimes \mId_D) \mA (\widehat{\va} \otimes \mId_D)$.  The matrices involved in these updates are restricted version of $\mA$ with separable projection operators.

In order to get a tightened perturbation bound for the estimates, we introduce a new matrix norm with this separability structure.  To define the new norm, we need operators that rearrange an $M$-by-$D$ matrix into a column vector of length $MD$ and vice versa.  For $\mV = [\vv_1,\dots,\vv_M] \in \mathbb{C}^{M \times D}$, define
\[
\mathrm{vec}(\mV) = [\vv_1^\transpose,\dots,\vv_M^\transpose]^\transpose.
\]
Let $\mathrm{mat}(\cdot)$ denote the inverse of $\mathrm{vec}(\cdot)$ so that
\[
\mathrm{mat}(\mathrm{vec}(\mV)) = \mV, \quad \forall \mV \in \mathbb{C}^{M \times D}
\]
and
\[
\mathrm{vec}(\mathrm{mat}(\vv)) = \vv, \quad \forall \vv \in \mathbb{C}^{MD}.
\]
With these vectorization and matricization operators, we define the matricized $S_p$-norm of $\vv \in \mathbb{C}^{MD}$ by
\[
\tnorm{\vv}_{S_p} = \norm{\mathrm{mat}(\vv)}_{S_p}.
\]
Then the matricized operator norm of $\mM \in \mathbb{C}^{MD \times MD}$ is defined by
\[
\tnorm{\mM}_{S_p \to S_q} := \max_{\tnorm{\vv}_{S_p} \leq 1} \tnorm{\mM \vv}_{S_q}.
\]

For $p=1$ and $q=\infty$, by the Courant-Fischer minimax principle, the matricized operator norm is written as a variational form
\[
\begin{array}{rcl}
\tnorm{\mM}_{S_1 \to S_\infty} = & \displaystyle \max_{\mUpsilon, \mUpsilon' \in \mathbb{C}^{M \times D}}
& |\langle \mathrm{vec}(\mUpsilon') , \mM \mathrm{vec}(\mUpsilon) \rangle| \\
& \mathrm{subject~to} & \norm{\mUpsilon}_{S_1} \leq 1,~ \norm{\mUpsilon'}_{S_1} \leq 1.
\end{array}
\]
Since the unit ball with respect to the $S_1$-norm is given as the convex hull of all unit-$S_2$-norm matrices of rank-1, $\tnorm{\mM}_{S_1 \to S_\infty}$ is equivalently rewritten as
\begin{equation}
\label{eq:varforn_S1S8norm}
\begin{array}{rcl}
\tnorm{\mM}_{S_1 \to S_\infty} = & \displaystyle \max_{\mUpsilon, \mUpsilon' \in \mathbb{C}^{M \times D}}
& |\langle \mathrm{vec}(\mUpsilon') , \mM \mathrm{vec}(\mUpsilon) \rangle| \\
& \mathrm{subject~to} & \norm{\mUpsilon}_{S_2} \leq 1,~ \norm{\mUpsilon'}_{S_2} \leq 1 \\
& & \mathrm{rank}(\mUpsilon) = \mathrm{rank}(\mUpsilon') = 1.
\end{array}
\end{equation}
Therefore, by dropping the rank-1 constraints in \eqref{eq:varforn_S1S8norm}, we obtain
\begin{equation}
\label{eq:matSopnormlessthanSN}
\tnorm{\mM}_{S_1 \to S_\infty} \leq \norm{\mM}, \quad \forall \mM \in \mathbb{C}^{MD \times MD}.
\end{equation}

The following lemma provides a tail estimate of $\tnorm{\mE}_{S_1 \to S_\infty}$ divided by $K^2 \norm{\vx}_2^2 \norm{\vu}_2^2$, which amounts to the spectral gap between the two smallest eigenvalues of $\underline{\mA}$.  Compared to the analogous tail estimate for its spectral norm, derived in \cite[Section~3.2]{lee2017spectral}, the tail estimate for $\tnorm{\mE}_{S_1 \to S_\infty}$ is smaller in order.  This is the reason why we obtain a better sample complexity by introducing the extra rank-1 structure to the prior model on impulse responses.

\begin{lemma}
\label{lemma:bnd_S1S8norm_E}
Let $\mE = \mA - \underline{\mA}$. For any $\beta \in \mathbb{N}$, there exist a numerical constant $C$ and a constant $C(\beta)$ that depends only on $\beta$ such that
\begin{equation}
\label{eq:bnd_S1S8normE}
\begin{aligned}
\frac{\tnorm{\mE}_{S_1 \to S_\infty}}{K^2 \norm{\vx}_2^2 \norm{\vu}_2^2}
&\leq C(\beta) \log^\alpha(MKL) \Big[ \Big(\frac{1}{\sqrt{M}} + \sqrt{\frac{D}{K}} \, \Big) \mu^2 \\
& + \frac{\rho_{x,w}}{\sqrt{\eta KL}\sigma_w \norm{\vx}_2}
\Big(
\mu\Big(\frac{\sqrt{K}}{M} + \sqrt{\frac{D}{M}} + \sqrt{\frac{D}{K}} \, \Big) + 1
\Big)
+ \frac{\sqrt{D}}{\eta \sqrt{L}}
\Big]
\end{aligned}
\end{equation}
holds with probability $1-CK^{-\beta}$.
\end{lemma}

\begin{proof}[Proof of Lemma~\ref{lemma:bnd_S1S8norm_E}]
The derivation of \eqref{eq:bnd_S1S8normE} is similar to that for the analogous tail estimate for $\norm{\mE}$ in \cite[Section~3.2]{lee2017spectral}.  We use the same decomposition of $\mE$, which is briefly summarized below.

We decompose $\mY$ as $\mY = \mY_{\mathrm{s}} + \mY_{\mathrm{n}}$, where the noise-free portion $\mY_{\mathrm{s}}$ (resp. the noise portion $\mY_{\mathrm{n}}$) is obtained as we replace $\vy_m=\vh_m \conv \vx +\vw_m$ in $\mY$ by its first summand $\vh_m \conv \vx$ (resp. by its second summand $\vw_m$) for all $m=1,\dots,M$.  Then $\mE$ is written as the sum of three matrices whose entries are given as polynomials of subgaussian random variables of different order as follows.
\begin{align}
\mE &=
\mPhi^* \mY_{\mathrm{s}}^* \mY_{\mathrm{s}} \mPhi - \mathbb{E}[\mPhi^* \mY_{\mathrm{s}}^* \mY_{\mathrm{s}} \mPhi] \label{eq:Es_Edecomp} \\
&+ \mPhi^* \mY_{\mathrm{s}}^* \mY_{\mathrm{n}} \mPhi
+ \mPhi^* \mY_{\mathrm{n}}^* \mY_{\mathrm{s}} \mPhi \label{eq:Ec_Edecomp} \\
&+ \mPhi^* (\mY_{\mathrm{n}}^* \mY_{\mathrm{n}}- \sigma_w^2 (M-1)L \mId_{MK}) \mPhi. \label{eq:En_Edecomp}
\end{align}

We first compute tail estimates of the components; the tail estimate in \eqref{eq:bnd_S1S8normE} is then obtained by combining these results via the triangle inequality.

For the first summand in \eqref{eq:Es_Edecomp} and the last summand in \eqref{eq:En_Edecomp}, we were not able to reduce their tail estimates in order compared to the spectral norms.  Thus we use their tail estimates on the spectral norms derived in \cite[Section~3.2]{lee2017spectral}, which are also valid tail estimates by \eqref{eq:matSopnormlessthanSN}.  For the completeness, we provide the corresponding lemmas below.

\begin{lemma}[{\hspace{1sp}\cite[Lemma~3.5]{lee2017spectral}}]
\label{lemma:Es}
Suppose that (A1) holds.
For any $\beta \in \mathbb{N}$, there exist a numerical constant $\alpha \in \mathbb{N}$ and a constant $C(\beta)$ that depends only on $\beta$ such that
\begin{equation}
\label{eq:bnd_snEs}
\frac{\norm{\mPhi^* \mY_{\mathrm{s}}^* \mY_{\mathrm{s}} \mPhi - \mathbb{E}[\mPhi^* \mY_{\mathrm{s}}^* \mY_{\mathrm{s}} \mPhi]}}{K^2 \norm{\vx}_2^2 \norm{\vu}_2^2}
\leq
C(\beta) \log^\alpha(MKL) \Big(\sqrt{\frac{1}{M}} + \sqrt{\frac{D}{K}} \, \Big) \mu^2
\end{equation}
holds with probability $1-CK^{-\beta}$.
\end{lemma}

\begin{lemma}[{\hspace{1sp}\cite[Lemma~3.7]{lee2017spectral}}]
\label{lemma:En}
Suppose that (A1) holds.
For any $\beta \in \mathbb{N}$, there is a constant $C(\beta)$ that depends only on $\beta$ such that
\begin{equation}
\label{eq:bnd_snEn}
\frac{\norm{\mPhi^* (\mY_{\mathrm{n}}^* \mY_{\mathrm{n}}- \sigma_w^2 (M-1)L \mId_{MK}) \mPhi}}{K^2 \norm{\vx}_2^2 \norm{\vu}_2^2}
\leq
\frac{C(\beta) \log^\alpha (MKL)}{\eta} \cdot \sqrt{\frac{D}{L}}
\end{equation}
with probability $1-CK^{-\beta}$.
\end{lemma}

For the second and third terms in \eqref{eq:Ec_Edecomp}, we use their tail estimates given in the following lemma, the proof of which is provided in Appendix~\ref{sec:proof:lemma:Ec}.

\begin{lemma}
\label{lemma:Ec}
Suppose that (A1) holds.
For any $\beta \in \mathbb{N}$, there exists a constant $C(\beta)$ that depends only on $\beta$ such that, conditional on the noise vector $\vw$,
\begin{equation}
\label{eq:bnd_S1S8normEc}
\frac{\tnorm{\mPhi^* \mY_{\mathrm{s}}^* \mY_{\mathrm{n}} \mPhi}_{S_1 \to S_\infty}}{K^2 \norm{\vx}_2^2 \norm{\vu}_2^2}
\leq
\frac{C(\beta) \rho_{x,w}}{\sqrt{\eta KL}\sigma_w \norm{\vx}_2}
\Big(
\mu\Big(\frac{\sqrt{K}}{M} + \sqrt{\frac{D}{M}} + \sqrt{\frac{D}{K}} \, \Big) + 1
\Big)
\end{equation}
holds with probability $1-CK^{-\beta}$.
\end{lemma}

Finally, the tail estimate in \eqref{eq:bnd_S1S8normE} is obtained by combining \eqref{eq:bnd_snEs}, \eqref{eq:bnd_snEs}, and \eqref{eq:bnd_S1S8normEc} via the triangle inequality.  This completes the proof.
\end{proof}

We will also make use of a tail estimate of $\tnorm{\mE \vu}_{S_\infty} / \norm{\vu}_2$, again normalized by factor $K^2 \norm{\vx}_2^2 \norm{\vu}_2^2$.  The following lemma, which provides a relevant tail estimate, is a direct consequence of Lemma~\ref{lemma:Ec} and \cite[Lemma~3.8]{lee2017spectral}.

\begin{lemma}
\label{lemma:bnd_2norm_Eu}
Let $\mE = \mA - \underline{\mA}$. For any $\beta \in \mathbb{N}$, there exist a numerical constant $C$ and a constant $C(\beta)$ that depends only on $\beta$ such that
\begin{equation}
\label{eq:bnd_2normEu}
\begin{aligned}
\frac{\tnorm{\mE \vu}_{S_\infty}}{K^2 \norm{\vx}_2^2 \norm{\vu}_2^3}
&\leq C(\beta) \log^\alpha(MKL) \Big[ \frac{\rho_{x,w}}{\sqrt{\eta KL}\sigma_w \norm{\vx}_2}
\Big(
\mu\Big(\frac{\sqrt{K}}{M} + \sqrt{\frac{D}{M}} + \sqrt{\frac{D}{K}} \, \Big) + 1
\Big)
+ \frac{\sqrt{D}}{\eta \sqrt{ML}}
\Big]
\end{aligned}
\end{equation}
holds with probability $1-CK^{-\beta}$.
\end{lemma}

\subsection{Proof of Proposition~\ref{prop:update_a}}
\label{subsec:proof_prop_update_a}

To simplify notations, let $\theta = \angle(\vb,\widehat{\vb})$ denote the principal angle between the two subspaces spanned respectively by $\vb$ and $\widehat{\vb}$, i.e., $\theta \in [0, \pi/2]$ satisfies
\[
\sin\theta = \norm{\mP_{\vb^\perp} \widehat{\vb}}_2,
\quad \cos\theta = \norm{\mP_{\vb} \widehat{\vb}}_2,
\]
where $\mP_{\vb}$ denotes the orthogonal projection onto the span of $\vb$.
The assumption in \eqref{eq:starterrb} implies $\theta \leq \pi/4$.

Recall that Algorithm~\ref{alg:alteig} updates $\widehat{\va}$ from a given estimate $\widehat{\vb}$ in the previous step as the eigenvector of the matrix $(\mId_M \otimes \widehat{\vb}^*) \mA (\mId_M \otimes \widehat{\vb})$ corresponding to the smallest eigenvalue.  Without loss of generality, we may assume that $\norm{\widehat{\vb}}_2 = 1$.

By direct calculation, we obtain that $\underline{\mA} = \mathbb{E}[\mA]$ is rewritten as
\begin{equation}
\label{eq:mA:lemma:update_a}
\underline{\mA} = K^2 \norm{\vx}_2^2 \norm{\vb}_2^2 (\norm{\va}_2^2 \mId_M - \mathrm{diag}(|\va|^2)) \otimes \mP_{\vb^\perp}
+ K^2 \norm{\vx}_2^2 \norm{\vb}_2^2 (\norm{\va}_2^2 \mId_M - \va \va^*) \otimes \mP_{\vb}.
\end{equation}
Then
\begin{equation}
\label{eq:restrictedmA:lemma:update_a}
(\mId_M \otimes \widehat{\vb}^*) \underline{\mA} (\mId_M \otimes \widehat{\vb})
= K^2 \norm{\vx}_2^2 \norm{\vb}_2^2 (\norm{\va}_2^2 \mId_M - \cos^2\theta \, \va \va^*)
- K^2 \norm{\vx}_2^2 \norm{\vb}_2^2 \sin^2\theta \, \mathrm{diag}(|\va|^2).
\end{equation}
Here $|\va|^2$ denotes the vector whose $k$th entry is the squared magnitude of the $k$th entry of $\va$ and $\mathrm{diag}(|\va|^2)$ is a diagonal matrix whose diagonal entries are given by $|\va|^2$.

We verify that the matrix $\norm{\va}_2^2 \mId_M - \cos^2\theta \, \va \va^*$ is positive definite and its smallest eigenvalue, which has multiplicity 1, is smaller than the next smallest eigenvalue by $\norm{\va}_2^2 \cos^2\theta$.  Furthermore, $\va$ is collinear with the eigenvector corresponding to the smallest eigenvalue.

Let us consider the following matrix:
\begin{align*}
& K^2 \norm{\vx}_2^2 \norm{\vb}_2^2 \norm{\va}_2^2 \mId_M - (\mId_M \otimes \widehat{\vb}^*) \mA (\mId_M \otimes \widehat{\vb}) \\
& = K^2 \norm{\vx}_2^2 \norm{\vb}_2^2 \cos^2\theta \, \va \va^* \\
& + K^2 \norm{\vx}_2^2 \norm{\vb}_2^2 \sin^2\theta \, \mathrm{diag}(|\va|^2)
- (\mId_M \otimes \widehat{\vb}^*) \mE (\mId_M \otimes \widehat{\vb}),
\end{align*}
which we considered as a perturbed version of $K^2 \norm{\vx}_2^2 \norm{\vb}_2^2 \cos^2\theta \, \va \va^*$.  Then the perturbation, that is the difference of the two matrices, satisfies
\begin{align}
& \Bignorm{
K^2 \norm{\vx}_2^2 \norm{\vb}_2^2 \norm{\va}_2^2 \mId_M - (\mId_M \otimes \widehat{\vb}^*) \mA (\mId_M \otimes \widehat{\vb})
- K^2 \norm{\vx}_2^2 \norm{\vb}_2^2 \cos^2\theta \, \va \va^*} \nonumber \\
& \leq \Bignorm{K^2 \norm{\vx}_2^2 \norm{\vb}_2^2 \sin^2\theta \, \mathrm{diag}(|\va|^2)}
+ \norm{(\mId_M \otimes \widehat{\vb}^*) \mE (\mId_M \otimes \widehat{\vb})} \nonumber \\
&\leq K^2 \norm{\vx}_2^2 \norm{\vb}_2^2 \norm{\va}_\infty^2 \sin^2\theta
+ \tnorm{\mE}_{S_1 \to S_\infty}.
\label{eq:ub_perturb:lemma:update_a}
\end{align}

For sufficiently large $C_1(\beta)$, the conditions in \eqref{eq:starterrb},
\eqref{eq:updatea_condK_prop}, \eqref{eq:updatea_condM_prop}, \eqref{eq:updatea_condL_prop} imply
\[
K^2 \norm{\vx}_2^2 \norm{\vb}_2^2 \norm{\va}_2^2 \cos^2\theta
> 2 (K^2 \norm{\vx}_2^2 \norm{\vb}_2^2 \norm{\va}_\infty^2 \sin^2\theta
+ \tnorm{\mE}_{S_1 \to S_\infty}).
\]
Therefore, $\widehat{\va}$ is a unique dominant eigenvector of $K^2 \norm{\vx}_2^2 \norm{\vb}_2^2 \norm{\va}_2^2 \mId_M - (\mId_M \otimes \widehat{\vb}^*) \mA (\mId_M \otimes \widehat{\vb})$.

Next we apply Lemma~\ref{lemma:daviskahan} for
\begin{align*}
\underline{\mM} &= K^2 \norm{\vx}_2^2 \norm{\vb}_2^2 \cos^2\theta \, \va \va^*, \\
\mM &= K^2 \norm{\vx}_2^2 \norm{\vb}_2^2 \norm{\va}_2^2 \mId_M - (\mId_M \otimes \widehat{\vb}^*) \mA (\mId_M \otimes \widehat{\vb}), \\
\vq_1 &= \frac{\va}{\norm{\va}_2}, \qquad \widetilde{\vq} = \widehat{\va}.
\end{align*}
Then $\lambda$ and $\mD$ in Lemma~\ref{lemma:daviskahan} are given as $\lambda = K^2 \norm{\vx}_2^2 \norm{\vb}_2^2 \norm{\va}_2^2 \cos^2\theta$ and $\mD = \vzero$.

By \eqref{eq:ub_perturb:lemma:update_a}, we have
\begin{align*}
\frac{\norm{\mM - \underline{\mM}}}{\lambda}
& \leq
\frac{\norm{\va}_\infty^2 \sin^2\theta}{\norm{\va}_2^2 \cos^2\theta}
+ \frac{\tnorm{\mE}_{S_1 \to S_\infty}}{K^2 \norm{\vx}_2^2 \norm{\vb}_2^2 \norm{\va}_2^2 \cos^2\theta} \\
& \leq
\frac{\mu^2}{M}
+ \frac{2 \tnorm{\mE}_{S_1 \to S_\infty}}{K^2 \norm{\vx}_2^2 \norm{\vb}_2^2 \norm{\va}_2^2},
\end{align*}
where the last step follows from \eqref{eq:starterrb}.
Therefore, for sufficiently large $C_1(\beta)$, the conditions in \eqref{eq:updatea_condK_prop}, \eqref{eq:updatea_condM_prop}, \eqref{eq:updatea_condL_prop} combined with Lemma~\ref{lemma:bnd_S1S8norm_E} satisfy \eqref{eq:smallperturb} in Lemma~\ref{lemma:daviskahan} and we obtain the error bound in \eqref{eq:errbnd}.

It remains to compute $\norm{(\mN - \underline{\mM}) \vq_1}_2 / \lambda$. The $\ell_2$-norm of $(\mM - \underline{\mM}) \vq_1$ satisfies
\begin{align}
\norm{(\mM - \underline{\mM}) \vq_1}_2
&\leq \frac{K^2 \norm{\vx}_2^2 \norm{\vb}_2^2 \sin^2\theta \norm{\mathrm{diag}(|\va|^2) \va}_2}{\norm{\va}_2}
+ \frac{\norm{(\mId_M \otimes \widehat{\vb}^*) \mE (\va \otimes \widehat{\vb})}_2}{\norm{\va}_2} \nonumber \\
&\leq K^2 \norm{\vx}_2^2 \norm{\vb}_2^2 \norm{\va}_\infty^2 \, \sin^2\theta
+ 3 \sin\theta \, \tnorm{\mE}_{S_1 \to S_\infty}
+ \frac{\cos^2\theta \, \tnorm{\mE (\va \otimes \vb)}_{S_\infty}}{\norm{\va}_2 \norm{\vb}_2},
\label{eq:errbnd:lemma:update_a}
\end{align}
where the second step follow from the decomposition of $\widehat{\vb}$ given by
\[
\widehat{\vb} = \mP_{\vb} \widehat{\vb} + \mP_{\vb^\perp} \widehat{\vb}.
\]
which satisfies $\norm{\mP_{\vb} \widehat{\vb}}_2 = \cos\theta$ and $\norm{\mP_{\vb^\perp} \widehat{\vb}}_2 = \sin\theta$.  By dividing the right-hand side of \eqref{eq:errbnd:lemma:update_a} by $\lambda$, we obtain
\begin{align}
\frac{4 \norm{(\mM - \underline{\mM}) \vq_1}_2}{\lambda}
& \leq
\frac{4 \norm{\va}_\infty^2 \, \sin^2\theta}{\norm{\va}_2^2 \, \cos^2\theta}
+ \frac{12 \sin^2\theta \, \tnorm{\mE}_{S_1 \to S_\infty}}{K^2 \norm{\vx}_2^2 \norm{\vb}_2^2 \norm{\va}_2^2 \cos^2\theta}
+ \frac{4 \tnorm{\mE \vu}_{S_\infty}}{K^2 \norm{\vx}_2^2 \norm{\vu}_2^3} \nonumber \\
& \leq
\Big( \frac{8\mu^2}{M} + \frac{24 \tnorm{\mE}_{S_1 \to S_\infty}}{K^2 \norm{\vx}_2^2 \norm{\vu}_2^2} \Big) \sin\theta
+ \frac{4 \tnorm{\mE \vu}_{S_\infty}}{K^2 \norm{\vx}_2^2 \norm{\vu}_2^3},
\label{eq:errbnd2:lemma:update_a}
\end{align}
where the second step follows from \eqref{eq:starterrb}.

By Lemma~\ref{lemma:bnd_S1S8norm_E}, the constant factor for $\sin\theta$ in \eqref{eq:errbnd2:lemma:update_a} becomes less than $1/2$ as we choose $C_1(\beta)$ in \eqref{eq:updatea_condK_prop}, \eqref{eq:updatea_condM_prop}, \eqref{eq:updatea_condL_prop} sufficiently large.  This gives \eqref{eq:updatea_recursion}, where the expression for $\kappa$ follows from Lemma~\ref{lemma:bnd_2norm_Eu}.
This completes the proof.

\subsection{Proof of Proposition~\ref{prop:update_b}}
\label{subsec:proof_prop_update_b}

The proof of Proposition~\ref{prop:update_b} is similar to that of Proposition~\ref{prop:update_a}.  Thus we will only highlight the differences between the two proofs.

Without loss of generality, we assume that $\norm{\widehat{\va}}_2 = 1$.  Let $\breve{\theta} = \angle(\va,\widehat{\va})$.  The assumption in \eqref{eq:starterra} implies $\breve{\theta} \leq \pi/4$.  This time, we compute the least dominant eigenvector of $(\widehat{\va}^* \otimes \mId_D) \mA (\widehat{\va} \otimes \mId_D)$.  From \eqref{eq:mA:lemma:update_a}, we obtain
\begin{equation}
\label{eq:restrictedmA:lemma:update_b}
(\widehat{\va}^* \otimes \mId_D) \underline{\mA} (\widehat{\va} \otimes \mId_D)
=
K^2 \norm{\vx}_2^2 \norm{\va}_2^2
(\norm{\vb}_2^2 \mId_D - \cos^2\breve{\theta} \, \vb \vb^*)
- K^2 \norm{\vx}_2^2 \norm{\vb}_2^2 \norm{|\va| \odot \widehat{\va}}_2^2 \mP_{\vb^\perp}.
\end{equation}

We consider the matrix
\begin{align*}
& K^2 \norm{\vx}_2^2 \norm{\vb}_2^2 \norm{\va}_2^2 \mId_D - (\widehat{\va}^* \otimes \mId_D) \mA (\widehat{\va} \otimes \mId_D) \\
&= K^2 \norm{\vx}_2^2 \norm{\va}_2^2 \cos^2\breve{\theta} \, \vb \vb^* \\
&+ K^2 \norm{\vx}_2^2 \norm{\vb}_2^2 \norm{|\va| \odot \widehat{\va}}_2^2 \mP_{\vb^\perp}
- (\widehat{\va}^* \otimes \mId_D) \mE (\widehat{\va} \otimes \mId_D)
\end{align*}
as a perturbed version of $K^2 \norm{\vx}_2^2 \norm{\va}_2^2 \cos^2\breve{\theta} \, \vb \vb^*$.  The difference of the two matrices satisfies
\begin{align*}
& \Bignorm{K^2 \norm{\vx}_2^2 \norm{\vb}_2^2 \norm{\va}_2^2 \mId_D - (\widehat{\va}^* \otimes \mId_D) \underline{\mA} (\widehat{\va} \otimes \mId_D) - K^2 \norm{\vx}_2^2 \norm{\vb}_2^2 \norm{\va}_2^2 \cos^2\breve{\theta} \mP_{\vb}} \\
& \leq \Bignorm{K^2 \norm{\vx}_2^2 \norm{\vb}_2^2 \norm{|\va| \odot \widehat{\va}}_2^2 \mP_{\vb^\perp}}
+ \norm{(\widehat{\va}^* \otimes \mId_D) \mE (\widehat{\va} \otimes \mId_D)} \\
& \leq K^2 \norm{\vx}_2^2 \norm{\vb}_2^2 \norm{\va}_\infty^2
+ \tnorm{\mE}_{S_1 \to S_\infty}.
\end{align*}

For sufficiently large $C_1(\beta)$, the conditions in \eqref{eq:starterrb},
\eqref{eq:updatea_condK_prop}, \eqref{eq:updatea_condM_prop}, \eqref{eq:updatea_condL_prop} imply
\[
K^2 \norm{\vx}_2^2 \norm{\vb}_2^2 \norm{\va}_2^2 \cos^2\theta
> 2 (K^2 \norm{\vx}_2^2 \norm{\vb}_2^2 \norm{\va}_\infty^2 \sin^2\theta
+ \tnorm{\mE}_{S_1 \to S_\infty}).
\]
Therefore, $\widehat{\vb}$ is also a unique dominant eigenvector of $K^2 \norm{\vx}_2^2 \norm{\vb}_2^2 \norm{\va}_2^2 \mId_D - (\widehat{\va}^* \otimes \mId_D) \mA (\widehat{\va} \otimes \mId_D)$.

Next we apply Lemma~\ref{lemma:daviskahan} for
\begin{align*}
\underline{\mM} &= K^2 \norm{\vx}_2^2 \norm{\va}_2^2 \cos^2\breve{\theta} \, \vb \vb^*, \\
\mM &= K^2 \norm{\vx}_2^2 \norm{\vb}_2^2 \norm{\va}_2^2 \mId_D - (\widehat{\va}^* \otimes \mId_D) \mA (\widehat{\va} \otimes \mId_D), \\
\vq_1 &= \frac{\vb}{\norm{\vb}_2}, \qquad \widetilde{\vq} = \widehat{\vb}.
\end{align*}
Then $\lambda$ and $\mD$ in Lemma~\ref{lemma:daviskahan} are given as $\lambda = K^2 \norm{\vx}_2^2 \norm{\vb}_2^2 \norm{\va}_2^2 \cos^2\breve{\theta}$ and $\mD = \vzero$.

Similarly to the proof of Proposition~\ref{prop:update_a}, we show
\begin{align*}
\frac{\norm{\mM - \underline{\mM}}}{\lambda}
& \leq
\frac{2\mu^2}{M}
+ \frac{2 \tnorm{\mE}_{S_1 \to S_\infty}}{K^2 \norm{\vx}_2^2 \norm{\vb}_2^2 \norm{\va}_2^2}
\end{align*}
and
\begin{align*}
\frac{4 \norm{(\mM - \underline{\mM}) \vq_1}_2}{\lambda}
& \leq
\frac{24 \tnorm{\mE}_{S_1 \to S_\infty}}{K^2 \norm{\vx}_2^2 \norm{\vu}_2^2} \sin\theta
+ \frac{4 \tnorm{\mE \vu}_{S_\infty}}{K^2 \norm{\vx}_2^2 \norm{\vu}_2^3}.
\end{align*}
Here we used the decomposition of $\widehat{\va}$ given by
\[
\widehat{\va} = \mP_{\va} \widehat{\va} + \mP_{\va^\perp} \widehat{\va}.
\]
which satisfies $\norm{\mP_{\va} \widehat{\va}}_2 = \cos\breve{\theta}$ and $\norm{\mP_{\va^\perp} \widehat{\va}}_2 = \sin\breve{\theta}$.

The remaining steps are identical to those in the proof of Proposition~\ref{prop:update_a} and we omit further details.

\section{Convergence of Rank-1 Truncated Power Method}
\label{sec:conv_tpm}

In this section, we prove Proposition~\ref{prop:tpm4mbd}.  First we present a theorem that shows local convergence of the rank-1 truncated power method for general matrix input $\mB$.  Then we will show the proof of Proposition~\ref{prop:tpm4mbd} as its corollary.

The separability structure in \eqref{eq:const_evd} corresponds to the rank-1 structure when the eigenvector is rearranged as a matrix.  We introduce a collection of structured subspaces, where their Minkowski sum is analogous to the support in the sparsity model.  For $(\va,\vb) \in \mathbb{C}^M \times \mathbb{C}^D$, we define
\[
T(\va,\vb) := \{ \va \otimes \vxi + \vq \otimes \vb ~|~ \vxi \in \mathbb{C}^D,~ \vq \in \mathbb{C}^M \}.
\]
Then
\[
\mathrm{mat}(T(\va,\vb)) = \{ \mathrm{mat}(\vv) ~|~ \vv \in T(\va,\vb) \}
\]
is equivalent to the tangent space of the rank-1 matrix $\mU = \va \vb^\transpose$.

Now we state a local convergence result for the rank-1 truncated power method in the following theorem, the proof of which is postponed to Section~\ref{sec:proof:thm:conv_tpm}.

\begin{theorem}
\label{thm:conv_tpm}
Let $\vu = \va \otimes \vb$ be a unique dominant eigenvector of $\underline{\mB}$.
Define
\[
\widetilde{\lambda}_2(\underline{\mB}) :=
\sup_{\vv,(\widehat{\va},\widehat{\vb}),(\widetilde{\va},\widetilde{\vb})}
\Big\{
\vv^* \underline{\mB} \vv ~|~ \norm{\vv}_2 \leq 1,~ \vv \in \vu^\perp \cap [T(\va,\vb)+T(\widehat{\va},\widehat{\vb})+T(\widetilde{\va},\widetilde{\vb})]
\Big\}.
\]
Suppose that
\begin{equation}
\frac{\sqrt{5} (\widetilde{\lambda}_2(\underline{\mB}) + 6 \tnorm{\mB - \underline{\mB}}_{S_1 \to S_\infty})}{\sqrt{1-\tau^2} \, \lambda_1(\underline{\mB}) - \tau \, \widetilde{\lambda}_2(\underline{\mB}) - 6 (\sqrt{1-\tau^2} + \tau) \tnorm{\mB - \underline{\mB}}_{S_1 \to S_\infty}}
< \mu,
\label{eq:conv_tpm_cond1}
\end{equation}
\begin{equation}
\frac{4\sqrt{6} \tnorm{\mB - \underline{\mB}}_{S_1 \to S_\infty}}{\lambda_1(\underline{\mB})}
\leq \min\Big[ \frac{1}{3\sqrt{2}} , \frac{(1-\mu) \tau}{1+\mu} \Big],
\label{eq:conv_tpm_cond2}
\end{equation}
and
\begin{equation}
\widetilde{\lambda}_2(\underline{\mB}) + 6 \tnorm{\mB - \underline{\mB}}_{S_1 \to S_\infty}
\leq \frac{\lambda_1(\underline{\mB})}{5}
\label{eq:conv_tpm_cond3}
\end{equation}
hold for some $0 < \mu < 1$ and $0 < \tau < \frac{1}{3\sqrt{2}}$.
If $\sin \angle(\vu_0, \vu) \leq \tau$, then $(\vu_t)_{t\in\mathbb{N}}$ produced by Algorithm~\ref{alg:tpm} satisfies
\begin{equation}
\sin \angle(\vu_t, \vu) \leq \mu \sin \angle(\vu_{t-1}, \vu)
+ \frac{(1+\mu)4\sqrt{6} \tnorm{(\mB - \underline{\mB}) \vu}_{S_\infty}}{\lambda_1(\underline{\mB})}, \quad \forall t \in \mathbb{N}.
\label{eq:conv_tpm}
\end{equation}
\end{theorem}

Proposition~\ref{prop:tpm4mbd} is a direct consequence of Theorem~\ref{thm:conv_tpm} for the case where the input matrix $\mB$ is given as $\mB = \norm{\mathbb{E}[\mA]}\, \mId_{MD} - \mA$.  We provide the proof below.

\begin{proof}[Proof of Proposition~\ref{prop:tpm4mbd}]
Given $\mB = \norm{\mathbb{E}[\mA]}\, \mId_{MD} - \mA$, we apply Theorem~\ref{thm:conv_tpm} for
\[
\underline{\mB} = K^2 \norm{\vx}_2^2 \vu \vu^*.
\]
Then the difference between $\mB$ and $\underline{\mB}$ is given by
\begin{equation}
\label{eq_diff:proof:prop:tpm4mbd}
\mB - \underline{\mB} =
(\norm{\mathbb{E}[\mA]} - K^2 \norm{\vx}_2^2 \norm{\vu}_2^2) \mId_{MD}
+ K^2 \norm{\vx}_2^2 \mUpsilon
- \mE.
\end{equation}

In Section~\ref{subsec:proof_prop_update_a}, we have computed $\underline{\mA} = \mathbb{E}[\mA]$ in \eqref{eq:mA:lemma:update_a}, which is rewritten as
\begin{equation}
\label{eq_decomp_A:proof:prop:tpm4mbd}
\underline{\mA} = K^2 \norm{\vx}_2^2 (\norm{\vu}_2^2 \mP_{\vu^\perp} - \mUpsilon)
\end{equation}
with
\[
\mUpsilon = \mathrm{diag}(|\va|^2) \otimes \norm{\vb}_2^2 \mP_{\vb^\perp},
\]
where $\vu = \va \otimes \vb$.

Therefore, it follows from \eqref{eq_decomp_A:proof:prop:tpm4mbd} that
\begin{equation}
\label{eq_bnd1:proof:prop:tpm4mbd}
|\norm{\underline{\mA}} - K^2 \norm{\vx}_2^2 \norm{\vu}_2^2|
\leq K^2 \norm{\vx}_2^2 \norm{\mUpsilon}
\leq K^2 \norm{\vx}_2^2 \norm{\vb}_2^2 \norm{\va}_\infty^2.
\end{equation}

Then by plugging in \eqref{eq_bnd1:proof:prop:tpm4mbd} to \eqref{eq_diff:proof:prop:tpm4mbd}, we obtain
\begin{equation}
\label{eq_bnd2:proof:prop:tpm4mbd}
\tnorm{\mB - \underline{\mB}}_{S_1 \to S_\infty}
\leq
2K^2 \norm{\vx}_2^2 \norm{\vb}_2^2 \norm{\va}_\infty^2
+ \tnorm{\mE}_{S_1 \to S_\infty}.
\end{equation}

On the other hand, $\underline{\mB}$ is a rank-1 matrix whose eigenvector is collinear with $\vu$ and the largest eigenvalue is given by
\begin{equation}
\label{eq_la1:proof:prop:tpm4mbd}
\lambda_1(\underline{\mB}) = K^2 \norm{\vx}_2^2 \norm{\vb}_2^2 \norm{\va}_2^2.
\end{equation}
Therefore, $\underline{\mB}$ also satisfies
\[
\widetilde{\lambda}_2(\underline{\mB}) = 0.
\]

Since $\widetilde{\lambda}_2(\underline{\mB}) = 0$, \eqref{eq:conv_tpm_cond1} and \eqref{eq:conv_tpm_cond2} are implied by
\begin{equation}
\label{eq_suff:proof:prop:tpm4mbd}
\frac{\tnorm{\mB - \underline{\mB}}_{S_1 \to S_\infty}}{\lambda_1(\underline{\mB})}
\leq C_0 \min\Big[ \mu\sqrt{1-\tau^2}, \frac{(1-\mu) \tau}{1+\mu} \Big],
\end{equation}
for a numerical constant $C_0$.

By applying \eqref{eq_la1:proof:prop:tpm4mbd} and the tail estimate of $\tnorm{\mE}_{S_1 \to S_\infty}$ given in Lemma~\ref{lemma:bnd_S1S8norm_E} to \eqref{eq_bnd2:proof:prop:tpm4mbd}, we verify that the sufficient condition in \eqref{eq_suff:proof:prop:tpm4mbd} is implied by \eqref{eq:updatea_condK_prop}, \eqref{eq:updatea_condM_prop}, and \eqref{eq:updatea_condL_prop} for $C_1 = c(\mu,\tau) C_1', C_2 = c(\mu,\tau) C_2'$ where $C_1'$ and $C_2'$ are constants that only depend on $\beta$.

Since the conditions in \eqref{eq:conv_tpm_cond1} and \eqref{eq:conv_tpm_cond2} are satisfied, Theorem~\ref{thm:conv_tpm} provides the error bound in \eqref{eq:conv_tpm_mbd}.  This completes the proof.
\end{proof}

\subsection{Proof of Theorem~\ref{thm:conv_tpm}}
\label{sec:proof:thm:conv_tpm}

In order to prove Theorem~\ref{thm:conv_tpm}, we first provide lemmas, which show upper bounds on the estimation error, given in terms of the principal angle, in the corresponding steps of Algorithm~\ref{alg:tpm}.

The first lemma provides upper bounds on norms of a matrix and a vector when they are restricted with a projection operator onto a subspace with the separability structure.

\begin{lemma}
\label{lemma:projnorm}
Let
\[
\breve{T} = \sum_{k=1}^r T(\va_k,\vb_k)
\]
for $\{(\va_k,\vb_k)\}_{k=1}^r \subset \mathbb{C}^M \times \mathbb{C}^D$,
$\mM \in \mathbb{C}^{MD \times MD}$, and $\vu \in \mathbb{C}^{MD}$.
Then
\[
\norm{\mP_{\breve{T}} \mM \mP_{\breve{T}}} \leq 2r \tnorm{\mM}_{S_1 \to S_\infty}
\]
and
\[
\norm{\mP_{\breve{T}} \mM \vu}_2 \leq \sqrt{2r} \tnorm{\mM \vu}_{S_\infty}.
\]
\end{lemma}
\begin{proof}
Let $\vv \in \breve{T}$. Then $\mathrm{rank}(\mathrm{mat}(\vv)) \leq 2r$.
Let
\[
\mathrm{mat}(\vv) = \sum_{l=1}^{2r} \sigma_l \vq_l \vxi_l^\transpose
\]
denotes the singular value decomposition of $\mathrm{mat}(\vv)$, where $\norm{\vq_l}_2 = \norm{\vxi_l}_2 = 1$ and $\sigma_l \geq 0$ for $k=1,\dots,2r$. Then
\[
\vv = \sum_{l=1}^{2r} \sigma_l \vq_l \otimes \vxi_l.
\]
Similarly, we can represent $\vv' \in \breve{T}$ as
\[
\vv' = \sum_{j=1}^{2r} \sigma_j' \vq_j' \otimes \vxi_j'.
\]

Then
\begin{align*}
|\langle \vv', \mM \vv \rangle|
& \leq \sum_{j,l=1}^{2r} \sigma_l \sigma_j' |\langle (\vq_j' \otimes \vxi_j'), \mM (\vq_l \otimes \vxi_l) \rangle| \\
& \leq \sum_{l=1}^{2r} \sigma_l \sum_{j=1}^{2r} \sigma_j' \tnorm{\mM}_{S_1 \to S_\infty} \\
& \leq 2r \norm{\vv}_2 \norm{\vv'}_2 \tnorm{\mM}_{S_1 \to S_\infty}.
\end{align*}

Therefore,
\[
\norm{\mP_{\breve{T}} \mM \mP_{\breve{T}}}
= \sup_{\vv,\vv' \in \breve{T}}
\{
\langle \vv', \mM \vv \rangle
~|~
\norm{\vv}_2 = \norm{\vv'}_2 = 1
\}
\leq 2r \tnorm{\mM}_{S_1 \to S_\infty}.
\]
This proves the first assertion.
The second assertion is obtained in a similar way by fixing $\vv = \vu$.
\end{proof}

The following lemma is a direct consequence of the Davis-Kahan Theorem together with Lemma~\ref{lemma:projnorm}.

\begin{lemma}[Perturbation]
\label{lemma:constrained_perturbation}
Let $\{(\va_k,\vb_k)\}_{k=1}^r \subset \mathbb{C}^M \times \mathbb{C}^D$ satisfy
\[
T(\va,\vb) \subset \sum_{k=1}^r T(\va_k,\vb_k) =: \breve{T}.
\]
Let $\vv$ (resp. $\vu$) be a unique most dominant eigenvector of $\mP_{\breve{T}} \mM_1 \mP_{\breve{T}}$ (resp. $\mP_{\breve{T}} \mM_2 \mP_{\breve{T}}$).
If
\begin{equation}
\lambda_2(\mP_{\breve{T}} \mM_2 \mP_{\breve{T}}) + 2r \, \tnorm{\mM_1 - \mM_2}_{S_1 \to S_\infty} \leq \frac{\lambda_1(\mP_{\breve{T}} \mM_2 \mP_{\breve{T}})}{5},
\label{eq:cond_constrained_perturbation}
\end{equation}
then
\[
\sin \angle (\vv,\vu) \leq \frac{4 \sqrt{2r} \tnorm{(\mM_1 - \mM_2) \vu}_{S_\infty}}{\lambda_1(\mP_{\breve{T}} \mM_2 \mP_{\breve{T}})}.
\]
\end{lemma}

The following lemma shows how the conventional power method converges depending on the largest and second largest eigenvalues.

\begin{lemma}[{A Single Iteration of Power Method \cite[Theorem~1.1]{stewart2001matrix}}]
\label{lemma:iter_power_method}
Let $\mM$ have a unique dominant eigenvector $\vv$.
Then
\[
\sin \angle(\mM \widehat{\vv}, \vv)
\leq \frac{\lambda_2(\mM) \sin \angle(\widehat{\vv}, \vv)}{\lambda_1(\mM) \cos \angle(\widehat{\vv}, \vv) - \lambda_2(\mM) \sin \angle(\widehat{\vv}, \vv)}
\]
for any $\widehat{\vv}$ such that $\langle \widehat{\vv}, \vv \rangle \neq 0$.
\end{lemma}

The following lemma is a modification of \cite[Lemma~12]{yuan2013truncated} and shows that the correlation is partially preserved after the rank-1 truncation. Unlike the canonical sparsity model, where the atoms are mutually orthogonal, in the low-rank atomic model, atoms in an atomic decomposition may have correlation.  Our proof addresses this general case and the argument here also applies to an abstract atomic model.

\begin{lemma}[Correlation after the Rank-1 Truncation]
\label{lemma:cor_after_rank1}
Let $\breve{\vv} \in \mathbb{C}^{MD}$ satisfy $\norm{\breve{\vv}}_2 = 1$ and $\mathrm{rank}(\mathrm{mat}(\breve{\vv})) = 1$. For $\vv \in \mathbb{C}^{MD}$ such that $\norm{\vv}_2 = 1$, let $\widehat{\mV} \in \mathbb{C}^{M \times D}$ denote the best rank-1 approximation of $\mV = \mathrm{mat}(\vv)$ and $\widehat{\vv} = \mathrm{vec}(\widehat{\mV})$. Then
\begin{equation}
|\langle \widehat{\vv}, \breve{\vv} \rangle|
\geq |\langle \vv, \breve{\vv} \rangle| - \min\Big(\sqrt{1 - |\langle \vv, \breve{\vv} \rangle|^2}, 2 (1 - |\langle \vv, \breve{\vv} \rangle|^2)\Big).
\label{eq:cor_after_rank1}
\end{equation}
\end{lemma}

\begin{proof}[Proof of Lemma~\ref{lemma:cor_after_rank1}]
There exist $\breve{\va} \in \mathbb{C}^M$ and $\breve{\vb} \in \mathbb{C}^D$ such that
\[
\mU = \mathrm{mat}(\breve{\vv}) = \breve{\va} \breve{\vb}^\transpose.
\]

Let $\widehat{\va} \in \mathbb{C}^D$ and $\widehat{\vb} \in \mathbb{C}^D$ respectively denote the left and right singular vectors of the rank-1 matrix $\widehat{\mV}$.
Define $T_1 = T(\{(\breve{\va},\breve{\vb})\})$, $T_2 = T(\{(\widehat{\va},\widehat{\vb})\})$, and $T_3 = T_1 \cap T_2$.
Then $T_1 + T_2$ is rewritten as
\begin{equation}
T_1 + T_2
= \mP_{T_2^\perp} T_1 \oplus T_2
= \mP_{T_2^\perp} T_1 \oplus T_3 \oplus \mP_{T_3^\perp} T_2.
\label{eq:decompT1T2_A}
\end{equation}
Similarly, we also have
\begin{equation}
T_1 + T_2
= T_1 \oplus \mP_{T_1^\perp} T_2
= \mP_{T_3^\perp} T_1 \oplus T_3 \oplus \mP_{T_1^\perp} T_2.
\label{eq:decompT1T2_B}
\end{equation}

By the definition of $T_2$, we have
\[
\norm{\mP_{T_2} \vv}_2 \geq \norm{\mP_{T_1} \vv}_2.
\]
Therefore,
\[
\norm{\mP_{\mP_{T_3^\perp} T_2} \vv}_2 \geq \norm{\mP_{\mP_{T_3^\perp} T_1} \vv}_2.
\]
Then by \eqref{eq:decompT1T2_A} and \eqref{eq:decompT1T2_B} it follows that
\begin{equation}
\norm{\mP_{\mP_{T_1^\perp} T_2} \vv}_2 \geq \norm{\mP_{\mP_{T_2^\perp} T_1} \vv}_2.
\label{eq:bigger_proj}
\end{equation}

By the Cauchy-Schwartz inequality and the Pythagorean identity, we have
\begin{align*}
|\langle \vv, \breve{\vv} \rangle|^2
= |\langle \mP_{T_1} \vv, \breve{\vv} \rangle|^2
\leq \norm{\mP_{T_1} \vv}_2^2
\leq 1 - \norm{\mP_{T_1^\perp} \vv}_2^2
\leq 1 - \norm{\mP_{\mP_{T_1^\perp} T_2} \vv}_2^2
\leq 1 - \norm{\mP_{\mP_{T_2^\perp} T_1} \vv}_2^2,
\end{align*}
where the last step follow from \eqref{eq:bigger_proj}.
The above inequality is rearranged as
\begin{equation}
\norm{\mP_{\mP_{T_2^\perp} T_1} \vv}_2 \leq \sqrt{1 - |\langle \vv, \breve{\vv} \rangle|^2}.
\label{eq:ub_Pvv}
\end{equation}

We may assume that $|\langle \vv, \breve{\vv} \rangle| > 2^{-1/2}$. Otherwise, the right-hand side of \eqref{eq:cor_after_rank1} becomes negative and the inequality holds trivially.
Then by \eqref{eq:ub_Pvv} we have
\[
\norm{\mP_{\mP_{T_2^\perp} T_1} \vv}_2 < |\langle \vv, \breve{\vv} \rangle|,
\]
which also implies
\begin{equation}
\norm{\mP_{\mP_{T_2^\perp} T_1} \vv}_2 \norm{\mP_{\mP_{T_2^\perp} T_1} \breve{\vv}}_2 < |\langle \vv, \breve{\vv} \rangle|.
\label{cor_after_rank1_cond1}
\end{equation}

Since $\mP_{T_1+T_2} \breve{\vv} = \breve{\vv}$, we have
\begin{align*}
|\langle \vv, \breve{\vv} \rangle|
&= |\langle \mP_{T_1+T_2} \vv, \breve{\vv} \rangle| \\
&= |\langle (\mP_{\mP_{T_2^\perp} T_1} + \mP_{T_2}) \vv, \breve{\vv} \rangle| \\
&= |\langle \mP_{\mP_{T_2^\perp} T_1} \vv, \breve{\vv} \rangle| + |\langle \mP_{T_2} \vv, \breve{\vv} \rangle| \\
&\leq \norm{\mP_{\mP_{T_2^\perp} T_1} \vv}_2 \norm{\mP_{\mP_{T_2^\perp} T_1} \breve{\vv}}_2
+ \norm{\mP_{T_2} \vv}_2 \norm{\mP_{T_2} \breve{\vv}}_2 \\
&\leq \norm{\mP_{\mP_{T_2^\perp} T_1} \vv}_2 \norm{\mP_{\mP_{T_2^\perp} T_1} \breve{\vv}}_2
+ \sqrt{1 - \norm{\mP_{\mP_{T_2^\perp} T_1} \vv}_2^2} \sqrt{1 - \norm{\mP_{\mP_{T_2^\perp} T_1} \breve{\vv}}_2^2}
\end{align*}

By solving the above inequality for $\norm{\mP_{\mP_{T_2^\perp} T_1} \breve{\vv}}_2$ under the condition in \eqref{cor_after_rank1_cond1}, we obtain
\begin{align}
\norm{\mP_{\mP_{T_2^\perp} T_1} \breve{\vv}}_2
&\leq \norm{\mP_{\mP_{T_2^\perp} T_1} \vv}_2 |\langle \vv, \breve{\vv} \rangle|
+ \sqrt{1 - \norm{\mP_{\mP_{T_2^\perp} T_1} \vv}_2^2} \sqrt{1 - |\langle \vv, \breve{\vv} \rangle|^2} \nonumber \\
&\leq \min(1, 2 \sqrt{1 - |\langle \vv, \breve{\vv} \rangle|^2}). \label{eq:ub_Pvu}
\end{align}

Since $\mP_{T_2} (\vv - \widehat{\vv}) = \vzero$, we have
\begin{align*}
|\langle \vv, \breve{\vv} \rangle| - |\langle \widehat{\vv}, \breve{\vv} \rangle|
&\leq |\langle \vv - \widehat{\vv}, \breve{\vv} \rangle| \\
&= |\langle \mP_{\mP_{T_2^\perp} T_1} (\vv - \widehat{\vv}), \breve{\vv} \rangle| \\
&= |\langle \mP_{\mP_{T_2^\perp} T_1} \vv, \breve{\vv} \rangle| \\
&\leq \norm{\mP_{\mP_{T_2^\perp} T_1} \vv}_2 \norm{\mP_{\mP_{T_2^\perp} T_1} \breve{\vv}}_2 \\
&\leq \min\Big(\sqrt{1 - |\langle \vv, \breve{\vv} \rangle|^2}, 2 (1 - |\langle \vv, \breve{\vv} \rangle|^2)\Big),
\end{align*}
where the last step follows from \eqref{eq:ub_Pvv} and \eqref{eq:ub_Pvu}.
The assertion is obtained by a rearrangement.
\end{proof}

\begin{proof}[Proof of Theorem~\ref{thm:conv_tpm}]
We use the mathematical induction and it suffices to show $\sin\angle(\vv_t,\vu) \leq \tau$ and \eqref{eq:conv_tpm} hold provided that $\sin\angle(\vv_{t-1},\vu) \leq \tau$ for fixed $t$.

Since $\mathrm{rank}(\mathrm{mat}(\vv_t)) = 1$, there exist $\va_t \in \mathbb{C}^M$ and $\vb_t \in \mathbb{C}^D$ such that $\vv_t = \va_t \otimes \vb_t$.
Similarly, there exist $\va_{t-1} \in \mathbb{C}^M$ and $\vb_{t-1} \in \mathbb{C}^D$ that satisfy $\vv_{t-1} = \va_{t-1} \otimes \vb_{t-1}$.
Let
\[
\breve{T} = T(\va_{t-1},\vb_{t-1}) + T(\va_t,\vb_t) + T(\va,\vb).
\]
Then define
\[
\widetilde{\vv}_t' = \frac{\mP_{\breve{T}} \mB \mP_{\breve{T}} \vv_{t-1}}{\norm{\mP_{\breve{T}} \mB \mP_{\breve{T}} \vv_{t-1}}_2}.
\]
Note that Algorithm~\ref{alg:tpm} produces the same result even when $\widetilde{\vv}_t$ is replaced by $\widetilde{\vv}_t'$.
Indeed, since $\mP_{\breve{T}} \vv_{t-1} = \vv_{t-1}$, it follows that $\mathrm{mat}(\mB \vv_{t-1})$ and $\mathrm{mat}(\mB \mP_{\breve{T}} \vv_{t-1})$ are collinear, so are their rank-1 approximations. Moreover, by $\vv_t$ is obtained normalizing as the normalized rearrangement of the rank-1 approximation of $\mathrm{mat}(\mB \vv_{t-1})$, by the construction of $\breve{T}$, it follows that $\mathrm{mat}(\mP_{\breve{T}} \mB \mP_{\breve{T}} \vv_{t-1})$ is also collinear with $\mathrm{mat}(\mB \vv_{t-1})$.

Let $\widehat{\mV}_t'$ denote the rank-1 approximation of $\mathrm{mat}(\widetilde{\vv}_t')$ and $\widehat{\vu}_t' = \mathrm{vec}(\widehat{\mV}_t')$. Then we have
\[
\vv_t = \widehat{\vu}_t' / \norm{\widehat{\vu}_t'}_2.
\]

Let $\vv(\breve{T})$ denote a unique most dominant eigenvector of $\mP_{\breve{T}} \mB \mP_{\breve{T}}$.
Since $\norm{\widetilde{\vv}_t'}_2 = 1$, we have $\norm{\widehat{\vu}_t'}_2 \leq 1$.
Therefore,
\[
\sin \angle(\vv_t,\vv(\breve{T}))
= \sqrt{1 - |\langle \vv_t, \vv(\breve{T}) \rangle|^2}
\leq \sqrt{1 - |\langle \widehat{\vu}_t', \vv(\breve{T}) \rangle|^2}.
\]

We apply Lemma~\ref{lemma:cor_after_rank1} with $\breve{\vv} = \vv(\breve{T})$ and $\vv = \widetilde{\vv}_t'$. By Lemma~\ref{lemma:cor_after_rank1}, we have
\[
|\langle \widehat{\vu}_t', \vv(\breve{T}) \rangle|
\geq |\langle \widetilde{\vv}_t', \vv(\breve{T}) \rangle| - \min\Big(\sqrt{1 - |\langle \widetilde{\vv}_t', \vv(\breve{T}) \rangle|^2}, 2 (1 - |\langle \widetilde{\vv}_t', \vv(\breve{T}) \rangle|^2)\Big),
\]
which implies
\begin{align*}
\sqrt{1 - |\langle \widehat{\vu}_t', \vv(\breve{T}) \rangle|^2}
\leq \sqrt{5} \sqrt{1 - |\langle \widetilde{\vv}_t', \vu (\breve{T}) \rangle|^2}
= \sqrt{5} \sin \angle(\widetilde{\vv}_t',\vv(\breve{T})).
\end{align*}

We apply Lemma~\ref{lemma:iter_power_method} with $\mM = \mP_{\breve{T}} \mB \mP_{\breve{T}}$, $\vv = \vv(\breve{T})$, and $\widehat{\vv} = \vv_{t-1}$. Then
\begin{align}
\sin \angle(\widetilde{\vv}_t', \vv(\breve{T}))
& \leq \frac{\lambda_2(\mM) \sin \angle(\vv_{t-1}, \vv(\breve{T}))}{\lambda_1(\mM) \cos \angle(\vv_{t-1}, \vv(\breve{T})) - \lambda_2(\mM) \sin \angle(\vv_{t-1}, \vv(\breve{T}))} \nonumber \\
& \leq \frac{\lambda_2(\mM)}{\lambda_1(\mM) \, \sqrt{1-\tau^2} - \lambda_2(\mM) \, \tau} \cdot \sin \angle(\vv_{t-1}, \vv(\breve{T})),
\label{eq1:proof_conv_tpm}
\end{align}
where the last step follows from $\sin \angle(\vv_{t-1}, \vv(\breve{T})) \leq \tau'$.

Next we compute the two largest eigenvalues of $\mP_{\breve{T}} \mB \mP_{\breve{T}}$. Since $\vu$ is a unique dominant eigenvector of $\underline{\mB}$ and $\mP_{\breve{T}} \vu = \vu$, we have $\lambda_1(\mP_{\breve{T}} \underline{\mB} \mP_{\breve{T}}) = \lambda_1(\underline{\mB})$. Therefore, by the triangle inequality,
\begin{align}
\lambda_1(\mP_{\breve{T}} \mB \mP_{\breve{T}})
& \geq \lambda_1(\mP_{\breve{T}} \underline{\mB} \mP_{\breve{T}}) - \norm{\mP_{\breve{T}} (\mB - \underline{\mB}) \mP_{\breve{T}}} \nonumber \\
& \geq \lambda_1(\underline{\mB}) - 6 \, \tnorm{\mB - \underline{\mB}}_{S_1 \to S_\infty}.
\label{eq2:proof_conv_tpm}
\end{align}
By the variational characterization of eigenvalues, we have
\begin{align*}
\lambda_2(\mP_{\breve{T}} \underline{\mB} \mP_{\breve{T}})
=
\sup_{\vv} \{
\vv^* \underline{\mB} \vv ~|~ \norm{\vv}_2 \leq 1,~ \vv \in \vu^\perp \cap \breve{T}
\}
\leq \widetilde{\lambda}_2(\underline{\mB}).
\end{align*}
Therefore,
\begin{align}
\lambda_2(\mP_{\breve{T}} \mB \mP_{\breve{T}})
& \leq \lambda_2(\mP_{\breve{T}} \underline{\mB} \mP_{\breve{T}}) + \norm{\mP_{\breve{T}} (\mB - \underline{\mB}) \mP_{\breve{T}}} \nonumber \\
& \leq \widetilde{\lambda}_2(\underline{\mB}) + 6 \, \tnorm{\mB - \underline{\mB}}_{S_1 \to S_\infty}.
\label{eq3:proof_conv_tpm}
\end{align}

By plugging in \eqref{eq2:proof_conv_tpm} and \eqref{eq3:proof_conv_tpm} into \eqref{eq1:proof_conv_tpm}, we obtain that \eqref{eq:conv_tpm_cond1} implies
\begin{equation}
\sin \angle(\vv_t,\vv(\breve{T})) \leq \mu \sin \angle(\vv_{t-1},\vv(\breve{T})).
\label{eq4:proof_conv_tpm}
\end{equation}

Moreover, by the transitivity of the angle function \cite{wedin1983angles}, we also have
\begin{equation}
\angle(\vv_{t-1},\vv(\breve{T}))
\leq \angle(\vv_{t-1},\vu) + \angle(\vu,\vv(\breve{T})).
\label{eq5:proof_conv_tpm}
\end{equation}

Next we apply Lemma~\ref{lemma:constrained_perturbation} for $\mM_1 = \mB$, $\mM_2 = \underline{\mB}$, and $\vv = \vv(\breve{T})$.  Since \eqref{eq:conv_tpm_cond3} implies \eqref{eq:cond_constrained_perturbation}, it follows from Lemma~\ref{lemma:constrained_perturbation} that
\[
\sin \angle (\vu,\vv(\breve{T})) \leq \frac{4 \sqrt{6} \tnorm{(\mB - \underline{\mB}) \vu}_{S_\infty}}{\lambda_1(\underline{\mB})}.
\]
Then \eqref{eq:conv_tpm_cond2} implies
\[
\sin \angle(\vv_t,\vv(\breve{T})) < \frac{1}{3\sqrt{2}}.
\]
Since $\sin \angle(\vv_{t-1},\vu) \leq \tau < \frac{1}{3\sqrt{2}}$, it follows that \eqref{eq5:proof_conv_tpm} implies
\begin{equation}
\sin \angle(\vv_{t-1},\vv(\breve{T}))
\leq \sin \angle(\vv_{t-1},\vu) + \sin \angle(\vu,\vv(\breve{T})).
\label{eq6:proof_conv_tpm}
\end{equation}

By \eqref{eq:conv_tpm_cond2}, \eqref{eq4:proof_conv_tpm}, and \eqref{eq6:proof_conv_tpm},
\[
\sin \angle(\vv_t,\vv(\breve{T})) < \frac{1}{3\sqrt{2}}.
\]
Similarly to the previous case, the transitivity of the angle function implies
\[
\angle(\vv_t,\vu)
\leq \angle(\vv_t,\vv(\breve{T})) + \angle(\vv(\breve{T}),\vu).
\]
Then it follows that
\[
\sin \angle(\vv_t,\vu)
\leq \sin \angle(\vv_t,\vv(\breve{T})) + \sin \angle(\vv(\breve{T}),\vu).
\]

By collecting the above inequalities, we obtain
\begin{equation}
\sin \angle(\vv_t,\vu)
\leq \mu \, \sin \angle(\vv_t,\vv(\breve{T})) + (1+\mu) \, \sin \angle(\vv(\breve{T}),\vu).
\label{eq7:proof_conv_tpm}
\end{equation}

Finally, we verify that \eqref{eq7:proof_conv_tpm} and \eqref{eq:conv_tpm_cond2} imply $\sin \angle(\vv_t,\vu) \leq \tau$.
This completes the proof.
\end{proof}

\section{Conclusion}
\label{sec:conclusion}

We studied two iterative algorithms and their performance guarantees for a multichannel blind deconvolution that imposes a bilinear model on channel impulse responses.  Such a bilinear model is obtained, for example, by embedding a parametric model for the shapes of the impulse responses into a low-dimensional subspace through manifold embedding, while the channel gains are treated as independent variables.  Unlike recent theoretical results on blind deconvolution in the literature, we do not impose a strong geometric constraint on the input source signal.  Under the bilinear model, we modified classical cross-convolution method based on the commutativity of the convolution to overcome its critical weakness of sensitivity to noise.  The bilinear system model imposes a strong prior on the unknown channel impulse responses, which enables us to recover the system with short observation.  The constraint by the bilinear model, on the other hand, makes the recovery no longer a simple eigenvalue decomposition problem.  Therefore, standard algorithms in numerical linear algebra do not apply to this non-convex optimization problem.  We propose two iterative algorithms along with a simple spectral initialization.  When the basis in the bilinear model is generic, we have shown that the proposed algorithms converge linearly to a stable estimate of the unknown channel parameters with provable non-asymptotic performance guarantees.

Mathematically, our analysis involves tail estimates of norms of several structured random matrices, which are written as suprema of coupled high-order subgaussian processes.  In an earlier version of our approach \cite{lee2016fast_arxiv}, we used the concentration of a polynomial in subgaussian random vector \cite{adamczak2015concentration} together with the union bound through the $\epsilon$-net argument.  In this revised analysis, we factorized high-order random processes using gaussian processes of the first or second order and computed the supremum using sharp tail estimates in the literature (e.g., \cite{krahmer2014suprema}).  This change has already provided a significant improvement in scaling laws of key parameters in the main results but the sharpened scaling law is still suboptimal compared to the degrees of freedom in the underlying model.  It seems that due to the nature of our model, where the degrees of randomness remains the same (except noise) as we increase the length of observation, the lower bound on scaling laws given by the heuristic that counts the number of unknowns and equations may not be achieved.  However, we expect that it would be possible to further sharpen the estimates for structured random matrices.  It remains as an interesting open question how to extend the sharp estimates on suprema of second-order chaos processes \cite{krahmer2014suprema} to higher orders similarly to the extension of the Hanson-Wright inequality for concentration of subgaussian quadratic forms to higher-order polynomials.

\section*{Acknowledgement}

K.L. thanks S. Bahmani, Y. Eldar, F. Krahmer, and T. Strohmer for discussions and feedback on an earlier version \cite{lee2016fast_arxiv}.  The authors appreciate anonymous referees for their constructive comments, which helped improve both the technical results and presentation.

\appendices
\section{Toolbox}

In this section, we provide a collection of lemmas, which serve as mathematical tools to derive estimates of structured random matrices.

\begin{lemma}[{Complexification of Hanson-Wright Inequality \cite[Theorem 1.1]{rudelson2013hanson}}]
\label{lemma:hansonwright_ori}
Let $\mA \in \mathbb{C}^{m \times n}$.
Let $\vg \in \mathbb{C}^n$ be a standard complex Gaussian vector.
For any $0 < \zeta < 1$, there exists an absolute constant $C$ such that
\begin{align*}
|\norm{\mA \vg}_2^2 - \mathbb{E}_{\vg}[\norm{\mA \vg}_2^2]|
\leq C (
\norm{\mA^* \mA}_{\mathrm{F}} \sqrt{\log(2\zeta^{-1})}
\vee
\norm{\mA}^2 \log(2\zeta^{-1})
)
\end{align*}
holds with probability $1-\zeta$.
\end{lemma}

\begin{lemma}[{Complexification of Hanson-Wright Inequality \cite[Theorem 2.1]{rudelson2013hanson}}]
\label{lemma:hansonwright}
Let $\mA \in \mathbb{C}^{m \times n}$.
Let $\vg \in \mathbb{C}^n$ be a standard complex Gaussian vector.
For any $0 < \zeta < 1$, there exists an absolute constant $C$ such that
\begin{align*}
|\norm{\mA \vg}_2 - \norm{\mA}_{\mathrm{F}}| \leq C \norm{\mA} \sqrt{\log(2\zeta^{-1})}
\end{align*}
holds with probability $1-\zeta$.
\end{lemma}

The following lemma is a direct consequence of Maurey's empirical method \cite{carl1985inequalities}.

\begin{lemma}[{Maurey's empirical method \cite[Lemma~3.1]{junge2017ripI}}]
\label{lemma:maurey}
Let $k,m,n\in \mathbb{N}$ and $T:\ell_1^k(\mathbb{R}) \to \ell_\infty^m(\ell_2^d(\mathbb{R}))$ be a linear operator.
Then
\[
\int_0^\infty \sqrt{\log N(T(B_{\ell_1^k(\mathbb{R})}), \norm{\cdot}_{\ell_\infty^m(\ell_2^d(\mathbb{R}))}, t)} dt
\leq C \sqrt{(1+\log k)(1+\log m)}(1+\log m+\log d) \norm{T}.
\]
\end{lemma}

Lemma~\ref{lemma:maurey} extends to the complex field case, which is shown in the following corollary.

\begin{cor}
\label{cor:maurey}
Let $k,m,n\in \mathbb{N}$ and $T:\ell_1^k(\mathbb{C}) \to \ell_\infty^m(\ell_2^d(\mathbb{C}))$ be a linear operator.
Then
\[
\int_0^\infty \sqrt{\log N(T(B_{\ell_1^k(\mathbb{C})}), \norm{\cdot}_{\ell_\infty^m(\ell_2^d(\mathbb{C}))}, t)} dt
\leq C \sqrt{(1+\log k)(1+\log m)}(1+\log m+\log d) \norm{T}.
\]
\end{cor}

The following lemma provides tail estimates of suprema of subgaussian processes.
\begin{lemma}
\label{lemma:firstorder}
Let $\vxi \in \cz^n$ be a standard Gaussian vector with $\mathbb{E} \vxi \vxi^* = \mId_n$, $\Delta \subset \cz^n$, and $0 < \zeta < e^{-1/2}$.
There is an absolute constants $C$ such that
\[
\sup_{\vf \in \Delta} |\vf^* \vxi|
\leq C \sqrt{\log(\zeta^{-1})} \int_0^\infty \sqrt{\log N(\Delta,\norm{\cdot}_2,t)} dt
\]
holds with probability $1-\zeta$.
\end{lemma}

\begin{theorem}[{\hspace{1sp}\cite[Theorem~3.1]{krahmer2014suprema}}]
\label{thm:kmr}
Let $\vxi \in \cz^n$ be an $L$-subgaussian vector with $\mathbb{E} \vxi \vxi^* = \mId_n$, $\Delta \subset \cz^{m \times n}$, and $0 < \zeta < 1$.
There exists a constant $C(L)$ that only depends on $L$ such that
\[
\sup_{\mM \in \Delta} |\norm{\mM \vxi}_2^2 - \mathbb{E}[\norm{\mM \vxi}_2^2]|
\leq C(L) ( K_1 + K_2 \sqrt{\log(2\zeta^{-1})} + K_3 \log(2\zeta^{-1}) )
\]
holds with probability $1-\zeta$, where $K_1$, $K_2$, and $K_3$ are given by
\begin{align*}
K_1 {} & := \gamma_2(\Delta, \norm{\cdot})
\Big( \gamma_2(\Delta, \norm{\cdot}) + d_{\mathrm{F}}(\Delta) \Big)
+ d_{\mathrm{F}}(\Delta) d_{\mathrm{S}}(\Delta), \\
K_2 {} & := d_{\mathrm{S}}(\Delta) ( \gamma_2(\Delta, \norm{\cdot}) + d_{\mathrm{F}}(\Delta) ), \\
K_3 {} & := d_{\mathrm{S}}^2(\Delta).
\end{align*}
\end{theorem}

Using the polarization identity, this result on the suprema of second order chaos processes has been extended from a subgaussian quadratic form to a subgaussian bilinear form \cite{lee2015rip}.

\begin{theorem}[{A corollary of \cite[Theorem~2.3]{lee2015rip}}]
\label{thm:ip}
Let $\vxi \in \cz^n$ be an $L$-subgaussian vector with $\mathbb{E} \vxi \vxi^* = \mId_n$, $\Delta_2,\Delta_1 \subset \cz^{m \times n}$, $0 < \zeta < 1$, and $a > 0$.
There exists a constants $C(L)$ that only depends on $L$ such that
\begin{align*}
\sup_{\mM_1 \in \Delta_1, \mM_2 \in \Delta_2}
\big| \langle \mM_1 \vxi, \mM_2 \vxi \rangle
- \mathbb{E}[\langle \mM_1 \vxi, \mM_2 \vxi \rangle] \big|
\leq
C(L)( \widetilde{K}_1 + \widetilde{K}_2 \sqrt{\log(8\zeta^{-1})} + \widetilde{K}_3 \log(8\zeta^{-1}) ),
\end{align*}
holds with probability $1-\zeta$, where $\widetilde{K}_1$, $\widetilde{K}_2$, and $\widetilde{K}_3$ are given by
\begin{align*}
\widetilde{K}_1 {} & := \Big(a\gamma_2(\Delta_1, \norm{\cdot}) + a^{-1}\gamma_2(\Delta_2, \norm{\cdot})\Big) \Big(a\gamma_2(\Delta_1, \norm{\cdot}) + a^{-1}\gamma_2(\Delta_2, \norm{\cdot}) + a d_{\mathrm{F}}(\Delta_1) + a^{-1} d_{\mathrm{F}}(\Delta_2)\Big) \\
&+ \Big(a d_{\mathrm{F}}(\Delta_1) + a^{-1} d_{\mathrm{F}}(\Delta_2)\Big)
\Big(a d_{\mathrm{S}}(\Delta_1) + a^{-1} d_{\mathrm{S}}(\Delta_2)\Big), \\
\widetilde{K}_2 {} & := \Big(a d_{\mathrm{S}}(\Delta_1) + a^{-1} d_{\mathrm{S}}(\Delta_2)\Big) \Big( a \gamma_2(\Delta_1, \norm{\cdot}) + a^{-1} \gamma_2(\Delta_2, \norm{\cdot}) + a d_{\mathrm{F}}(\Delta_1) + a^{-1} d_{\mathrm{F}}(\Delta_2)\Big), \\
\widetilde{K}_3 {} & := \Big(a d_{\mathrm{S}}(\Delta_1) + a^{-1} d_{\mathrm{S}}(\Delta_2)\Big)^2.
\end{align*}
\end{theorem}

A special case of Theorem~\ref{thm:ip} where $a = 1$ was shown in \cite[Theorem~2.3]{lee2015rip}.
Note that the bilinear form satisfies
\[
\langle \mM_1 \vxi, \mM_2 \vxi \rangle
= \langle a \mM_1 \vxi, a^{-1} \mM_2 \vxi \rangle,
\quad \forall a > 0.
\]
Moreover, the $\gamma_2$ functional and the radii with respect to the Frobenius and spectral norms are all 1-homogeneous functions.
Therefore, Theorem~\ref{thm:ip} is a direct consequence of \cite[Theorem~2.3]{lee2015rip}.

Since $a > 0$ in Theorem~\ref{thm:ip} is arbitrary, one can minimize the tail estimate over $a > 0$.

\section{Expectations}

The following lemmas on the expectation of structured random matrices are derived in \cite{lee2017spectral}.  For the convenience of the readers, we include the lemmas.  Here the matrix $\widetilde{\mPhi}_m \in \mathbb{C}^{L \times D}$ denotes the zero-padded matrix of $\mPhi_m$ given by $\widetilde{\mPhi}_m \mS^\transpose \mPhi_m$ for $m=1,\dots,M$, where $\mS \in \mathbb{R}^{K \times L}$ is defined in \eqref{eq:def_FIR}.

\begin{lemma}[{\hspace{1sp}\cite[Lemma~B.1]{lee2017spectral}}]
\label{lemma:expectation1}
Under the assumption in (A1),
\[
\mathbb{E}[\mC_{\widetilde{\mPhi}_m \vu_m}^* \mC_{\widetilde{\mPhi}_m \vu_m}]
= K \norm{\vu_m}_2^2 \mId_L.
\]
\end{lemma}

\begin{lemma}[{\hspace{1sp}\cite[Lemma~B.2]{lee2017spectral}}]
\label{lemma:expectation2}
Under the assumption in (A1),
\[
\mathbb{E}[\mC_{\widetilde{\mPhi}_m \vu_m}^* \widetilde{\mPhi}_m]
= K \ve_1 \vu_m^*.
\]
\end{lemma}

\begin{lemma}[{\hspace{1sp}\cite[Lemma~B.3]{lee2017spectral}}]
\label{lemma:expectation3}
Under the assumption in (A1),
\[
\mathbb{E}[\widetilde{\mPhi}_m^* \mC_{\widetilde{\mPhi}_{m'} \vu_{m'}}^* \mC_{\vx}^* \mC_{\vx} \mC_{\widetilde{\mPhi}_{m'} \vu_{m'}} \widetilde{\mPhi}_m]
= \begin{cases}
K^2 \norm{\vx}_2^2 \norm{\vu_{m'}}_2^2 \mId_D & m\neq m', \\
K^2 \norm{\vx}_2^2 (\norm{\vu_{m'}}_2^2 \mId_D + \vu_{m'} \vu_{m'}^*) & m = m'.
\end{cases}
\]
\end{lemma}

\section{Proof of Lemma~\ref{lemma:matZ}}
\label{sec:proof:lemma:matZ}

Let $\vx' \in \mathbb{C}^L$ and $\vb' \in \mathbb{C}^D$.  By the definition of an adjoint operator, we have
\begin{align*}
\langle \vx' \otimes \vb' \otimes \bm{1}_{M,1}, \mathcal{A}^*(\underline{\vy}) \rangle
= \langle \mathcal{A}(\vx' \otimes \vb' \otimes \bm{1}_{M,1}), \underline{\vy} \rangle.
\end{align*}
Then by the definition of $\mathcal{A}$, we continue as
\begin{align*}
\langle \mathcal{A}(\vx' \otimes \vb' \otimes \bm{1}_{M,1}), \underline{\vy} \rangle
&= \sum_{m=1}^M \langle \mC_{\mS^*\mPhi_m \vb'} \vx', \vy_m \rangle
= \sum_{m=1}^M \vx'^* \mC_{\mS^*\mPhi_m \vb'}^* \vy_m \\
&= \sum_{m=1}^M \vx'^* (\opconj \mS^* \overline{\mPhi_m} \overline{\vb'} \circledast \vy_m)
= \sum_{m=1}^M \vx'^* \opconj(\mS^* \overline{\mPhi_m} \overline{\vb'} \circledast \opconj \vy_m) \\
&= \sum_{m=1}^M \vx'^* \opconj \mC_{\vy_m}^\transpose \mS^* \overline{\mPhi_m} \overline{\vb'}.
\end{align*}
Here we used the fact that the transpose of $\mC_{\vh}$ satisfies $\mC_{\vh}^\transpose = \mC_{\opconj \vh}$.

Finally, by tensorizing the last term, we obtain
\begin{align*}
\sum_{m=1}^M \vx'^* \opconj \mC_{\vy_m}^\transpose \mS^* \overline{\mPhi_m} \overline{\vb'}
&= \sum_{m=1}^M \vx'^* ((\vb')^* \otimes \mId_L) \mathrm{vec}(\opconj \mC_{\vy_m}^\transpose \mS^* \overline{\mPhi_m}) \\
&= \sum_{m=1}^M (\vb' \otimes \vx')^* \mathrm{vec}(\opconj \mC_{\vy_m}^\transpose \mS^* \overline{\mPhi_m})
= \sum_{m=1}^M (\vx' \otimes \vb')^* \mathrm{vec}(\mPhi_m^* \mS \mC_{\vy_m} \opconj).
\end{align*}
The assertion follows since $\vx'$ and $\vb'$ were arbitrary.

\section{Proof of Lemma~\ref{lemma:init_Gamma_Es}}
\label{sec:proof:lemma:init_Gamma_Es}

The left-hand side of \eqref{eq:obj_init_Zs} is rewritten as a variational form given by
\begin{equation}
\label{eq:varform1}
\sup_{\vz \in B_2^{2K-1}} \sup_{\vq \in B_2^D}
\sum_{m=1}^M a_m \vq^* \mPhi_m^* \mS \mC_{\mS^* \mPhi_m \vb} \breve{\mS}^* \vz
- \mathbb{E}_{\phi}[a_m \vq^* \mPhi_m^* \mS \mC_{\mS^* \mPhi_m \vb} \breve{\mS}^* \vz].
\end{equation}

We rewrite $\sum_{m=1}^M a_m \vq^* \mPhi_m^* \mS \mC_{\mS^* \mPhi_m \vb} \breve{\mS}^* \vz$ as
\begin{align*}
 \sum_{m=1}^M a_m \vq^* \mPhi_m^* \mS \mC_{\breve{\mS}^* \vz} \mS^* \mPhi_m \vb
&= \sum_{m=1}^M a_m \mathrm{vec}(\mPhi_m)^* (\overline{\vq} \otimes \mId_K) \mS \mC_{\breve{\mS}^* \vz} \mS^* (\vb^\transpose \otimes \mId_K) \mathrm{vec}(\mPhi_m) \\
&= \sum_{m=1}^M a_m \mathrm{vec}(\mPhi_m)^* (\overline{\vq} \vb^\transpose \otimes \mS \mC_{\breve{\mS}^* \vz} \mS^*) \mathrm{vec}(\mPhi_m).
\end{align*}

Let $\vphi = [\mathrm{vec}(\mPhi_1)^\transpose,\dots,\mathrm{vec}(\mPhi_M)^\transpose]^\transpose$.
Then
\begin{align*}
\sum_{m=1}^M a_m \mathrm{vec}(\mPhi_m)^* (\overline{\vq} \vb^\transpose \otimes \mS \mC_{\breve{\mS}^* \vz} \mS^*) \mathrm{vec}(\mPhi_m)
= \vphi^*
\Big( \sum_{m=1}^M a_m \ve_m \ve_m^* \otimes \overline{\vq} \vb^\transpose \otimes \mS \mC_{\breve{\mS}^* \vz} \mS^* \Big)
\vphi.
\end{align*}
Therefore, the objective function in the supremum in \eqref{eq:varform1} becomes a second-order chaos process.
We compute the tail estimate of the supremum by applying Theorem~\ref{thm:ip} with
\[
\Delta_1 = \Big\{ \sum_{m=1}^M a_m \ve_m \ve_m^* \otimes \vq^\transpose \otimes \mId_K ~|~ \vq \in B_2^D \Big\}
\]
and
\[
\Delta_2 = \Big\{ \sum_{m=1}^M \ve_m \ve_m^* \otimes \vb^\transpose \otimes \mS \mC_{\breve{\mS}^* \vz} \mS^* ~|~ \vq \in B_2^{2K-1} \Big\}.
\]

By direct calculation, the radii of $\Delta_1$ and $\Delta_2$ are given as follows:
\begin{align*}
d_{\mathrm{S}}(\Delta_1)
&\leq \norm{\va}_\infty, \\
d_{\mathrm{F}}(\Delta_1)
&\leq \norm{\va}_2 \sqrt{K}, \\
d_{\mathrm{S}}(\Delta_2)
&\leq \norm{\vb}_2 \sqrt{K}, \\
d_{\mathrm{F}}(\Delta_2)
&\leq \norm{\vb}_2 \sqrt{MK}.
\end{align*}
Here, we used the fact that
\[
\norm{\mS \mC_{\breve{\mS}^* \vz} \mS^*}
\leq \norm{\mS \mC_{\breve{\mS}^* \vz} \mS^*}_{\mathrm{F}}
\leq \sqrt{K} \norm{\vz}_2 \leq \sqrt{K}.
\]

Moreover, since
\[
d_{\mathrm{S}}(\Delta_1) \leq \norm{\va}_\infty \norm{\vq}_2,
\]
by Dudley's inequality and a standard volume argument, we have
\[
\gamma_2(\Delta_1) \leq C_1 \norm{\va}_\infty \int_0^\infty \sqrt{\log N(B_2^D,\norm{\cdot}_2,t)} dt
\leq C_2 \norm{\va}_\infty \sqrt{D}.
\]
On the other hand, since
\[
\mC_{\breve{\mS}^* \vz} = \sqrt{L} \mF^* \mathrm{diag}(\mF \breve{\mS}^* \vz) \mF,
\]
we have
\[
d_{\mathrm{S}}(\Delta_2) \leq \norm{\vb}_2 \sqrt{L} \norm{\mF \breve{\mS}^* \vz}_\infty,
\]
which implies
\begin{align*}
\gamma_2(\Delta_2)
&\leq C_1 \norm{\vb}_2 \sqrt{L} \int_0^\infty \sqrt{\log N(\mF \breve{\mS}^* B_2^{2K-1},\norm{\cdot}_\infty,t)} dt \\
&\leq C_3 \norm{\vb}_2 \sqrt{LK} \int_0^\infty \sqrt{\log N(\mF \breve{\mS}^* B_1^{2K-1},\norm{\cdot}_\infty,t)} dt \\
&\leq C_4 \norm{\vb}_2 \sqrt{K} \sqrt{\log(2K-1)} \log^{3/2}L \\
&\leq C_5 \norm{\vb}_2 \sqrt{K} \sqrt{\log K} \log^{3/2}L,
\end{align*}
where the third step follows from Corollary~\ref{cor:maurey}.

By applying these estimates to Theorem~\ref{thm:ip} with
\begin{equation}
\label{eq:a4gm}
a = \sqrt{\frac{\gamma_2(\Delta_2,\norm{\cdot}) d_{\mathrm{F}}(\Delta_2)}{\gamma_2(\Delta_1,\norm{\cdot}) d_{\mathrm{F}}(\Delta_1)}},
\end{equation}
we obtain that the supremum is less than
\[
C'(\beta) \log^\alpha(MKL) (\sqrt{M} K^{3/4} D^{1/4} + \sqrt{M} K + \sqrt{MKD})
\]
holds with probability $1-K^{-\beta}$.
By the arithmetic-geometric mean inequality,
\[
\sqrt{M} K^{3/4} D^{1/4} \leq \frac{\sqrt{M} K + \sqrt{MKD}}{2}.
\]
We also have $\sqrt{M} K \geq \sqrt{MKD}$ since $K \geq D$.
This completes the proof.

\section{Proof of Lemma~\ref{lemma:init_Gamma_Ec}}
\label{sec:proof:lemma:init_Gamma_Ec}

The spectral norm in the left-hand side of \eqref{eq:init_Gamma_Ec} is expressed as a variational form given by
\begin{equation}
\label{eq:varform2}
\sup_{\vz \in B_2^{2K-1}} \sup_{\vq \in B_2^D}
\sum_{m=1}^M
\vz^* \breve{\mS} \mC_{\vx} \mC_{\vw_m}^* \mS^* \mPhi_m \vq.
\end{equation}
The objective function in \eqref{eq:varform2} is rewritten as
\begin{align*}
\sum_{m=1}^M \vz^* \breve{\mS} \mC_{\vx} \mC_{\vw_m}^* \mS^* \mPhi_m \vq
= \sum_{m=1}^M (\vq^\transpose \otimes \vz^* \breve{\mS} \mC_{\vx} \mC_{\vw_m}^* \mS^*) \mathrm{vec}(\mPhi_m)
= \Big\langle \vf(\vq,\vz), \sum_{m=1}^M \ve_m \otimes \mathrm{vec}(\mPhi_m) \Big\rangle,
\end{align*}
where
\[
\vf(\vq,\vz) = \sum_{m=1}^M \ve_m \otimes \overline{\vq} \otimes \mS \mC_{\vw_m} \mC_{\vx}^* \breve{\mS} \vz.
\]
Since $\sum_{m=1}^M \ve_m \otimes \mathrm{vec}(\mPhi_m) = [\mathrm{vec}(\mPhi_1)^\transpose,\dots,\mathrm{vec}(\mPhi_M)^\transpose]^\transpose$ is a standard complex Gaussian vector of length $MKD$, we compute a tail estimate of the supremum in \eqref{eq:varform2} by applying Lemma~\ref{lemma:firstorder} with
\[
\Delta = \Big\{ \vf(\vq,\vz) ~|~ \vq \in B_2^D,~ \vz \in B_2^{2K-1}
\Big\}.
\]

Since
\begin{align*}
\norm{\vf(\vq,\vz) - \vf(\vq',\vz')}_2
&\leq \norm{\vf(\vq,\vz) - \vf(\vq',\vz)}_2 + \norm{\vf(\vq',\vz) - \vf(\vq',\vz')}_2 \\
&\leq \sqrt{M} \norm{\mS \mC_{\vw_m} \mC_{\vx}^* \breve{\mS}} (\norm{\vq-\vq'}_2 + \norm{\vz-\vz'}_2) \\
&\leq \sqrt{M} \rho_{x,w} (\norm{\vq-\vq'}_2 + \norm{\vz-\vz'}_2),
\end{align*}
we have
\begin{align*}
&\int_0^\infty \sqrt{\log N(\Delta,\norm{\cdot}_2,t)} dt \\
&\leq \sqrt{M} \rho_{x,w}
(\int_0^\infty \sqrt{\log N(B_2^{2K-1},\norm{\cdot}_2,t)} dt
+ \int_0^\infty \sqrt{\log N(B_2^D,\norm{\cdot}_2,t)} dt) \\
&\leq C \rho_{x,w} \sqrt{MK},
\end{align*}
where the last step follows from a standard volume argument and the fact that $K \geq D$.
The assertion then follows by applying the above estimate to Lemma~\ref{lemma:firstorder}.

\section{Proof of Lemma~\ref{lemma:init_Gamma_En}}
\label{sec:proof:lemma:init_Gamma_En}

We use the following lemma from \cite{lee2017spectral} to prove Lemma~\ref{lemma:init_Gamma_En}.

\begin{lemma}[{\hspace{1sp}\cite[Lemma~5.3]{lee2017spectral}}]
\label{lemma:gaussquadasym}
Let $\mPsi \in \mathbb{C}^{K \times D}$ satisfy that $\mathrm{vec}(\mPsi)$ follows $\mathcal{CN}(0,\mId_{KD})$, $0 < \zeta < 1$, and $\mA \in \mathbb{C}^{K \times K}$.
Then
\begin{align*}
& \norm{
\mPsi^* \mA \mPsi - \mathbb{E}_{\vphi}[\mPsi^* \mA \mPsi]
}
\leq C \norm{\mA} \sqrt{KD} \log(8 \zeta^{-1})
\end{align*}
holds with probability $1-\zeta$.
\end{lemma}

Note that $\mGamma_{\mathrm{n}} \mGamma_{\mathrm{n}}^*$ is expressed as
\[
\mGamma_{\mathrm{n}} \mGamma_{\mathrm{n}}^*
= \sum_{m,m'=1}^M
\mPhi_m^* \mS \mC_{\vw_m} \mC_{\vw_{m'}}^* \mS^* \mPhi_{m'}.
\]
We apply Lemma~\ref{lemma:gaussquadasym} with
\[
\mPsi = [\mPhi_1^\transpose,\dots,\mPhi_M^\transpose]^\transpose
\]
and
\[
\mA = \sum_{m,m'=1}^M
\ve_m \ve_{m'}^* \mS (\mC_{\vw_m} \mC_{\vw_{m'}}^* - \mathbb{E}_w[\mC_{\vw_m} \mC_{\vw_{m'}}^*]) \mS^*.
\]

By the block Gershgorin disk theorem \cite{feingold1962block}, it follows that
\[
\norm{\mA}
\leq \max_{1\leq m\leq M} \sum_{m'=1}^M \norm{\mS (\mC_{\vw_m} \mC_{\vw_{m'}}^* - \mathbb{E}_w[\mC_{\vw_m} \mC_{\vw_{m'}}^*]) \mS^*}
\leq M \rho_w.
\]
Then the assertion follows by Lemma~\ref{lemma:gaussquadasym}.

\section{Proof of Lemma~\ref{lemma:Ec}}
\label{sec:proof:lemma:Ec}

We decompose $\mPhi^* \mY_{\mathrm{s}}^* \mY_{\mathrm{n}} \mPhi$ into two parts respectively corresponding to the diagonal block portion and the off-diagonal block portion of $\mY_{\mathrm{s}}^* \mY_{\mathrm{n}}$:
\[
\mPhi^* \mY_{\mathrm{s}}^* \mY_{\mathrm{n}} \mPhi = \text{(g)} + \text{(h)},
\]
where
\begin{equation}
\label{eqn_dbl_xw}
\begin{aligned}
\text{(g)} &=
\sum_{m=1}^M \ve_m \ve_m^* \otimes \Big( \sum_{\begin{subarray}{c} m'=1 \\ m' \neq m \end{subarray}}^M
\overline{a_{m'}} \widetilde{\mPhi}_m^* \mC_{\widetilde{\mPhi}_{m'} \vb}^* \mC_{\vx}^* \mC_{\vw_{m'}} \widetilde{\mPhi}_m \Big), \\
\text{(h)} &= -
\sum_{m=1}^M \sum_{\begin{subarray}{c} m'=1 \\ m' \neq m \end{subarray}}^M
\ve_m \ve_{m'}^* \otimes
\overline{a_{m'}} \widetilde{\mPhi}_m^* \mC_{\widetilde{\mPhi}_{m'} \vb}^* \mC_{\vx}^* \mC_{\vw_m} \widetilde{\mPhi}_{m'}.
\end{aligned}
\end{equation}

Since $\tnorm{\cdot}_{S_1 \to S_\infty}$ is a valid norm, by the triangle inequality, we have
\[
\tnorm{\mPhi^* \mY_{\mathrm{s}}^* \mY_{\mathrm{n}} \mPhi}_{S_1 \to S_\infty}
\leq \tnorm{\text{(g)}}_{S_1 \to S_\infty} + \tnorm{\text{(h)}}_{S_1 \to S_\infty}.
\]
Furthermore, by \eqref{eq:matSopnormlessthanSN}, we also have
\[
\tnorm{\text{(g)}}_{S_1 \to S_\infty} \leq \norm{\text{(g)}}.
\]
We use a tail estimate of $\norm{\text{(g)}}$ derived in the proof of \cite[Lemma~3.6]{lee2017spectral}.  It has been shown that
\begin{equation}
\label{eq:bnd_snEc_part1}
\norm{\text{(g)}}
\leq
C(\beta) \rho_{x,w} K \sqrt{D} \norm{\va}_2 \norm{\vb}_2 \log^{\alpha}(MKL)
\end{equation}
with probability $1-CK^{-\beta}$ (See \cite[Section~5.3]{lee2017spectral}).
We will show that the tail estimate of $\norm{\text{(g)}}$ is dominated by that for $\tnorm{\text{(h)}}_{S_1 \to S_\infty}$.

For the part corresponding to the off-diagonal portion of $\mY_{\mathrm{s}}^* \mY_{\mathrm{n}}$, we add and subtract the diagonal sum in $\text{(h)}$ and obtain
\[
\text{(h)} = \text{(k)} + \text{(l)}
\]
for
\begin{equation}
\label{eq:ubXic2}
\begin{aligned}
\text{(k)} &=
\sum_{m=1}^M \ve_m \ve_m^* \otimes
\mPhi_m^* \mS \mC_{\vx}^* \mC_{\vw_m} \breve{\mS}^* \mZ_m, \\
\text{(l)}
&= -\sum_{m,m'=1}^M \ve_m \ve_{m'}^* \otimes
\mPhi_m^* \mS \mC_{\vx}^* \mC_{\vw_m} \breve{\mS}^* \mZ_{m'},
\end{aligned}
\end{equation}
where
\[
\mZ_m := \overline{a_{m}} \breve{\mS} \mC_{\widetilde{\mPhi}_m \vb}^* \widetilde{\mPhi}_m, \quad m=1,\dots,M.
\]

Again, since $\tnorm{\text{(k)}}_{S_1 \to S_\infty} \leq \norm{\text{(k)}}$, we can use a tail estimate of $\norm{\text{(k)}}$ derived in the proof of \cite[Lemma~3.6]{lee2017spectral}.  It has been shown that
\begin{align*}
\norm{\text{(k)}}
\leq \rho_{x,w} C(\beta) K^{3/2} \norm{\va}_\infty \norm{\vb}_2 \log^\alpha(MKL)
\end{align*}
holds with probability $1-CK^{-\beta}$.
We will show that the tail estimate of $\norm{\text{(k)}}$ is dominated by that for $\tnorm{\text{(l)}}_{S_1 \to S_\infty}$, which we derive below.

Through a factorization of the full 2D summation in $\text{(l)}$, we obtain
\begin{align*}
\tnorm{\text{(l)}}_{S_1 \to S_\infty}
& \leq
\Bigtnorm{
\underbrace{
\sum_{m=1}^M
\ve_m^* \otimes \breve{\mS} \mC_{\vw_m}^* \mC_{\vx} \mS^* \mPhi_m
}_{\text{(o)}}
}_{S_1 \to S_2}
\Bigtnorm{
\underbrace{
\sum_{m'=1}^M
\ve_{m'}^* \otimes (\mZ_{m'} - \mathbb{E}[\mZ_{m'}])
}_{\text{(p)}}
}_{S_1 \to S_2} \\
& +
\Bigtnorm{
\underbrace{
\sum_{m,m'=1}^M \ve_m \ve_{m'}^* \otimes
\mPhi_m^* \mS \mC_{\vx}^* \mC_{\vw_m} \breve{\mS}^* \mathbb{E}[\mZ_{m'}]
}_{\text{(q)}}
}_{S_1 \to S_\infty}.
\end{align*}

Note that $\tnorm{\text{(o)}}_{S_1 \to S_2}$ is written as the supremum of a Gaussian process and is bounded by the following lemma.
\begin{lemma}
\label{lemma:comp_cross_rank1}
Suppose that (A1) holds.
For any $\beta \in \mathbb{N}$, there exists a constant $C(\beta)$ that depends only on $\beta$ such that, conditional on the noise vector $\vw$,
\[
\Bigtnorm{
\sum_{m=1}^M
\ve_m^* \otimes \breve{\mS} \mC_{\vw_m}^* \mC_{\vx} \mS^* \mPhi_m
}_{S_1 \to S_2}
\leq C \sqrt{1+\beta} \rho_{x,w} \sqrt{M+D+K} \log K
\]
holds with probability $1-K^{-\beta}$.
\end{lemma}

\begin{proof}[Proof of Lemma~\ref{lemma:comp_cross_rank1}]
Let $\vphi_m = \mathrm{vec}(\mPhi_m)$ for $m=1,\dots,M$ and $\vphi = [\vphi_1^\transpose, \dots, \vphi_M^\transpose]^\transpose$.
Let $\vq = [q_1,\dots,q_M]^\transpose \in \mathbb{C}^M$.
Then it follows from \eqref{eq:varforn_S1S8norm} that
\begin{align*}
\Bigtnorm{
\sum_{m=1}^M
\ve_m^* \otimes \breve{\mS} \mC_{\vw_m}^* \mC_{\vx} \mS^* \mPhi_m
}_{S_1 \to S_2}
& = \sup_{\vz \in B_2^{2K-1}} \sup_{\vxi \in B_2^D} \sup_{\vq \in B_2^M}
\Big|
\sum_{m=1}^M q_m \vz^* \breve{\mS} \mC_{\vw_m}^* \mC_{\vx} \mS^* \mPhi_m \vxi
\Big| \\
& = \sup_{\vz \in B_2^{2K-1}} \sup_{\vxi \in B_2^D} \sup_{\vq \in B_2^M}
\Big|
\sum_{m=1}^M (q_m \vxi^\transpose \otimes \vz^* \breve{\mS} \mC_{\vw_m}^* \mC_{\vx} \mS^*) \vphi_m
\Big| \\
& = \sup_{\vz \in B_2^{2K-1}} \sup_{\vxi \in B_2^D} \sup_{\vq \in B_2^M}
\Big|
\sum_{m=1}^M
q_m \ve_m^* \otimes (\vxi^\transpose \otimes \vz^* \breve{\mS} \mC_{\vw_m}^* \mC_{\vx} \mS^*) \vphi
\Big|.
\end{align*}

Let
\[
\vf(\vz,\vxi,\vq)
= \sum_{m=1}^M
\overline{q_m} \ve_m \otimes (\overline{\vxi} \otimes \mS \mC_{\vx}^* \mC_{\vw_m} \breve{\mS}^* \vz).
\]
Then we obtain
\[
\Big\|
\sum_{m=1}^M
\ve_m^* \otimes \breve{\mS} \mC_{\vw_m}^* \mC_{\vx} \mS^* \mPhi_m
\Big\|
= \sup_{\vz \in B_2^{2K-1}} \sup_{\vxi \in B_2^D} \sup_{\vq \in B_2^M} |\vf(\vz,\vxi,\vq)^* \vphi|.
\]
Note that $\vf(\vz,\vxi)^* \vphi$, conditioned on $\vw$, is a centered Gaussian process.  We compute a tail estimate of this supremum by applying Lemma~\ref{lemma:firstorder} with
\[
\Delta = \{ \vf(\vz,\vxi,\vq) ~|~ \vz \in B_2^{2K-1},~ \vxi \in B_2^D,~ \vq \in B_2^M \}.
\]
Then we need to compute the entropy integral for $\Delta$.
Recall
\[
\rho_{x,w} = \max_{1\leq m\leq M} \norm{\widetilde{\mS} \mC_{\vx}^* \mC_{\vw_1} \widetilde{\mS}^*}
\geq \norm{\breve{\mS} \mC_{\vw_m}^* \mC_{\vx} \mS^*}, \quad \forall m=1,\dots,M.
\]

By the triangle inequality, we obtain
\begin{align*}
& \norm{\vf(\vz,\vxi,\vq) - \vf(\vz',\vxi',\vq')}_2 \\
& \leq
\norm{\vf(\vz,\vxi,\vq) - \vf(\vz,\vxi,\vq')}_2
+ \norm{\vf(\vz,\vxi,\vq') - \vf(\vz,\vxi',\vq')}_2
+ \norm{\vf(\vz,\vxi',\vq') - \vf(\vz',\vxi',\vq')}_2 \\
& \leq \rho_{x,w}
(
\norm{\vz}_2 \norm{\vxi}_2 \norm{\vq-\vq'}_2
+ \norm{\vz}_2 \norm{\vxi-\vxi'}_2 \norm{\vq'}_2
+ \norm{\vz-\vz'}_2 \norm{\vxi'}_2 \norm{\vq'}_2 \\
& \leq \rho_{x,w}
(\norm{\vq-\vq'}_2 + \norm{\vxi-\vxi'}_2 + \norm{\vz-\vz'}_2).
\end{align*}

The integral of the log-entropy number is computed as
\begin{align*}
& \sup_{\vz \in B_2^{2K-1}} \sup_{\vxi \in B_2^D} \sup_{\vq \in B_2^M} |\vf(\vz,\vxi,\vq)^* \vphi| \\
& \leq C_1 \int_0^\infty \sqrt{\log N(\Delta,\norm{\cdot}_2,t)} dt \\
& \leq C_1 \rho_{x,w}
\Big( \int_0^\infty \sqrt{\log N(B_2^M,\norm{\cdot}_2,t)} dt + \int_0^\infty \sqrt{\log N(B_2^M,\norm{\cdot}_2,t)} dt \\
& \qquad + \int_0^\infty \sqrt{\log N(B_2^{2K-1},\norm{\cdot}_2,t)} dt \Big) \\
& \leq C_2 \rho_{x,w} \sqrt{M+D+K},
\end{align*}
where the last step follows from a standard volume argument.
Then the assertion follows from Lemma~\ref{lemma:firstorder}.
\end{proof}

Next $\tnorm{\text{(p)}}_{S_1 \to S_2}$ is written as the supremum of a second-order Gaussian chaos process and its tail estimate can be derived by Theorem~\ref{thm:ip}.  However, the rank-1 constraint on the domain does not provide any gain in reducing the tail estimate in this case.  Therefore, we use a previous estimate on $\norm{\text{(p)}}_{S_1 \to S_2}$ derived in \cite[Lemma~5.6]{lee2017spectral}, which is stated in the following lemma.

\begin{lemma}
\label{lemma:comp_sum}
Suppose that (A1) holds.
For any $\beta \in \mathbb{N}$, there exist a numerical constant $\alpha \in \mathbb{N}$ and a constant $C(\beta)$ that depends only on $\beta$ such that
\[
\Bignorm{
\sum_{m=1}^M \ve_m^* \otimes (\mZ_m - \mathbb{E}[\mZ_m])
}
\leq C(\beta) \norm{\va}_\infty \norm{\vb}_2 (K + \sqrt{MKD}) \log^\alpha (MKL)
\]
holds with probability $1-K^{-\beta}$.
\end{lemma}

Similarly to $\tnorm{\text{(o)}}_{S_1 \to S_2}$, one can rewrite $\tnorm{\text{(q)}}_{S_1 \to S_\infty}$ as the supremum of a Gaussian process. The following lemma provides its tail estimate.
\begin{lemma}
\label{lemma:comp_cross_exp_rank1}
Suppose that (A1) holds.
For any $\beta \in \mathbb{N}$, there exists a constant $C(\beta)$ that depends only on $\beta$ such that, conditional on the noise vector $\vw$,
\[
\Bigtnorm{
\sum_{m,m'=1}^M \ve_m \ve_{m'}^* \otimes
\mPhi_m^* \mS \mC_{\vx}^* \mC_{\vw_m} \breve{\mS}^* \mathbb{E}[\mZ_{m'}]
}_{S_1 \to S_\infty}
\leq C \sqrt{1+\beta} \rho_{x,w} \norm{\va}_2 \norm{\vb}_2 K \sqrt{M+D} \log K
\]
holds with probability $1-K^{-\beta}$.
\end{lemma}

\begin{proof}[Proof of Lemma~\ref{lemma:comp_cross_exp_rank1}]
It follows from the variational form in \eqref{eq:varforn_S1S8norm} and Lemma~\ref{lemma:expectation2} that
\begin{align}
& \Bigtnorm{
\sum_{m,m'=1}^M \ve_m \ve_{m'}^* \otimes
\mPhi_m^* \mS \mC_{\vx}^* \mC_{\vw_m} \breve{\mS}^* \mathbb{E}[\mZ_{m'}]
}_{S_1 \to S_\infty} \nonumber \\
&= K \sup_{\vq,\widetilde{\vq} \in B_2^M} \sup_{\vxi,\widetilde{\vxi} \in B_2^D}
\Big| \sum_{m,m'=1}^M q_m q_{m'} a_{m'} \widetilde{\vxi}^* \mPhi_m^* \mS \mC_{\vx}^* \vw_m \vb^* \vxi \Big| \nonumber \\
&= K \sup_{\widetilde{\vq} \in B_2^M} \Big| \sum_{m'=1}^M q_{m'} a_{m'} \Big|
\sup_{\vxi \in B_2^D} |\vb^* \vxi|
\sup_{\vq \in B_2^M} \sup_{\widetilde{\vxi} \in B_2^D}
\Big|\sum_{m=1}^M q_m \widetilde{\vxi}^* \mPhi_m^* \mS \mC_{\vx}^* \vw_m \Big| \nonumber \\
&\leq K \norm{\va}_2 \norm{\vb}_2
\sup_{\vq \in B_2^M} \sup_{\widetilde{\vxi} \in B_2^D}
\Big|\sum_{m=1}^M q_m \widetilde{\vxi}^* \mPhi_m^* \mS \mC_{\vx}^* \vw_m \Big|.
\label{eq:comp_cross_exp_bnd1}
\end{align}

Note that the objective in the supremum in \eqref{eq:comp_cross_exp_bnd1} is rewritten as
\begin{align*}
\Big|\sum_{m=1}^M q_m \widetilde{\vxi}^* \mPhi_m^* \mS \mC_{\vx}^* \vw_m \Big|
= \Big|\sum_{m=1}^M \overline{q_m} \vw_m^* \mC_{\vx} \mS^* \mPhi_m \widetilde{\vxi} \Big|
= \Big|\Big(\sum_{m=1}^M \ve_m^* \otimes \vw_m^* \mC_{\vx} \mS^* \mPhi_m\Big) (\overline{\vq} \otimes \widetilde{\vxi}) \Big|.
\end{align*}
Then it follows that
\begin{equation}
\label{eq:comp_cross_exp_bnd2}
\Bigtnorm{
\sum_{m=1}^M
\ve_m^* \otimes \vw_m^* \mC_{\vx} \mS^* \mPhi_m
}_{S_1 \to S_2}
\leq C(\beta) \rho_{x,w} \sqrt{M+D} \log K
\end{equation}
holds with probability $1-K^{-\beta}$.  The proof of \eqref{eq:comp_cross_exp_bnd2} is obtained as we replace $\vz \in B_2^{2K-1}$ in the proof of Lemma~\ref{lemma:comp_cross_rank1} by $[1, \vzero_{1,2K}]^\transpose$.

The proof completes by plugging in the tail bound in \eqref{eq:comp_cross_exp_bnd2} into \eqref{eq:comp_cross_exp_bnd1}.
\end{proof}

By collecting the estimates, we obtain that
\[
\norm{\text{(l)}}_{S_1 \to S_\infty}
\leq C(\beta) \rho_{x,w} \norm{\va}_2 \norm{\vb}_2 \log^\alpha (MKL)
\Big(
\frac{\mu \sqrt{M+D+K} (K+\sqrt{MKD})}{\sqrt{M}} + K\sqrt{M+D} \,
\Big)
\]
holds with probability $1-CK^{-\beta}$.  Then the tail estimate of $\norm{\text{(l)}}_{S_1 \to S_\infty}$ dominates those for $\norm{\text{(g)}}$ and $\norm{\text{(k)}}$.  Therefore, we may ignore $\tnorm{\text{(g)}}_{S_1 \to S_\infty}$ and $\tnorm{\text{(k)}}_{S_1 \to S_\infty}$.

Therefore, by plugging in \eqref{eq:etasimp}, we obtain that with probability $1-CK^{-\beta}$, the relative perturbation due to $\mPhi^* \mY_{\mathrm{s}}^* \mY_{\mathrm{n}} \mPhi$ is upper bounded by
\begin{align*}
\frac{\tnorm{\mPhi^* \mY_{\mathrm{s}}^* \mY_{\mathrm{n}} \mPhi}_{S_1 \to S_\infty}}{K^2 \norm{\vx}_2^2 \norm{\va}_2^2 \norm{\vb}_2^2}
& \leq
\frac{C(\beta) \rho_{x,w} \log^\alpha (MKL)}{K^2 \norm{\vx}_2^2 \norm{\va}_2 \norm{\vb}_2}
\cdot
\Big(
\frac{\mu \sqrt{M+D+K} (K+\sqrt{MKD})}{\sqrt{M}} + K\sqrt{M+D} \,
\Big) \\
& \leq
\frac{C'(\beta) \log^{\alpha}(MKL)}{\sqrt{\eta L}}
\cdot
\frac{\rho_{x,w}}{\sqrt{K} \sigma_w \norm{\vx}_2}
\cdot
\Big(
\mu\Big(\frac{\sqrt{K}}{M} + \sqrt{\frac{D}{M}} + \sqrt{\frac{D}{K}} \, \Big) + 1
\Big).
\end{align*}
This completes the proof.



\end{document}